\documentclass[12pt]{amsart}

\usepackage{hyperref}
\usepackage{amsmath, amsthm, amsfonts, amssymb}
\usepackage{mathtools}
\usepackage{booktabs}
\usepackage{upgreek}
\usepackage[all]{xy}

\usepackage[dvipsnames]{xcolor}
\usepackage{caption}
\usepackage{listings}
\allowdisplaybreaks[1]

\theoremstyle{plain}
\newtheorem{thm}{Theorem}[section]
\newtheorem*{thm*}{Theorem}
\newtheorem{coro}[thm]{Corollary}
\newtheorem{prop}[thm]{Proposition}

\newtheorem{lemm}[thm]{Lemma}

\theoremstyle{definition}
\newtheorem{deff}[thm]{Definition}
\newtheorem{examp}[thm]{Example}
\newtheorem*{oq}{Open question}

\theoremstyle{remark}
\newtheorem{rema}[thm]{Remark}
\newtheorem{conv}{Convention}

\newcommand\twoscript[2]{\substack{{#1} \\ {#2}}}
\newcommand\threescript[3]{\substack{{#1} \\ {#2} \\ {#3}}}

\newcommand\rmi{\mathrm{i}}
\newcommand\rme{\mathrm{e}}
\newcommand\dimn{\mathfrak{n}}
\newcommand\dimm{\mathfrak{m}}

\newcommand\legendre[2]{\genfrac{(}{)}{}{}{#1}{#2}}
\newcommand\tbtmat[4]{\left(\begin{smallmatrix}{#1} & {#2} \\ {#3} & {#4}\end{smallmatrix}\right)}
\newcommand*\abs[1]{\lvert#1\rvert}
\newcommand\etp[1]{\mathfrak{e}\left(#1\right)}
\newcommand\vol[1]{\mathop{\mathrm{vol}}\left(#1\right)}
\newcommand\diff{\,\mathrm{d}}
\newcommand\dodt{\mathop{\frac{\mathrm{d}}{\mathrm{d}\tau}}}
\newcommand\dodth[1]{\mathop{\frac{\mathrm{d}^{#1}}{\mathrm{d}\tau^{#1}}}}
\newcommand\dodtha[2]{\frac{\mathop{\mathrm{d}^{#1}}{#2}}{\mathop{\mathrm{d}}\tau^{#1}}}
\newcommand\emb{\mathop{emb}_{\mathfrak{B}}}
\newcommand\sgn[1]{\mathop{\mathrm{sgn}}\left(#1\right)}

\newcommand\numZ{\mathbb{Z}}
\newcommand\numQ{\mathbb{Q}}
\newcommand\numR{\mathbb{R}}
\newcommand\numC{\mathbb{C}}
\newcommand\projQ{\mathbb{P}^1(\mathbb{Q})}
\newcommand\mulindZn{\mathbb{Z}^{\mathfrak{n}}}
\newcommand\mulindRn{\mathbb{R}^{\mathfrak{n}}}
\newcommand\mulindCn{\mathbb{C}^{\mathfrak{n}}}
\newcommand\halfint{\frac{1}{2}\mathbb{Z}}
\newcommand\numgeq[2]{\mathbb{#1}_{\geq #2}}

\newcommand\slZ{\mathrm{SL}_2(\mathbb{Z})}

\newcommand\slR{\mathrm{SL}_2(\mathbb{R})}

\newcommand\pslZ{\mathrm{PSL}_2(\mathbb{Z})}

\newcommand\glpR{\mathrm{GL}_2^{+}(\mathbb{R})}
\newcommand\glptR{\widetilde{\mathrm{GL}_2^{+}(\mathbb{R})}}
\newcommand\sltR{\widetilde{\mathrm{SL}_2(\mathbb{R})}}
\newcommand\sltZ{\widetilde{\mathrm{SL}_2(\mathbb{Z})}}
\newcommand\GlW{\mathrm{GL}(\mathcal{W})}

\newcommand\mymatset[3]{{#1}^{{#2}\times{#3}}}

\newcommand\heigrp[1]{H(\underline{#1})}
\newcommand\heigrpa[2]{H(\underline{#1},{#2})}
\newcommand\actgrp[1]{\glptR \ltimes \heigrp{#1}}
\newcommand\Jacgrp[2]{\widetilde{#1} \ltimes \heigrpa{#2}{1}}

\newcommand\uhp{\mathfrak{H}}
\newcommand\mydom{\uhp \times \mathcal{V}}
\newcommand\myran{\mathcal{W}}
\newcommand\contfns{C(\mydom, \myran)}
\newcommand\smofns{C^\infty(\mydom, \myran)}
\newcommand\holfns{\mathcal{O}(\mydom, \myran)}

\newcommand\holfnsm{\mathcal{O}(\uhp, \myran)}
\newcommand\holfnsmC{\mathcal{O}(\uhp, \mathbb{C})}
\newcommand\JacFormCont[4]{J_{#1, \underline{#2}}^{cont.}(#3,#4)}
\newcommand\JacFormHol[4]{J_{#1, \underline{#2}}^{\mathcal{O}}(#3,#4)}
\newcommand\JacFormWHol[4]{J_{#1, \underline{#2}}^{!}(#3,#4)}
\newcommand\JacFormWeak[4]{J_{#1, \underline{#2}}^{weak}(#3,#4)}
\newcommand\JacForm[4]{J_{#1, \underline{#2}}(#3,#4)}
\newcommand\JacFormCusp[4]{J_{#1, \underline{#2}}^{cusp}(#3,#4)}

\newcommand\ModFormHol[3]{M_{#1}^{\mathcal{O}}(#2,#3)}
\newcommand\ModFormWHol[3]{M_{#1}^{!}(#2,#3)}
\newcommand\ModForm[3]{M_{#1}(#2,#3)}
\newcommand\ModFormCusp[3]{S_{#1}(#2,#3)}

\newcommand\FormalSeriesSet{\sum_{\mathbf{j}\in \mulindZn}\holfnsm T^\mathbf{j}}
\newcommand\FormalSeriesSetLb{\sum_{\mathbf{j} \gg -\infty}\holfnsm T^\mathbf{j}}
\newcommand\FormalSeriesSetLbf[1]{\sum_{\mathbf{j} \geq {#1}}\holfnsm T^\mathbf{j}}
\newcommand\FormalModularFormSet[4]{M_{#1, #2}^{\mathfrak{B}}(#3, #4)}
\newcommand\FormalModularFormSubset[5]{M_{#1, #2}^{\mathfrak{B},\,{#5}}(#3, #4)}
\newcommand\FormalSeries[1]{\sum_{\mathbf{j}\in \mulindZn}{#1}_{\mathbf{j}} T^\mathbf{j}}
\newcommand\FormalSeriesi[2]{\sum_{\mathbf{#2}\in \mulindZn}{#1}_{\mathbf{#2}} T^\mathbf{#2}}

\newcommand\eleglptRa[5]{\left(\left(\begin{smallmatrix}{#1} & {#2} \\ {#3} & {#4}\end{smallmatrix}\right),{#5}\left({#3}^\prime \tau+{#4}^\prime \right)^{\frac{1}{2}}\right)}
\newcommand\eleglptRsimple[5]{\left(\left(\begin{smallmatrix}{#1} & {#2} \\ {#3} & {#4}\end{smallmatrix}\right),{#5}\right)}

\newcommand\eleactgrp[8]{\left(\left(\begin{smallmatrix}{#1} & {#2} \\ {#3} & {#4}\end{smallmatrix}\right),{#5}\left({#3}^\prime \tau+{#4}^\prime \right)^{\frac{1}{2}},[#6,#7],#8\right)}

\newcommand\eleactgrpnsimple[8]{\left(\left(\begin{smallmatrix}{#1} & {#2} \\ {#3} & {#4}\end{smallmatrix}\right),{#5},[#6,#7],#8\right)}

\title{Taylor expansions of Jacobi forms and linear relations among theta series}
\author{Xiao-Jie Zhu}

\begin{document}

\begin{abstract}
We study Taylor expansions of Jacobi forms of lattice index. As the main result, we give an embedding from certain space of such forms, whether scalar-valued or vector-valued, integral-weight or half-integral-weight, of any level, with any character, into a product of finitely many spaces of modular forms. As an application, we investigate linear relations among Jacobi theta series of lattice index. Many linear relations among the second powers of such theta series associated with the $D_4$ lattice and $A_3$ lattice are obtained, along with relations among the third powers of series associated with the $A_2$ lattice. We present the complete SageMath code for the $D_4$ lattice.
\end{abstract}

\maketitle
\let\oldthefootnote=\thefootnote
\let\thefootnote\relax
\footnotetext{This work is supported by the Fundamental Research Funds for the Central Universities of China (Grant No. 22120210556).}
\footnotetext{\textsl{Address.} School of Mathematical Sciences, Tongji University, 1239 Siping Road, Shanghai, P.R. China}
\footnotetext{\textsl{E-mail.} zhuxiaojiemath@outlook.com}
\let\thefootnote=\oldthefootnote

\tableofcontents

\section{Introduction}
\label{sec:Introduction}

\subsection{Taylor expansions of Jacobi forms}
\label{subsec:Taylor expansions of Jacobi forms}
The concept of Jacobi forms of lattice index, or equivalently, Jacobi forms on $\uhp \times \numC^\dimn$ ($\uhp$ denotes the upper half plane), was first investigated by Gritsenko \cite{Gri88}, where the author studied their relationship to modular forms on the orthogonal group of a quadratic space of signature $(2,\dimn)$, and the action of Hecke operators on them. This generalized the theory of classical Jacobi forms developed by Eichler and Zagier \cite{EZ85}. A further generalization to Siegel-Jacobi forms on $\uhp_j \times \numC^{j \times \dimn}$ was obtained by Ziegler \cite{Zie89} in the spirit of Eichler and Zagier. The present paper focuses on Jacobi forms of lattice index, so by the term ``Jacobi forms'' we mean forms of lattice index. Ajouz, under supervision of Skoruppa and Walling, developed systematically the Hecke theory of such forms, and the relation to elliptic modular forms, in his doctoral thesis \cite{Ajo15}.

Ajouz investigated liftings to and from elliptic modular forms via many tools, such as a generalization of Skoruppa-Zagier liftings \cite{SZ88}, Shimura correspondence (\cite{Shi73}, \cite{Niw75} and \cite{Cip83}), stable isomorphisms between lattices (\cite{Nik79}), theta decompositions and Weil representations (\cite[\S 2.3, \S 2.4]{Ajo15}). On the contrary, in the present paper, we shall investigate the relation between Jacobi forms and elliptic modular forms via Taylor expansions with respect to the point $0 \in \numC^\dimn$. In this direction, there are many excellent works. Eichler and Zagier has worked out the scalar index case. Ibukiyama \cite{Ibu12} investigated Taylor expansions of Jacobi forms on $\uhp_j \times \numC^{j \times 1}$. Bringmann, Mahlburg, and Rhoades \cite{BMR14} considered those of mock-Jacobi forms. Later Bringmann \cite{Bri18} considered those of non-holomorphic Jacobi forms. More recently, Ittersum \cite{vI21}, motivated by the problem of finding complex-valued functions on all partitions that are quasimodular forms for congruence subgroups, developed a theory of Taylor expansions of strictly meromorphic quasi-Jacobi forms.

We describe briefly our main result (Theorem \ref{thm:mainEmbed}). For any weight (integral or half-integral), any integral lattice (even or odd), any finite index subgroup of the modular group, and any representation with a finite index kernel, we obtain embeddings of the space of Jacobi forms (vector-valued or scalar-valued) into the direct product of finite number of spaces of elliptic modular forms. A key point is how to choose these spaces of modular forms, for which we give a computational criterion (Proposition \ref{prop:whenFLP0Nonzero}, realized as a SageMath program \cite{Sage}).

To avoid overlapping the main text, we illustrate the main theorem with a typical example (scalar-valued and integral-weight). Let $L=\numZ^2$, equipped with a quadratic form\footnote{The reason why we choose such a ``strange'' quadratic form is that for more familiar form of determinant $3$, such as $x^2-xy+y^2$, we need at least five spaces of modular forms to construct an embedding to the best of our ability, while for the form presented here, we need only three spaces.} $Q(x, y)=\frac{1}{2}x^2+xy+2y^2$. The bilinear form associated with this quadratic form is $B((x_1,y_1), (x_2,y_2))=Q(x_1+x_2, y_1+y_2)-Q(x_1,y_1)-Q(x_2,y_2)$. The Gram matrix is $G_{\mathfrak{B}}=\tbtmat{1}{1}{1}{4}$, so $Q(x,y)=\frac{1}{2}(x,y)G_{\mathfrak{B}}(x,y)^{T}$. Put $\underline{L}=(L, B)$. Let $k$ be a positive integer, $G$ be a finite index subgroup of $\slZ$, and $\chi \colon G \rightarrow \numC^\times$ be a group character whose kernel is of finite index. The space $\JacForm{k}{L}{G}{\chi}$, by definition, consists of all holomorphic functions $\phi(\tau,z)$ on $\uhp \times \numC^2$ subject to the following conditions:
\begin{enumerate}
\item For any $\tbtmat{a}{b}{c}{d} \in G$,
\begin{equation*}
\phi\left(\frac{a\tau+b}{c\tau+d},\frac{z}{c\tau+d}\right)=(c\tau+d)^k\etp{\frac{cQ(z)}{c\tau+d}}\chi\tbtmat{a}{b}{c}{d}\phi(\tau),
\end{equation*}
\item For any $v,w \in L$,
\begin{equation*}
\phi\left(\tau,z+\tau v+w\right)=\etp{-\tau Q(v)-B(v,z)}\phi(\tau,z),
\end{equation*}
\item For any $\tbtmat{a}{b}{c}{d} \in \slZ$, we have
\begin{multline*}
(c\tau+d)^{-k}\etp{-\frac{cQ(z)}{c\tau+d}}\phi\left(\frac{a\tau+b}{c\tau+d},\frac{z}{c\tau+d}\right) \\
=\sum_{\twoscript{n \in \numgeq{Q}{0},\, t \in L^\sharp}{Q(t) \leq n}}c(n,t)\etp{n\tau + B(t,z)},
\end{multline*}
where $c(n,t) \in \numC$, and the series converges absolutely.
\end{enumerate}
The notation $\etp{x}$ means $\exp(2\uppi\rmi x)$, and $L^\sharp$ denotes the dual lattice of $L$. The reader may also compare this definition to \cite[\S 9, Definition]{GSZ19}, and should notice a subtle difference: there is no factor $\etp{Q(v+w)}$ in the second condition in our definition, but there is such one in \cite{GSZ19}. (As a side effect, in \cite{GSZ19}, the variable $t$ in $c(n,t)$ ranges over $L_{ev}^\sharp$, not $L^\sharp$.) Both definitions make sense, but they correspond to different spaces! For the definition in broad generalities, including the case of half-integral weights, of even or odd lattices, and of scalar-valued or vector-valued forms, see Definition \ref{deff:JacobiLikeForm} and \ref{deff:JacobiForm}. Put $q^n=\etp{n\tau}$ for tradition. Let $\ModForm{k}{G}{\chi}$ and $\ModFormCusp{k}{G}{\chi}$ denote the space of modular forms, and that of cusp forms of weight $k$ on the group $G$ with character $\chi$ respectively. Then our main result, for this special case, says that:
\begin{thm*}[A special case of Corollary \ref{coro:embeddingDet3}]
The map
\begin{align*}
\JacForm{k}{L}{G}{\chi} &\rightarrow \ModForm{k}{G}{\chi} \times \ModFormCusp{k+1}{G}{\chi} \times \ModFormCusp{k+2}{G}{\chi} \\
\phi &\mapsto (D_0(\phi), D_1(\phi), D_2(\phi))
\end{align*}
is a $\numC$-linear embedding, where
\begin{align*}
\phi(\tau,z)&=\sum_{n,t}c(n,t)q^n\etp{B(t,z)}, \\
D_0(\phi) &= \sum_n \left(\sum_t c(n,t)\right)q^n, \\
D_1(\phi) &= \sum_n \left(\sum_t c(n,t)\cdot(t_1+4t_2)\right)q^n, \\
D_2(\phi) &= \sum_n \left(\sum_t c(n,t)\cdot\left(k(t_1+4t_2)^2-4n\right)\right)q^n.
\end{align*}
\end{thm*}
It is proper in this place to explain the relation of ``Taylor expansions'' in the title and the main theorem: each coefficient of terms of $D_0(\phi)$, $D_1(\phi)$, $D_2(\phi)$ is some modified Taylor coefficient of $\phi$.

We develop necessary tools to prove the main theorem from Section \ref{sec:Jacobi forms of lattice index} to \ref{sec:Jacobi theta series and theta decompositions}. In Section \ref{sec:Jacobi forms of lattice index}, we recall basic definitions and properties of Jacobi forms in broad generalities. In Section \ref{sec:Formal Laurent series which transform like modular forms}, we develop some operators used to define the sequence of operators $D_0(\phi)$, $D_1(\phi)$ and so on. To get rid of the problem of convergence, we work with formal Laurent series of several variables which transform like modular forms. At the end of this section, we obtain operators that map formal Laurent series of low weights to series of high weights. See Proposition \ref{prop:DkpMappingPlusMinus} for details. Note that results in this section may be of independent interest. In Section \ref{sec:Connections with modular forms}, we describe operators that map Jacobi forms to modular forms. It turns out that for each vector $\mathbf{p} \in \numgeq{Z}{0}^\dimn$, where $\dimn$ is the number of elliptic variables, there is an operator $D_{k, \mathbf{p}}$ that maps Jacobi forms of weight $k$ to elliptic modular forms of weight $k+s(\mathbf{p})$, where $s(\mathbf{p})$ is the sum of $\mathbf{p}$'s components. See Theorem \ref{thm:ModularFormFromJacobiForm} for details. As would be discussed in Remark \ref{rema:vanIttersum}, this theorem belongs to van Ittersum. We reprove it in a different approach. Then we give the $q$-expansion of each operator $D_{k, \mathbf{p}}$ acting on Jacobi forms (Proposition \ref{prop:FourierCoeffDkp}), which plays an important role in the following. At the end of this section (Proposition \ref{prop:isoFormalModForm}), we prove the direct product of all these operators, gives an isomorphism from all formal Laurent series which transform like modular forms, onto the direct product of infinitely many spaces of modular forms, hence gives an embedding of the space of Jacobi forms, into the same codomain. This proof is in the spirit of Eichler and Zagier (compare it to \cite[p. 34]{EZ85}). Now for proving the main theorem, it remains to show how to choose finitely many operators such that the direct product is still injective, for which we need the technique of theta decompositions. Thus, in Section \ref{sec:Jacobi theta series and theta decompositions} we review the theory of Jacobi theta series of lattice index, Weil representations and theta decompositions. We prove a version of theta decompositions (Theorem \ref{thm:thetaDecomposition}) which serves our purpose. The rest of this section is devoted to discussing the relation between theta decompositions and Taylor expansions. After these preparations we state and prove the main theorem in Section \ref{sec:The main theorem and its corollaries}. The finitely many operators $D_{k, \mathbf{p}}$ are choosen in such way that in the theta decomposition of any Jacobi form $\phi$, $D_{k, \mathbf{p}}\phi=0$ for all these $\mathbf{p}$ could imply the theta coefficients of $\phi$ all vanish, and hence $\phi$ itself vanishes. The rest of this section contains some immediate corollaries dealing with lattices of determinant $2$ and $3$, and a computational criterion (Proposition \ref{prop:whenFLP0Nonzero}) for deciding what $\mathbf{p}$'s to choose.

\subsection{Linear relations among theta series}
\label{subsec:Linear relations among theta series}
In a systematic treatment of ``cubic theta functions'' \cite{Sch13}, Schultz defined, among others, nine theta functions, related to $\sum q^{m^2+mn+n^2}\zeta_1^{m}\zeta_2^n$, and studied their inversion formulas, addition formulas and modular equations. See \cite[\S 2.3]{Sch13} for the precise definition of these nine functions. He proved some linear relations among the third powers of these functions (\cite[Corollary 4.3]{Sch13}), which generalize \cite[Equation (2.3)]{BB91}. These nine functions, can be rewritten as Jacobi theta series of lattice index associated with the $A_2$ lattice as follows:
\begin{equation*}
\theta_{\alpha,\beta}(\tau,z)=\sum_{v \in \numZ^2}\etp{B(\beta,v)}\etp{\tau Q(\alpha+v)+B(\alpha+v,z)},
\end{equation*}
where $(\tau,z) \in \uhp\times\numC^2$, $B(x,y)=x\cdot \tbtmat{2}{-1}{-1}{2} \cdot y^T$, and $Q(x)=\frac{1}{2}B(x,x)$ for $x, y \in \numC^{1 \times 2}$. Then Schultz's identities are equivalent to
\begin{thm*}
Let $\alpha$ be one of the followings: $(0,0),\, (2/3,1/3),\, (1/3,2/3)$, and $\beta$ be one of the followings: $(0,0),\, (1/3,1/3),\, (2/3,2/3)$. We have
\begin{equation*}
\theta_{\alpha, \beta}^3 + \theta_{0,0}^3 = \theta_{\alpha, 0}^3 + \theta_{0, \beta}^3.
\end{equation*}
\end{thm*}

As an application of the main theorem (Theorem \ref{thm:mainEmbed}), we give an algorithm for searching and proving this type of identities, not only for the $A_2$ lattice, but for arbitrary lattice in principle. More precisely, for any even integral positive definite symmetric bilinear form $B$ on $L=\numZ^\dimn$ (extended to $\numC^\dimn$ by bilinearity), and suitably choosen $\alpha$ and $\beta$, we could find all linear relations among functions $\theta_{\alpha,\beta}^{N'}$, where $N'$ is a fixed positive integer dividing the level of $(L,B)$ (the least positive integer $N$ such that $N\cdot Q(x) \in \numZ$ for all $x \in \numZ^\dimn$). The algorithm roughly says the following.
\begin{thm*}[A rough version of Theorem \ref{thm:algorithmFindLinearRelations}]
Let $\mathfrak{I}$ be the set of all pairs $(\mathbf{p},n)$ with $\mathbf{p}$ ranging over a finite subset of $\numgeq{Z}{0}^\dimn$ such that the product $\prod D_{N'\dimn/2, \mathbf{p}}$ is injective, and
\begin{equation*}
n=0,\,1,\,2,\dots, \left[\frac{\delta_N N^2}{24}\left(\frac{N'\dimn}{2}+s(\mathbf{p})\right)\prod_{p \mid N}\left(1-\frac{1}{p^2}\right)\right].
\end{equation*}
For any $(\alpha, \beta) \in L^\sharp/L\times (L^\sharp+\frac{1}{N'}L)/L^\sharp$ with $Q(\alpha) \in \frac{1}{N'}\numZ$, we associate the theta series $\theta_{\alpha,\beta}^{N'}$ with an $\mathfrak{I}$-indexed sequence $\Theta_{\alpha,\beta}\colon \mathfrak{I}\rightarrow \numC$ whose $(\mathbf{p},n)$-term is equal to
\begin{multline*}
\sum_{\twoscript{v_1,\dots, v_{N'} \in L}{Q(\alpha+v_1)+\dots+Q(\alpha+v_{N'})=n}}\etp{B(\beta,v_1+\dots +v_{N'})}\\
\times P_{\mathbf{p}}\left(n,\frac{(\alpha+v_1)+\dots (\alpha+v_{N'})}{N'}\right),
\end{multline*}
where $P_{\mathbf{p}}$ is certain $(\dimn+1)$-ary polynomial. Then any linear relation among $\theta_{\alpha,\beta}^{N'}$'s holds if and only if the linear relation with the same coefficients among $\Theta_{\alpha,\beta}$'s holds.
\end{thm*}
For the precise version (and definitions of concepts not introduced yet), see Theorem \ref{thm:algorithmFindLinearRelations}. For the definition of the polynomial $P_{\mathbf{p}}$, see Definition \ref{deff:PkpM}.

To find linear relations among $\theta_{\alpha,\beta}^{N'}$, we need only calculate the vectors $\Theta_{\alpha,\beta}$, and then produce linear relations among them using standard linear algebra algorithm. We use SageMath \cite{Sage} to do all computations. Besides Schultz's identities for the $A_2$ lattice, we have worked out identities for other lattices, which are concluded in Table \ref{table:identitiesTheta}.
\begin{table}[ht]
\centering
\caption{Identities of the form $\theta_{\alpha,\beta}^{N'}+\theta_{0,0}^{N'}=\theta_{\alpha,0}^{N'}+\theta_{0,\beta}^{N'}$ found and proved by Theorem \ref{thm:algorithmFindLinearRelations}, where $\alpha$'s ($\beta$'s resp.) means the number of $\alpha$'s ($\beta$'s resp.) and $\abs{\mathfrak{I}}$ equals the length of each vector $\Theta_{\alpha, \beta}$ \label{table:identitiesTheta}}
\begin{tabular}{llllll}
\toprule
Lattice & $N'$ & $\alpha$'s & $\beta$'s & $\abs{\mathfrak{I}}$ & Ref. \\
\midrule
$D_4$ & $2$ & $4$ & $4$ & $258$ & Theorem \ref{thm:D4linearRelation}\\
$A_2$ & $3$ & $3$ & $3$ & $99$ & Theorem \ref{thm:A2linearRelation}\\
$A_3$ & $2$ & $2$ & $4$ & $635$ & Theorem \ref{thm:A3linearRelation2}\\
\bottomrule
\end{tabular}
\end{table}

Section \ref{sec:Application: Linear relations among theta series} is devoted to this application. We first prove that (Proposition \ref{prop:thetaAlphaBetaJacobiForm}) the functions $\theta_{\alpha, \beta}$ are Jacobi forms, more exactly, in the space $\JacForm{\dimn/2}{L}{\Gamma_1(N)}{\chi_{\alpha,\beta}}$, for any even integral lattice $\underline{L}=(L,B)$, where $\chi_{\alpha,\beta}$ is certain linear character depending on $\alpha$ and $\beta$. Then in Proposition \ref{prop:thetaNptransformation}, we answer the question for which $\alpha$ and $\beta$, functions $\theta_{\alpha, \beta}^{N'}$ are in the same space, that is, with the same character. Immediately after that, we state and prove the main theorem of this section, Theorem \ref{thm:algorithmFindLinearRelations}. The remaining of this section is divided into five subsections. Each of the first three subsections deals with one lattice in Table \ref{table:identitiesTheta}. In \S \ref{subsec:The root lattice A2A2}, we present a non-trivial example showing that, in Theorem \ref{thm:mainEmbed}, we can not choose the set of $\mathbf{p}$'s with small cardinality. In \S \ref{subsec:The principal binary form of discriminant -15}, we show that the third (fifth) powers of nine (twenty-five) functions $\theta_{\alpha, \beta}$ associated with the quadratic form $x^2+xy+4y^2$ are linearly independent.

In Appendix \ref{apx:SageMath code}, we present the complete SageMath source code used to produce and check linear relations among series for the $D_4$ lattice, along with detailed discussion on the usage and on how to fit the program for other lattices.

\subsection{Notations}
\label{subsec:Notations}
The symbols $\numZ$, $\numQ$, $\numR$, $\numC$ denote the ring of integers, rationals, reals and complex numbers respectively. The symbol $R_{> n}$ ($R_{\geq n}$ resp.) refers to the set of elements in $R$ that are greater than (or equal to) $n$, The symbol $\numC^\times$ denotes the multiplicative group of nonzero complex numbers. The symbol $\projQ$ denotes the set of projective lines over $\numQ$ ($\numQ$-subspaces of dimension one), whose elements can be represented by formal fractions $a/b$.

For a commutative ring $R$ (with multiplicative identity), the general linear group $\mathrm{GL}_{2}(R)$ is the group of $2 \times 2$ matrices over $R$ with invertible determinant, and the special linear group $\mathrm{SL}_{2}(R)$ is the subgroup of $\mathrm{GL}_{2}(R)$ consisting of matrices of determinant $1$. If $R$ is a subring of $\numR$, then the symbol $\mathrm{GL}^+_{2}(R)$ denote the subgroup consisting of matrices of positive determinant. For a subgroup $G$ of $\glpR$, $\widetilde{G}$ denote its double cover (see the second paragraph of Section \ref{sec:Jacobi forms of lattice index}), and $\overline{G}$ denotes $G/\{\pm I\}$ if $-I \in G$, or $G$ itself if $-I \notin G$, where $I=\tbtmat{1}{0}{0}{1}$. We put $\pslZ=\overline{\slZ}$.

Throughout this paper, $V$ always denotes a linear space of dimension $\dimn \in \numgeq{Z}{1}$, and $\mathcal{V}$ its complexification $\numC \otimes_\numR V$. The upper half plane $\uhp$, is the subset of $\numC$ consisting of numbers with positive imaginary part. Besides $\mathcal{V}$, the symbol $\mathcal{W}$ always denotes a complex linear space of dimension $\dimm \in \numgeq{Z}{1}$. These sets serve as domains and codomains of Jacobi forms considered here --- a Jacobi form is implicitly assumed as a function from $\mydom$ to $\myran$. See the first paragraph of Section \ref{sec:Jacobi forms of lattice index}. By $C(\mydom, \myran)$, $\mathcal{O}(\mydom, \myran)$, we mean the set of continuous maps, and that of holomorphic maps from $\mydom$ to $\myran$. By $\GlW$, we mean the group of non-singular linear operators on $\mathcal{W}$.

The symbols $\heigrp{V}$, $\heigrpa{L}{n}$ denote some kind of Heisenberg groups. Refer to the paragraph following the proof of Lemma \ref{lemm:slashOperatorIsAction} for details. The symbol $L$ usually denotes a $\numZ$-lattice in $V$, which is assumed to be integral at most time, and $\underline{L}$ means $L$ equipped with a bilinear form $B$ (i.e., $\underline{L}=(L,B)$). By $L^\sharp$, or more precisely, $\underline{L}^\sharp$, we mean the dual lattice of $L$ (see the paragraph followed by Proposition \ref{prop:ComplexFourierSeries}). For a positive integer $N$, the notation $\underline{L}_N$ denotes the lattice with the same underlying $\numZ$-module $L$, but a different bilinear form $N\cdot B$.

For simplicity, put $\mathfrak{e}(z)=\exp{2\uppi\rmi z}$ with $z \in \numC$, and $q^n=\etp{n\tau}$ with $\tau \in \uhp,\, n \in \numQ$. The notation $\Gamma(z)$ refers to the Euler $\Gamma$-function. For two maps $f_1$ and $f_2$, $f_1\circ f_2$ means the map that sends $x$ to $f_1(f_2(x))$, provided that the codomain of $f_2$ coincides with the domain of $f_1$. For $x \in \numR$, the floor $[x]$ equals the largest integer not exceeding $x$. For two groups $G$ and $H$ where $H$ is a subgroup of $G$, $[G \colon H]$ means the index of $H$ in $G$. For a nonzero integer $d$, $\legendre{d}{-1}=\sgn{d}$ denotes the sign of $d$, and $\legendre{0}{-1}=1$ by convention.%For non-negative integers $a$ and $b$ with $b \leq a$, the notation $\binom{a}{b}$ denotes the binomial coefficient.

Finally, for notations concerning formal Laurent series, refer to first two paragraphs of Section \ref{sec:Formal Laurent series which transform like modular forms}.

\section{Jacobi forms of lattice index}
\label{sec:Jacobi forms of lattice index}
We recall basic concepts and properties concerning \emph{Jacobi forms of lattice index}. The readers may consult, for instance, \cite[\S 2.4]{Ajo15}, \cite[\S 3.6]{Boy15}, \cite[\S 9]{GSZ19}, \cite[\S 2.2]{Moc19}, \cite[\S 2]{Wil19} for more details. Throughout this paper, the symbol $V$ always denotes a finite dimensional real vector space (whose dimension will be denoted by $\mathfrak{n}$), and $\mathcal{V}$ its complexification, i.e., $\mathcal{V}=V \oplus \rmi V$. Instead of $v+\rmi w$, we use the notation $[v,w]$ to represent elements in $\mathcal{V}$, where $v,w \in V$.The symbol $\mathcal{W}$ denotes another finite dimensional complex vector space of complex dimension $\mathfrak{m}$. The functions studied here are holomorphic ones from $\mydom$ to $\myran$, where $\uhp$ denotes the upper half plane of the complex numbers. Let $B \colon V \times V \rightarrow \numR$ be a symmetric $\numR$-bilinear form and by abuse of language, $B$ also denotes its $\numC$-bilinear extension to $\mathcal{V} \times \mathcal{V}$. For any additive subgroup $X$ of $\mathcal{V}$, we write $\underline{X}$ to mean the group $X$ equipped with the restriction of the form $B$ to $X \times {}X$, so $\underline{\mathcal{V}}=(\mathcal{V},B)$. There is another ``extension'' of $B$ on $V \times V$ to $\mathcal{V} \times \mathcal{V}$, which is denoted by $A$ and is defined by $A([v_1,w_1],[v_2,w_2])=B(v_1,w_2)-B(v_2,w_1)$. The form $A$ on $\mathcal{V} \times \mathcal{V}$ is alternating (anti-symmetric). We shall use $Q$ to denote the quadratic form induced by $B$, that is, $Q(v)=\frac{1}{2}B(v,v)$ with $v \in \mathcal{V}$.

To state the modular and abelian transformation laws, and to define operators such as \emph{Hecke operators} and \emph{double coset operators}, we need a group, whose elements represent transformations of functions studied in this paper. The group used here is the semidirect product of a general linear group and a \emph{Heisenberg group}, which we now explain in detail. Let $\glpR$ be the group of $2 \times 2$ real matrices with positive determinants, $\slR$ be the subgroup of matrices of determinant $1$, and $\slZ$ be the subgroup of matrices with integeral entries in $\slR$. To deal with half-integer-weight modular forms, we shall use the double cover of $\glpR$, which is denoted by $\glptR$, whose elements are of the form $\eleglptRa{a}{b}{c}{d}{\varepsilon}$, where $\tbtmat{a}{b}{c}{d} \in \glpR$, $\varepsilon = \pm 1$, and $\tbtmat{a^\prime}{b^\prime}{c^\prime}{d^\prime}$=$\left(\det \tbtmat{a}{b}{c}{d}\right)^{-\frac{1}{2}} \cdot \tbtmat{a}{b}{c}{d}$. Note that by the function $z \mapsto z^r$, we mean $z \mapsto \exp{r\log z}$, where we choose the branch of $\log$ such that $-\uppi < \Im \log z \leq \uppi$, and that $\varepsilon\left(c^\prime \tau+d^\prime \right)^{\frac{1}{2}}$ means a function on $\tau \in \uhp$.  For an element $\gamma = \tbtmat{a}{b}{c}{d} \in \glpR$, we define $j(\gamma,\tau)=c^\prime\tau+d^\prime$, and $\gamma(\tau)=\frac{a\tau+b}{c\tau+d}=\frac{a^\prime\tau+b^\prime}{c^\prime\tau+d^\prime}$. The composition of $\glptR$ is given by
\begin{multline*}
\eleglptRa{a_1}{b_1}{c_1}{d_1}{\varepsilon_1}\cdot\eleglptRa{a_2}{b_2}{c_2}{d_2}{\varepsilon_2} \\
= \left(\left(\begin{smallmatrix} a_1 & b_1 \\ c_1 & d_1\end{smallmatrix}\right)\left(\begin{smallmatrix} a_2 & b_2 \\ c_2 & d_2\end{smallmatrix}\right),\varepsilon_1\varepsilon_2\left(c_1^\prime\frac{a_2^\prime\tau+b_2^\prime}{c_2^\prime\tau+d_2^\prime}+d_1^\prime\right)^\frac{1}{2}\left(c_2^\prime\tau+d_2^\prime\right)^\frac{1}{2}\right).
\end{multline*}
For any subgroup $G$ of $\glpR$, the notation $\widetilde{G}$ means the preimage of $G$ under the natural projection from $\glptR$ onto $\glpR$, so we have $\sltR$ and $\sltZ$. For matrix $\gamma=\tbtmat{a}{b}{c}{d} \in \glpR$, the symbol $\widetilde{\gamma}$ means $\left(\tbtmat{a}{b}{c}{d},(c\tau+d)^{\frac{1}{2}}\right)$. Moreover, we sometimes write $\eleglptRsimple{a}{b}{c}{d}{\varepsilon}$ instead of $\eleglptRa{a}{b}{c}{d}{\varepsilon}$. For more details on such metaplectic groups, see \cite[\S 1.1]{Bru02}.

By the Heisenberg group $\heigrp{V}$, we mean the group of elements $([v,w],\xi)$ with $[v,w] \in \mathcal{V}$ and $\xi \in S^1=\{z \in \numC \colon \abs{z}=1 \}$. The composition law is given by
\begin{multline*}
([v_1,w_1],\xi_1)\cdot([v_2,w_2],\xi_2) \\
=([v_1+v_2,w_1+w_2],\xi_1\xi_2 \exp{\uppi\rmi A([v_1,w_1],[v_2,w_2])}).
\end{multline*}
The group $\glptR$ acts from right on $\heigrp{V}$ as follows:
\begin{align*}
\heigrp{V} \times \glptR &\rightarrow \heigrp{V} \\
\left(\left([v,w],\xi\right), \eleglptRa{a}{b}{c}{d}{\varepsilon}\right) &\mapsto \left([v,w]\tbtmat{a^\prime}{b^\prime}{c^\prime}{d^\prime},\xi\right),
\end{align*}
where $[v,w]\tbtmat{a^\prime}{b^\prime}{c^\prime}{d^\prime}$ means $[a^\prime v+c^\prime w,b^\prime v+d^\prime w]$. Moreover, since
\begin{equation*}
A([v_1,w_1]\tbtmat{a^\prime}{b^\prime}{c^\prime}{d^\prime},[v_2,w_2]\tbtmat{a^\prime}{b^\prime}{c^\prime}{d^\prime}) =
A([v_1,w_1],[v_2,w_2]),
\end{equation*}
each element of $\glptR$ induces a group isomorphism of $\heigrp{V}$. Hence we can define the semi-direct product $\actgrp{V}$, using the above action. The elements in $\actgrp{V}$ are written in the form $\eleactgrp{a}{b}{c}{d}{\varepsilon}{v}{w}{\xi}$. For $\gamma \in \glptR$, we always use the same notation to express $(\gamma,[0,0],1) \in \actgrp{V}$. Similarly, the symbol $([v, w], \xi) \in \heigrp{V}$ sometimes means corresponding $\left(\widetilde{I},[v,w],\xi\right) \in \actgrp{V}$, where $I$ is the $2\times 2$ identity matrix.

The group $\actgrp{V}$ acts from left on the domain $\mydom$ as follows:
\begin{align}
\label{eq:actOnDomain}
(\actgrp{V}) \times (\mydom) &\rightarrow \mydom \\
\left(\eleactgrp{a}{b}{c}{d}{\varepsilon}{v}{w}{\xi}, (\tau, z)\right) &\mapsto \left(\frac{a^\prime\tau+b^\prime}{c^\prime\tau+d^\prime},\frac{z+\tau v+w}{c^\prime\tau+d^\prime}\right). \notag
\end{align}
We use the notation $\gamma Z$ with $\gamma \in \actgrp{V}$ and $Z=(\tau,z) \in \mydom$ to denote the resulting point when $\gamma$ acts on $Z$.

Besides $j(\gamma, \tau)$, We need another \emph{automorphic factor}, which is denoted by $J_B$, and is defined as follows:
\begin{align*}
J_B \colon \left(\actgrp{V}\right) \times (\mydom)  &\rightarrow \numC^\times \\
\left(\eleactgrp{a}{b}{c}{d}{\varepsilon}{v}{w}{\xi}, (\tau, z)\right) &\mapsto \\
\xi\cdot\etp{-\frac{c^\prime}{c^\prime\tau+d^\prime}Q(z+\tau v+w)+\tau Q(v)+B(v,z)+\frac{1}{2}B(v,w)}.
\end{align*}
\begin{lemm}
\label{lemm:JBcocycle}
The function $J_B$ satisfies the cocycle condition. That is to say, for $\gamma_1, \gamma_2 \in \actgrp{V}$ and $Z \in \mydom$, one has
\begin{equation*}
J_B(\gamma_1\gamma_2, Z)=J_B(\gamma_1, \gamma_2 Z)\cdot J_B(\gamma_2, Z).
\end{equation*}
\end{lemm}
\begin{proof}
A tedious but straightforward calculation only using definitions.
\end{proof}

Now we come to a position to define the transformations on functions form $\mydom$ to $\myran$, using the group $\actgrp{V}$. The symbols $\contfns$, $\smofns$ and $\holfns$ denote the (complex) spaces of continuous functions, smooth functions, and holomorphic functions respectively. Then $\actgrp{V}$ acts from right on $\contfns$. To give this action, let $\gamma=\eleactgrp{a}{b}{c}{d}{\varepsilon}{v}{w}{\xi} \in \actgrp{V}$, $Z=(\tau,z) \in \mydom$ and $f \in \contfns$. Then
\begin{equation*}
(f \vert_{k,B}\gamma)(Z)=\left(\varepsilon\sqrt{c^\prime\tau+d^\prime}\right)^{-2k}J_B(\gamma, Z)f(\gamma Z).
\end{equation*}
In the above formula, $k$ is a number called the \emph{weight}, and $\vert_{k,B}$ is called the \emph{slash operator} of weight $k$ and form $B$.
\begin{lemm}
\label{lemm:slashOperatorIsAction}
Suppose $k \in \halfint$. Then the slash operater is a right group action; that is, for $f \in \contfns$ and $\gamma_1, \gamma_2 \in \actgrp{V}$, we have $\left(f \vert_{k,B} \gamma_1\right) \vert_{k,B} \gamma_2= f \vert_{k,B}(\gamma_1\gamma_2)$, and the identity element of $\actgrp{V}$ gives the identity map on $\contfns$.
\end{lemm}
Note that when we restrict the set of functions to $\smofns$, or to $\holfns$, this action still remains well-defined.
\begin{proof}
This follows from that $j(\gamma,\tau)$ and $J_B(\gamma, Z)$ satisfy cocycle condition (The case of $j(\gamma,\tau)$ is classical, and the case of $J_B(\gamma, Z)$ is in Lemma \ref{lemm:JBcocycle}.), and that \eqref{eq:actOnDomain} is a left group action.
\end{proof}

By a $\numZ$-\emph{lattice} $L$ in $\underline{V}$, we mean a free $\numZ$- module in $\underline{V}$, with a $\numZ$-basis which is also a $\numR$-basis of $V$. As mentioned above, set $\underline{L}=(L,B\vert_{L\times L})$. We say that $L$ is \emph{integral} if $B(L,L) \subseteq \numZ$, and $L$ is \emph{even} if $Q(L) \subseteq \numZ$. If $L$ is integral but not even, then we say $L$ is \emph{odd}. Set $\heigrp{L}=\{([v,w],\xi) \in \heigrp{V} \colon v,w \in L\}$. Then $\heigrp{L}$ is a subgroup of $\heigrp{V}$, and for any subgroup $G$ of $\slZ$, the set $\widetilde{G} \ltimes \heigrp{L}$ is a subgroup of $\actgrp{V}$. If $B(L,L) \subseteq \numQ$, we can make the Heisenberg group a cover of $L \times L$ of finite index. To do this, define $\heigrpa{L}{n}=\{([v,w],\xi) \in \heigrp{L} \colon \xi^{2n} = 1\}$ for $n \in \numgeq{Z}{1}$. This is a subgroup when $B(L,L) \subseteq \frac{1}{n}\numZ$. Particularly, if $L$ is integral, then $\heigrpa{L}{1}$ is a subgroup, which is used to state the elliptic transformation laws for Jacobi forms. Though there is a subgroup of $\heigrp{L}$ that is isomorphic to $L \times L$ when $L$ is even, we shall not use this subgroup for the sake of dealing with even and odd lattices simultaneously. Note that $\heigrpa{L}{1}$ is abelian. 
\begin{conv}
\label{conv1}
From now on, we always assume that $L$ is an integral lattice, and $G$ is a subgroup of $\slZ$. So $\Jacgrp{G}{L}$ is a subgroup of $\actgrp{V}$. Moreover, $k$ always denotes a half integer (or an integer), and the weight of Jacobi forms considered is assumed to be $k$, unless explicitly specified.
\end{conv}

\begin{deff}
\label{deff:JacobiLikeForm}
Let $\rho \colon \Jacgrp{G}{L} \rightarrow \GlW$ be a group representation, and $\phi \in \contfns$. We say that $\phi$ transforms like a Jacobi form of weight $k$ and index $\underline{L}$ under the group $G$ with the representation $\rho$, if for any $\gamma \in \Jacgrp{G}{L}$, we have $\phi \vert_{k,B}\gamma = \rho(\gamma) \circ \phi$. Denote the space of such functions by $\JacFormCont{k}{L}{G}{\rho}$, and the subspace of holomorphic ones by $\JacFormHol{k}{L}{G}{\rho}$.
\end{deff}
Note that the slash operator and applying a linear operator on the left commute, i.e., $(L\circ\phi)\vert_{k,B}\gamma=L\circ(\phi\vert_{k,B}\gamma)$ for any linear operator $L$ on $\myran$. As a consequence, to show that a map $\phi$ transforms like a Jacobi form, it suffices to verify the tansformation laws corresponding to a set of generators of $\Jacgrp{G}{L}$. 

To make spaces of such (holomorphic) functions finite-dimensional, hence arithmetically interesting, we shall pose growth conditions at cusps of the quotient $\mydom$ modulo $\Jacgrp{G}{L}$, which will be given by Fourier series. So we recall some fundamental facts about Fourier series on the $\dimn$-torus, stated in a form suits us. Assume that $B$ (on $V \times V$) is nondegenerate, that is, the map $v \mapsto (w \mapsto B(v,w))$ is a linear isomorphism from $V$ onto the dual space of $V$. Then the function $V \rightarrow \numC, \quad x \mapsto \etp{B(v,x)}$ is $L$-periodic, where $v$ is a vector in the dual lattice $L^\sharp$ of $L$ in $\underline{V}$. (The \emph{dual lattice} $L^\sharp$ is defined to be the subset of $v \in V$ such that $B(v, L) \subseteq \numZ$, which turns out to be a lattice of the same rank if $B$ is nondegenerate. To indicate the bilinear form, we also wrtie $\underline{L}^\sharp$.) On the other hand, there is an inner product on $V$ induced by $L$, given by $(f,g)=\vol{\mathfrak{F}}^{-1}\int_{\mathfrak{F}}f(x)\overline{g(x)}\diff x$, where $f$ and $g$ are complex-valued continuous $L$-periodic functions, and $\mathfrak{F}$ is any closed fundamental domain being the preimage of $[0,1]^{\dimn}$ under a linear isomorphism from $V$ onto $\mulindRn$ which is simultaneously an isomorphism of $L$ and $\mulindZn$ (recall that $\dimn=\dim V$). This inner product $(\cdot,\cdot)$, is independent of the choice of $\mathfrak{F}$ and the Lebesgue measure posed on $V$. A fundamental theorem in Fourier analysis asserts that any complex-valued continuous $L$-periodic function $f$ on $V$ can be decomposed as $f(x) = \sum_{t \in L^\sharp}a_t \etp{B(t,x)}$ with $a_t \in \numC$, where the equality means that the right-hand side converges unconditionally with repect to the inner product to the left-hand side. Put another way, the functions $\{x \mapsto \etp{B(t,x)} \colon t \in L^\sharp\}$ forms a maximal orthonormal basis of the space of complex-valued continuous $L$-periodic functions equipped with the inner product $(\cdot,\cdot)$. Since we deal mainly with holomorphic functions, the following proposition should be useful.

\begin{prop}
\label{prop:ComplexFourierSeries}
Let $\mathcal{V}^\prime$ be a complex vector space of dimension $d \in \numgeq{Z}{1}$ and $(e_1, e_2, \ldots, e_d)$ be a basis. Set $V^\prime=\oplus_{1 \leq j \leq d}\numR e_j$. Let $D$ be a domain (connected open set) in $V^\prime$, and $B^\prime \colon V^\prime \times V^\prime \rightarrow \numR$ be a nondegenerate symmetric bilinear form, which extends $\numC$-linearly to $\mathcal{V}^\prime \times \mathcal{V}^\prime$. Let $L^\prime$ be a $\numZ$-lattice in $(V^\prime, B^\prime)$, and $f \colon V^\prime \oplus \rmi D \rightarrow \numC$ be a holomorphic function. Suppose that $f(z+v)=f(z)$ for any $z \in V^\prime \oplus \rmi D$ and $v \in L^\prime$. Then there exists a unique sequence $c_t$ for $t \in L^{\prime\sharp}$, such that
\begin{equation*}
f(z)=\sum_{t \in L^{\prime\sharp}}c_t \etp{B^\prime(t,z)}.
\end{equation*}
The series converges normally on $V^\prime \oplus \rmi D$. The coefficients are given by
\begin{equation*}
c_t=\frac{1}{\vol{\mathfrak{F}^\prime}}\int_{\mathfrak{F}^\prime}f(x+ \rmi y) \etp{-B(t,x+ \rmi y)} \diff x,
\end{equation*}
where $y \in D$, the integral is with respect to the Lebesgue measure transformed from $\numR^d$ via the coordinate map from $V^\prime$ onto $\numR^d$ induced by the basis $(e_j)_{1 \leq j \leq d}$, and $\mathfrak{F}^\prime$ is the inverse image of $[0,1]^d$ under this coordinate map.
\end{prop}
\begin{proof}
For the case of $B'$ being positive definite, see \cite[Appendix to \S 6.8]{Fre11}. The general case follows from this special case.
\end{proof}
By saying that a function series $\sum_{t}f_t$ defined on $V^\prime \oplus \rmi D$ converges normally, we mean that for any compact subset $K$ of $V^\prime \oplus \rmi D$, we have $\sum_t \sup_{x \in K}\abs{f_t(x)} < \infty$. This implies absolute convergence and compactly uniform convergence. But normal convergence is stronger.
%Note that for the tube domain $V^\prime \oplus \rmi D$ (actually for any locally compact space) a series of functions defined on it converges locally uniformly if and only if this series converges uniformly on any compact subset of $V^\prime \oplus \rmi D$. Also
Note that this proposition generalizes immediately to functions taking values in finite-dimensional complex vector spaces. In this situation, in the definition of normal convergence, we shall substitute the absolute value by the norm. Since any norms on a finite-dimensional real or complex vector space are equivalent, we need not refer to the norm when talking about normal convergence. For us, the proposition applies to the case $V^\prime = \numR \times V$, $D = \numR_{>0} \times V$, $\mathcal{V}^\prime = \numC \times \mathcal{V}$, so $V^\prime \oplus \rmi D$ is $\uhp \times \mathcal{V} \subseteq \mathcal{V}^\prime$.

\begin{prop}
\label{prop:ForuierExpansionOfJacobiLikeForm}
Use the notations in Definition \ref{deff:JacobiLikeForm}. Suppose that the index of $G$ in $\slZ$ is finite, $B$ is positive definite, and $\ker \rho$ is of finite index in $\Jacgrp{G}{L}$. Put $N_1=[\sltZ: \widetilde{G}]$ and $N_2=[\Jacgrp{G}{L}: \ker \rho]$. Then for any $\phi \in \JacFormHol{k}{L}{G}{\rho}$ and $\gamma \in \Jacgrp{\slZ}{L}$, there exist a $\mathcal{W}$-valued sequence $c(n,t)$ with $n \in \frac{1}{N_1N_2}\numZ$ and $t \in \frac{1}{N_1N_2}L^\sharp$, such that
\begin{equation}
\label{eq:ForuierExpansionOfJacobiLikeForm}
\phi \vert_{k,B}\gamma(\tau,z)= \sum_{n,\, t}c(n,t) q^n \etp{B(t,z)}.
\end{equation}
This series converges normally. As a consequence, it converges absolutely and locally uniformly at any $(\tau,z) \in \mydom$.
\end{prop}
\begin{proof}
Put $f=\phi \vert_{k,B}\gamma$. We shall show that
\begin{equation*}
f(\tau+N_1N_2u,z+N_1N_2v)=f(\tau,z)
\end{equation*}
for $u \in \numZ$ and $v \in L$. Then \eqref{eq:ForuierExpansionOfJacobiLikeForm} follows from this and Proposition \ref{prop:ComplexFourierSeries}. Note that $T=\tbtmat{1}{1}{0}{1}$ as usual in the theory of elliptic modular forms. We have
\begin{align*}
f(\tau+N_1N_2u,z+N_1N_2v)&=\phi\vert_{k,B}\gamma\vert_{k,B}\widetilde{T}^{N_1N_2u}([0,v],1)^{N_1N_2}(\tau,z) \\
&=\phi\vert_{k,B}(\gamma\widetilde{T}^{u}\gamma^{-1})^{N_1N_2}\vert_{k,B}(\gamma([0,v],1)\gamma^{-1})^{N_1N_2}\vert_{k,B}\gamma(\tau,z) \\
&=\rho((\gamma\widetilde{T}^{u}\gamma^{-1})^{N_1N_2})\circ\rho((\gamma([0,v],1)\gamma^{-1})^{N_1N_2})\circ f(\tau,z)\\
&=f(\tau,z).
\end{align*}
In the last step of above dedection, we have use the fact $(\gamma\widetilde{T}^{u}\gamma^{-1})^{N_1} \in \Jacgrp{G}{L}$ and $(\gamma([0,v],1)\gamma^{-1})^{N_1} \in \Jacgrp{G}{L}$, which is a consequence of the definition of $N_1$, and the fact $\rho((\gamma\widetilde{T}^{u}\gamma^{-1})^{N_1N_2})=\rho((\gamma([0,v],1)\gamma^{-1})^{N_1N_2})$ are both the identity map on $\myran$, which is a consequence of the definition of $N_2$.
\end{proof}

\begin{deff}
\label{deff:JacobiForm}
Let $\rho \colon \Jacgrp{G}{L} \rightarrow \GlW$ be a group representation and $\phi \in \JacFormHol{k}{L}{G}{\rho}$. Suppose that $[\slZ: G] < \infty$, $[\Jacgrp{G}{L}: \ker \rho] < \infty$ and $B$ is positive definite. We say $\phi$ is a \emph{weakly holomorphic Jacobi form}, if for any $\gamma \in \Jacgrp{\slZ}{L}$, there exists an $N_0 \in \numQ$ such that in the series \eqref{eq:ForuierExpansionOfJacobiLikeForm}, $c(n,t)=0$ unless $n \geq N_0$. We say that $\phi$ is a \emph{weak Jacobi form}, if it is weakly holomorphic with $N_0=0$ for any $\gamma$. And we say that $\phi$ is a \emph{Jacobi form} (\emph{Jacobi cusp form}, resp.), if $c(n,t)=0$ unless $n-Q(t) \geq 0$ ($n-Q(t) > 0$ resp.) in the series \eqref{eq:ForuierExpansionOfJacobiLikeForm} for any $\gamma$. The spaces of weakly holomorphic forms, weak forms, Jacobi forms, Jacobi cusp forms are denoted by $\JacFormWHol{k}{L}{G}{\rho}$, $\JacFormWeak{k}{L}{G}{\rho}$, $\JacForm{k}{L}{G}{\rho}$ and $\JacFormCusp{k}{L}{G}{\rho}$ respectively.
\end{deff}
In practice, one need only check the expansion \eqref{eq:ForuierExpansionOfJacobiLikeForm} for finitely many $\gamma \in \Jacgrp{\slZ}{L}$:
\begin{prop}
\label{prop:FourierExpasionCusps}
Use the notations and assumptions in Definition \ref{deff:JacobiForm}. Let $\{s_i \colon 1 \leq i \leq w\}$ be a system of representatives of the quotient $\projQ$ modulo $G$. Suppose that $\gamma_i \in \slZ$ such that $\gamma_i(\infty)=s_i$ for $1 \leq i \leq w$. Suppose $\phi \in \JacFormHol{k}{L}{G}{\rho}$. Then $\phi \in \JacForm{k}{L}{G}{\rho}$ if and only if $\phi \vert_{k,B} \widetilde{\gamma_i}$ has an expansion of the form \eqref{eq:ForuierExpansionOfJacobiLikeForm} with $n-Q(t) \geq 0$ for each $i$. There are similar conclusions for $\JacFormWHol{k}{L}{G}{\rho}$, $\JacFormWeak{k}{L}{G}{\rho}$ and $\JacFormCusp{k}{L}{G}{\rho}$.
\end{prop}

Still use the notations and assumptions in Definition \ref{deff:JacobiForm}. We give some properties on Fourier coefficients, some of which are used in the proof of Proposition \ref{prop:FourierExpasionCusps}. We use $c^{\gamma}(n,t)$, or more precisely $c^\gamma_\phi(n,t)$ to denote the coefficients $c(n,t)$ in \eqref{eq:ForuierExpansionOfJacobiLikeForm}, to indicate that it depends on $\gamma$ and $\phi$. It is useful to extend the domain of  the map $(n,t) \mapsto c^\gamma(n,t)$ to $\numR \times V$, such that $c^\gamma(n,t) = 0$ unless $n \in \frac{1}{N}\numZ$ and $t \in \frac{1}{N}L^\sharp$ ($N=N_1N_2$). Let $v \in L$ and $\gamma \in \sltZ$; applying the operator $\vert_{k,B}([v,0],1)$ to $\phi \vert_{k,B} \gamma$ on the right, computing the Fourier coefficients in two ways and using the uniquessness of Fourier coefficients yield that
\begin{equation}
\label{eq:coefficientTransform}
\rho([v,0]\gamma^{-1},1)\circ c^\gamma(n+Q(v)+B(t,v), t+v)=c^\gamma(n,t)
\end{equation}
for any $(n,t) \in \numR \times V$. A similar formula holds for $\gamma \in \Jacgrp{\slZ}{L}$. As a consequence, for $\phi \in \JacFormHol{k}{L}{G}{\rho}$ and $\gamma \in \Jacgrp{\slZ}{L}$, if $c^\gamma_\phi(n,t) \neq 0$, then $ c^\gamma_\phi(n+Q(v)+B(t,v),t+v) \neq 0$ for any $v \in L$. Moreover, if $\phi \in \JacFormWHol{k}{L}{G}{\rho}$, and $c^\gamma_\phi(n,t) \neq 0$ implies that $n \geq N_0$ with $N_0 \in \numQ$, then $c^\gamma_\phi(n,t) \neq 0$ would imply that $n+Q(v)+B(t,v) \geq N_0$. Hence for any fixed $n \geq N_0$, the Fourier coefficient $c^\gamma_\phi(n,t)$ can be nonzero only if $Q(t) \leq n-N_0+\max_{x \in \mathfrak{F}}\{Q(x)\}$, where $\mathfrak{F}$ is the same as in the paragraph previous to Proposition \ref{prop:ComplexFourierSeries}. Now we can give the proof of Proposition \ref{prop:FourierExpasionCusps}.
\begin{proof}
The ``only if'' part is obvious, and the proofs for $\JacFormWHol{k}{L}{G}{\rho}$, $\JacFormWeak{k}{L}{G}{\rho}$ and $\JacFormCusp{k}{L}{G}{\rho}$ are similar to that for $\JacForm{k}{L}{G}{\rho}$. Hence we only prove the ``if'' part for $\JacForm{k}{L}{G}{\rho}$. Suppose that $\phi \vert_{k,B} \widetilde{\gamma_i}$ has an expansion of the form \eqref{eq:ForuierExpansionOfJacobiLikeForm} with $n-Q(t) \geq 0$ for each $i$, where the Fourier coefficient $c(n,t)$ is denoted by $c_i(n,t)$ here, to indicate its dependence on $\gamma_i$. Let
\begin{equation*}
\eleactgrpnsimple{a}{b}{c}{d}{\varepsilon}{v}{w}{\xi} \in \Jacgrp{\slZ}{L}
\end{equation*}
be arbitrary, and put $\gamma=\tbtmat{a}{b}{c}{d}$. We shall prove that $\phi\vert_{k,B}\eleactgrpnsimple{a}{b}{c}{d}{\varepsilon}{v}{w}{\xi}$ has an expansion of the form \eqref{eq:ForuierExpansionOfJacobiLikeForm} with the property $c(n,t) \neq 0 \implies n-Q(t) \geq 0$. Suppose $\gamma(\infty)$ is equivalent to $s_i$ modulo $G$, so there is some $\gamma' \in G$ such that $\gamma(\infty)=\gamma'(s_i)$. Thus, $\gamma_i^{-1}\gamma'^{-1}\gamma(\infty)=\infty$, and in consequence, there is some $e \in \numZ$ and $\delta = \pm 1$ such that $\gamma=\gamma'\gamma_i(\delta T^e)$, where $T=\tbtmat{1}{1}{0}{1}$. Hence, there is some $\varepsilon' = \pm 1$ such that
\begin{equation*}
\eleactgrpnsimple{a}{b}{c}{d}{\varepsilon}{v}{w}{\xi}=(\gamma',\varepsilon')\widetilde{\gamma_i}(\delta T^e,1,[v,w],\xi).
\end{equation*}
It follows that
\begin{align*}
&\phi\vert_{k,B}\eleactgrpnsimple{a}{b}{c}{d}{\varepsilon}{v}{w}{\xi}(\tau,z) \\
={} &\phi\vert_{k,B}(\gamma',\varepsilon')\vert_{k,B}\widetilde{\gamma_i}\vert_{k,B}(\delta T^e,1,[v,w],\xi)(\tau,z) \\
%={} &\rho(\gamma',\varepsilon') \circ \phi\vert_{k,B}\widetilde{\gamma_i}\vert_{k,B}(\delta T^e,1,[v,w],\xi)(\tau,z) \\
%={} &\sum_{n,\, t}\rho(\gamma',\varepsilon')(c_i(n,t)) q^n \etp{B(t,z)}\vert_{k,B}(\delta T^e,1,[v,w],\xi) \\
%={} &\delta^{-k}\xi\etp{\tau Q(v)+B(v,z)+\frac{1}{2}B(v,w)} \\
%&\times\sum_{n,\, t}\rho(\gamma',\varepsilon')(c_i(n,t)) q^n \etp{e\cdot n} \etp{B(t,\delta^{-1}(z+\tau v+w))} \\
={} &\delta^{-k}\xi\etp{\frac{1}{2}B(v,w)}\sum_{n,\, t}\rho(\gamma',\varepsilon')(c_i(n,t))\etp{e\cdot n+B(\delta t,w)} \\
& \times q^{n+Q(v)+B(\delta t, v)}\etp{B(\delta t+v, z)}.
\end{align*}
Using the change of variables corresponding to the following bijection
\begin{align*}
\numR \times V &\rightarrow \numR \times V \\
(n,t) &\mapsto (n+Q(v)+B(\delta t, v), \delta t + v),
\end{align*}
and applying \eqref{eq:coefficientTransform}, we obtain
\begin{multline*}
c(n,t)=\delta^{-k}\xi\etp{-\frac{1}{2}B(v,w)+e\cdot(n+Q(v)-B(t,v))+B(t,w)} \\
\times \rho(\gamma',\varepsilon')\circ\rho([\delta v,0]\widetilde{\gamma_i}^{-1},1)\circ c_i(n, \delta t).
\end{multline*}
Now if $n-Q(t)=n-Q(\delta t) < 0$, then $c_i(n, \delta t) = 0$, and hence $c(n, t) = 0$, which concludes the proof.
\end{proof}

\begin{conv}
\label{conv2}
This paper is aimed at investigating embeddings of $\JacFormHol{k}{L}{G}{\rho}$, $\JacFormWHol{k}{L}{G}{\rho}$, $\JacFormWeak{k}{L}{G}{\rho}$, $\JacForm{k}{L}{G}{\rho}$ and $\JacFormCusp{k}{L}{G}{\rho}$ into certain direct product of spaces of modular forms. Thus, we will at most time assume that $\rho \colon \Jacgrp{G}{L} \rightarrow \GlW$ is a group representation, $[\slZ: G] < \infty$, $[\Jacgrp{G}{L}: \ker \rho] < \infty$ and $B$ is positive definite. Moreover, we require that $\rho([0,0],-1)$ is the map sending $w \in \myran$ to $-w$ for the obvious reason.
\end{conv}

We conclude this section by recalling another basic property which would be used later, concerned with multiplication of two scalar-valued Jacobi forms of lattice index.

Suppose $B_1$ and $B_2$ are two positive definite symmetric $\numR$-bilinear forms on $V$. (As usual, $Q_i(x)=\frac{1}{2}B_i(x,x)$ for $i=1,2$.) Then their sum $B_1+B_2$, defined pointwisely, is also positive definite symmetric. Let $L_i^\sharp$ denote the dual lattice of $L$ in the inner product space $(V,B_i)$ ($i=1,2$), and $L_3^\sharp$ the dual of $L$ in $(V, B_1+B_2)$, where $L$ is any $\numZ$-lattice in $V$, not necessary integral. 
\begin{lemm}
\label{lemm:mulJacobiFormLemma1}
Suppose $N \in \numgeq{Z}{1}$. For any $t_1 \in \frac{1}{N}L_1^\sharp$ and $t_2 \in \frac{1}{N}L_2^\sharp$, there exists a unique $t \in V$ such that
\begin{equation*}
\etp{(B_1+B_2)(t,x)}=\etp{B_1(t_1,x)}\etp{B_2(t_2,x)}
\end{equation*}
for any $x \in V$. Moreover, this $t$ belongs to $\frac{1}{N}L_3^\sharp$, and can be obtained by the following way. Let $\mathfrak{B}$ be a $\numR$-basis of $V$, and $(t_1^1,\dots,t_1^\dimn)$, $(t_2^1,\dots,t_2^\dimn)$ be coordinate row vectors of $t_1$, $t_2$ with respect to $\mathfrak{B}$, respectively. Let $G_1, G_2$ be the Gram matrices of $B_1, B_2$ with respect to the basis $\mathfrak{B}$, respectively. Then the coordinate row vector $(t^1,\dots,t^\dimn)$ of $t$ is uniquely determined by the following matrix equation
\begin{equation*}
(t^1,\dots,t^\dimn)=((t_1^1,\dots,t_1^\dimn)G_1+(t_2^1,\dots,t_2^\dimn)G_2)(G_1+G_2)^{-1}.
\end{equation*}
\end{lemm}
\begin{proof}
A direct verification.
\end{proof}
\begin{lemm}
\label{lemm:mulJacobiFormLemma2}
Use notations and assumptions in the last lemma. Put $G'=G_1(G_1+G_2)^{-1}G_2$, then $G'$ is a positive definite symmetric matrix. Moreover, let $B'$ denote the bilinear form on $V\times V$ whose Gram matrix with respect to $\mathfrak{B}$ is $G'$, and put $Q'(x)=\frac{1}{2}B'(x,x)$. Then we have
\begin{equation*}
Q_1(t_1)+Q_2(t_2)-(Q_1+Q_2)(t)=Q'(t_1-t_2).
\end{equation*}
As a consequence, $(Q_1+Q_2)(t) \leq Q_1(t_1)+Q_2(t_2)$.
\end{lemm}
\begin{proof}
A direct verification.
\end{proof}
\begin{prop}
\label{prop:multiplicationScalarJacobiForm}
Let $k_1, k_2 \in \frac{1}{2}\numZ$ and $B_1, B_2$ be two positive definite symmetric $\numR$-bilinear forms on $V$. Let $L$ be a $\numZ$-lattice in $V$ that is integral both in $(V, B_1)$ and $(V,B_2)$. Let $G$ be a finite index subgroup of $\slZ$ and $\chi_i$ be a linear character on $\widetilde{G} \ltimes H((L, B_i), 1)$ whose kernel is of finite index, for $i=1,2$. It is required that $\chi_i([0,0],\xi)=\xi$. Let $f_1,f_2\colon \mydom \rightarrow \numC$ be two functions. Then we have following conclusions.
\begin{enumerate}
\item The map
\begin{align*}
\widetilde{G} \ltimes H((L, B_1+B_2), 1) &\rightarrow \numC^\times\\
\eleactgrpnsimple{a}{b}{c}{d}{\varepsilon}{v}{w}{\xi} &\mapsto \xi\chi_1\cdot\chi_2\eleactgrpnsimple{a}{b}{c}{d}{\varepsilon}{v}{w}{1}
\end{align*}
is a linear character, which is denoted by $\chi_1\ast\chi_2$. The notation $\chi_1\cdot\chi_2$ denotes the pointwise multiplication.
\item If $f_i \in J_{k_i, (L,B_i)}(G,\chi_i)$ for $i=1,2$, then
\begin{equation*}
f_1\cdot f_2 \in J_{k_1+k_2, (L,B_1+B_2)}(G,\chi_1\ast\chi_2).
\end{equation*}
Moreover, if one of $f_1$ and $f_2$ is a Jacobi cusp form, then so is $f_1\cdot f_2$.
\end{enumerate}
\end{prop}
\begin{proof}
The conclusion on $\chi_1\ast\chi_2$ can be proved by the definition of linear characters. Just pay attention to the fact that underlying group compositions of $H((L,B_1),1)$, $H((L,B_2),1)$ and $H((L,B_1+B_2),1)$ are different. Now suppose $f_i \in J_{k_i, (L,B_i)}(G,\chi_i)$ for $i=1,2$. We have
\begin{multline*}
f_1\cdot f_2 \vert_{k_1+k_2, B_1+B_2}\eleactgrpnsimple{a}{b}{c}{d}{\varepsilon}{v}{w}{\xi}\\
=\xi\cdot f_1\vert_{k_1, B_1}\eleactgrpnsimple{a}{b}{c}{d}{\varepsilon}{v}{w}{1}f_2\vert_{k_2, B_2}\eleactgrpnsimple{a}{b}{c}{d}{\varepsilon}{v}{w}{1}
\end{multline*}
for any $\eleactgrpnsimple{a}{b}{c}{d}{\varepsilon}{v}{w}{\xi} \in \glptR \ltimes H(V,B_1+B_2)$, from which the transformation laws (see Definition \ref{deff:JacobiLikeForm}) follow. The requirement on Fourier expansions (see Definition \ref{deff:JacobiForm}) follows from Lemma \ref{lemm:mulJacobiFormLemma1} and Lemma \ref{lemm:mulJacobiFormLemma2}.
\end{proof}
\begin{rema}
There are also parallel conclusions for weakly holomorphic Jacobi forms, and for weak Jacobi forms.
\end{rema}

\section{Formal Laurent series which transform like modular forms}
\label{sec:Formal Laurent series which transform like modular forms}
For the sake of freely defining differential operators used to construct embeddings of spaces of Jacobi forms into products of spaces of modular forms, we shift the focus from functions in $\holfns$ to formal series whose coefficients are modular forms in $\holfnsm$. This section is devoted to developing such a theory. Let $T_1,T_2,\ldots,T_\mathfrak{n}$ be formal variables. Let $\mathbf{j}=(j_1,j_2,\ldots,j_\mathfrak{n})$ be in $\mulindZn$ and set $T^\mathbf{j}=T_1^{j_1}T_2^{j_2}\cdots T_\mathfrak{n}^{j_{\mathfrak{n}}}$. Use the symbol $\FormalSeriesSet$ to denote the set of all formal series whose terms are of the form $h_\mathbf{j}(\tau)T^\mathbf{j}$ with $h_\mathbf{j} \in \holfnsm$ for each $\mathbf{j} \in \mulindZn$. This set becomes a $\numC$-space, also a $\holfnsmC$-module under the ordinary addition and scalar multiplication. The symbol $\FormalSeriesSetLb$ denotes the subspace of those series $\FormalSeries{h} \in \FormalSeriesSet$ with the property there exists a $\mathbf{j}_0 \in \mulindZn$ such that $h_\mathbf{j} \neq 0$ implies\footnote{By $\mathbf{j} \geq \mathbf{j}_0$, we mean each component of $\mathbf{j}$ is not less than that of $\mathbf{j}_0$. Moreover, $\mathbf{j} > \mathbf{j}_0$ means $\mathbf{j} \geq \mathbf{j}_0$ but $\mathbf{j} \neq \mathbf{j}_0$.} $\mathbf{j} \geq \mathbf{j}_0$. Equipped with the ordinary multiplication, the subspace $\FormalSeriesSetLb$ is a $\numC$-algebra (associative and unitary). We also need the subspace $\FormalSeriesSetLbf{\mathbf{j}_0}$ with $\mathbf{j}_0$ being a fixed vector.

Next we shall define slash operators on formal series. For this purpose, it is necessary to introduce some nonstandard notations. The symbol $\mymatset{R}{\dimn}{\dimn}$ denotes the set of all $\dimn\times \dimn$ matrices with entries in $R$. For $G$, $\Lambda$ in $\mymatset{R}{\dimn}{\dimn}$, $s(G)$ denotes the sum of all entries of $G$, and $G^{\Lambda}$ denotes the product of all $g^{\lambda}$ where $(g,\lambda)$ ranges over corresponding entries in $G$ and $\Lambda$. By $\Lambda_{\ast,j}$ and $\Lambda_{i,\ast}$, we mean the $j$-th column and $i$-th row of $\Lambda$ respectively. Put $\pi_i(\Lambda)=s(\Lambda_{i,\ast}) + s(\Lambda_{\ast,i})$, and $\pi(\Lambda)=(\pi_1(\Lambda),\dots,\pi_\dimn(\Lambda))$. We define
%\begin{equation*}
%\binom{s(\Lambda)}{\Lambda}=\frac{s(\Lambda)!}{\prod_{\lambda}\lambda!},
%\end{equation*}
\begin{equation*}
\Lambda!=\prod_{\lambda}\lambda!,
\end{equation*}
with $\lambda$ ranging over all entries of $\Lambda$.% For two matrices $\Lambda_1,\Lambda_2 \in \mymatset{R}{\dimn}{\dimn}$ with $\Lambda_1+\Lambda_2=\Lambda$, put
%\begin{equation*}
%\binom{s(\Lambda)}{\Lambda_1,\Lambda_2}=\frac{s(\Lambda)!}{\prod_{1\leq i,j \leq \dimn}\lambda^{(1)}_{i,j}!\lambda^{(2)}_{i,j}!},
%\end{equation*}
%where $\lambda^{(1)}_{i,j}$ and $\lambda^{(2)}_{i,j}$ are the $(i,j)$-th entry of $\Lambda_1$ and $\Lambda_2$ respectively.
For $\mathbf{j}\in\mulindZn$, $s(\mathbf{j})$ denotes the sum of all components of $\mathbf{j}$, as in the matrix case. Under such notations, one has, for instance,
\begin{equation}
\label{eq:quadraticFormPower}
\frac{\left(\sum_{1\leq i,j \leq \dimn}g_{ij}T_iT_j\right)^s}{s!}=\sum_{\twoscript{\Lambda \in \mymatset{\numgeq{Z}{0}}{\dimn}{\dimn}}{s(\Lambda)=s}}\frac{G^\Lambda}{\Lambda!}T^{\pi(\Lambda)},
\end{equation}
where $g_{i,j}$ is the $(i,j)$-entry of $G \in \mymatset{\numC}{\dimn}{\dimn}$, and $s \in \numgeq{Z}{0}$.
\begin{deff}
\label{deff:slashOperatorFormalSeries}
Let $\mathfrak{B}=(e_1,e_2,\dots,e_\dimn)$ be an ordered $\numR$-basis of V, and let $G_\mathfrak{B}$ be the Gram matrix of $B$ with respect to $\mathfrak{B}$ (i.e. the $(i,j)$-th entry is $B(e_i,e_j)$.). Let $\FormalSeries{h}$ be a formal series in $\FormalSeriesSetLb$, and $\eleglptRsimple{a}{b}{c}{d}{\varepsilon} \in \sltR$. We define $\left(\FormalSeries{h}\right) \vert_{k, B}^{\mathfrak{B}}\eleglptRsimple{a}{b}{c}{d}{\varepsilon}=\FormalSeriesi{g}{p}$, where
\begin{equation}
\label{eq:slashOperatorFormalSeries}
%g_{\mathbf{p}}(\tau)=\varepsilon^{-2k}\sum_{\Lambda \in \mymatset{\numgeq{Z}{0}}{\dimn}{\dimn}}\frac{(-\uppi\rmi c)^{s(\Lambda)}}{s(\Lambda)!}\binom{s(\Lambda)}{\Lambda}G_{\mathfrak{B}}^{\Lambda}(c\tau+d)^{-(k+s(\mathbf{p})-s(\Lambda))}h_{\mathbf{p}-\pi(\Lambda)}\left(\frac{a\tau+b}{c\tau+d}\right).
g_{\mathbf{p}}(\tau)=\varepsilon^{-2k}\sum_{\Lambda \in \mymatset{\numgeq{Z}{0}}{\dimn}{\dimn}}\frac{(-\uppi\rmi cG_{\mathfrak{B}})^{\Lambda}}{\Lambda!}(c\tau+d)^{-(k+s(\mathbf{p})-s(\Lambda))}h_{\mathbf{p}-\pi(\Lambda)}\left(\frac{a\tau+b}{c\tau+d}\right).
\end{equation}
\end{deff}
We warn the readers that $0^0=1$ as usual in the theory of power series. Note that the sum used to define $g_{\mathbf{p}}$ is actually a finite sum, since for $\Lambda$ with $s(\Lambda)$ sufficiently large, we have $h_{\mathbf{p}-\pi(\Lambda)}=0$. One can also verify that, the formal series $\FormalSeriesi{g}{p}$ is again in $\FormalSeriesSetLb$, and if, in addition, $\FormalSeriesi{h}{j} \in \FormalSeriesSetLbf{\mathbf{j}_0}$, then $\FormalSeriesi{g}{p} \in \FormalSeriesSetLbf{\mathbf{j}_0}$ too. We shall prove that this is a group action, for which the following technical lemma is useful.
\begin{lemm}
\label{lemm:binomMatrix}
Suppose $\Lambda \in \mymatset{\numgeq{Z}{0}}{\dimn}{\dimn}$, and $z_1,\,z_2 \in \numC$. We have
\begin{equation*}
%\binom{s(\Lambda)}{\Lambda}(z_1+z_2)^{s(\Lambda)}=\sum_{\twoscript{\Lambda_1,\,\Lambda_2 \in \mymatset{\numgeq{Z}{0}}{\dimn}{\dimn}}{\lambda_1+\Lambda_2=\Lambda}}\binom{s(\Lambda)}{\Lambda_1,\Lambda_2}z_1^{s(\Lambda_1)}z_2^{s(\Lambda_2)}.
\frac{1}{\Lambda!}(z_1+z_2)^{s(\Lambda)}=\sum_{\twoscript{\Lambda_1,\,\Lambda_2 \in \mymatset{\numgeq{Z}{0}}{\dimn}{\dimn}}{\lambda_1+\Lambda_2=\Lambda}}\frac{1}{\Lambda_1!\Lambda_2!}z_1^{s(\Lambda_1)}z_2^{s(\Lambda_2)}.
\end{equation*}
\end{lemm}
\begin{proof}
By the binomial theorem, it suffices to prove that, for any non-negative integer $u, v$ satisfying $u+v=s(\lambda)$, the following formula holds:
%\begin{equation*}
%\sum_{\twoscript{\Lambda_1+\Lambda_2=\Lambda}{s(\Lambda_1)=u,\,s(\Lambda_2)=v}}\frac{1}{\Lambda_1!\Lambda_2!}=\frac{(u+v)!}{u!v!}\frac{1}{\Lambda!},
%\end{equation*}
%or equivalently,
\begin{equation}
\label{eq:binomMatrixToprove}
\sum_{\twoscript{\Lambda_1+\Lambda_2=\Lambda}{s(\Lambda_1)=u,\,s(\Lambda_2)=v}}\frac{u!}{\Lambda_1!}\frac{v!}{\Lambda_2!}=\frac{(u+v)!}{\Lambda!}.
\end{equation}
Let $X_{i,j}$ be $\dimn^2$ indeterminants ($1\leq i,\,j\leq\dimn$) and let $X$ be the matrix whose $(i,j)$-th entry is $X_{i,j}$. It is trivial that
\begin{equation}
\left(\sum_{1\leq i,\,j\leq\dimn}X_{i,j}\right)^{u+v}=\left(\sum_{1\leq i,\,j\leq\dimn}X_{i,j}\right)^{u}\left(\sum_{1\leq i,\,j\leq\dimn}X_{i,j}\right)^{v}.
\end{equation} 
Now evaluating the coefficients of $X^\Lambda$ in both sides of the above formula leads to \eqref{eq:binomMatrixToprove}.
\end{proof}

\begin{prop}
The slash operator given in Definition \ref{deff:slashOperatorFormalSeries} is a right group action of $\sltR$ on $\FormalSeriesSetLb$.
\end{prop}
\begin{proof}
The fact that identity element in $\sltR$ gives the identity map on $\FormalSeriesSetLb$ is obvious. So we should prove that, for $\gamma_1, \gamma_2 \in \sltR$, and $\FormalSeries{h} \in \FormalSeriesSetLb$, the following formula holds:
\begin{equation}
\label{eq:groupActionFormalSeries}
\left(\FormalSeries{h}\right) \vert_{k,B}^{\mathfrak{B}}\gamma_1 \vert_{k,B}^{\mathfrak{B}}\gamma_2 = \left(\FormalSeries{h}\right) \vert_{k,B}^{\mathfrak{B}}\gamma_1\gamma_2.
\end{equation}
We put $\gamma_i=\eleglptRsimple{a_i}{b_i}{c_i}{d_i}{\varepsilon_i}$ for $i=1,\,2$ and $\gamma_1\gamma_2=\eleglptRsimple{a_3}{b_3}{c_3}{d_3}{\varepsilon_3}$. Then for $\mathbf{p} \in \mulindZn$, the coefficient of $T^{\mathbf{p}}$ in the left-hand side of \eqref{eq:groupActionFormalSeries} is
\begin{multline}
\label{eq:groupActionLeftHand}
\varepsilon_1^{-2k}\varepsilon_2^{-2k}\sum_{\Lambda_1,\,\Lambda_2}\frac{(-\uppi\rmi)^{s(\Lambda_1)+s(\Lambda_2)}c_1^{s(\Lambda_1)}c_2^{s(\Lambda_2)}}{\Lambda_1!\Lambda_2!}G_{\mathfrak{B}}^{\Lambda_1+\Lambda_2} \\
\cdot (c_2\tau+d_2)^{-(k+s(\mathbf{p})-s(\Lambda_2))}\left(c_1\frac{a_2\tau+b_2}{c_2\tau+d_2}+d_1\right)^{-(k+s(\mathbf{p})-2s(\Lambda_2)-s(\Lambda_1))} \\
\cdot h_{\mathbf{p}-\pi(\Lambda_1)-\pi(\Lambda_2)}\left(\frac{a_3\tau+b_3}{c_3\tau+d_3}\right).
\end{multline}
For a fixed $\Lambda \in \mymatset{\numgeq{Z}{0}}{\dimn}{\dimn}$, the subsum satisfying $\Lambda_1+\Lambda_2=\Lambda$ in \eqref{eq:groupActionLeftHand} is
\begin{multline}
\label{eq:groupActionLeftHandSubsum}
\varepsilon_3^{-2k}\sum_{\Lambda_1+\Lambda_2=\Lambda}\frac{(-\uppi\rmi)^{s(\Lambda)}}{\Lambda_1!\lambda_2!}G_{\mathfrak{B}}^{\Lambda}(c_3\tau+d_3)^{-(k+s(\mathbf{p})-s(\Lambda))}\\
\cdot\left(\frac{c_1}{c_2\tau+d_2}\right)^{s(\Lambda_1)}\left(c_2(c_1\gamma_2\tau+d_1)\right)^{s(\Lambda_2)}h_{\mathbf{p}-\pi(\Lambda)}\left(\frac{a_3\tau+b_3}{c_3\tau+d_3}\right).
\end{multline}
Using Lemma \ref{lemm:binomMatrix}, we have
\begin{equation*}
\sum_{\Lambda_1+\Lambda_2=\Lambda}\frac{1}{\Lambda_1!\Lambda_2!}\left(\frac{c_1}{c_2\tau+d_2}\right)^{s(\Lambda_1)}\left(c_2(c_1\gamma_2\tau+d_1)\right)^{s(\Lambda_2)}=\frac{1}{\Lambda!}c_3^{s(\Lambda)}.
\end{equation*}
Inserting this into \eqref{eq:groupActionLeftHandSubsum}, we find that \eqref{eq:groupActionLeftHand} equals the coefficient of $T^{\mathbf{p}}$ in the right-hand side of \eqref{eq:groupActionFormalSeries}, form which \eqref{eq:groupActionFormalSeries} follows.
\end{proof}

Definition \ref{deff:slashOperatorFormalSeries} has an equivalent form, namely,
\begin{multline}
\label{eq:deffSlashFormalSeriesAlternate}
\left(\FormalSeries{h}\right) \vert_{k, B}^{\mathfrak{B}}\eleglptRsimple{a}{b}{c}{d}{\varepsilon}=\varepsilon^{-2k}(c\tau+d)^{-k} \\
\cdot\etp{-\frac{c}{c\tau+d}\cdot\frac{1}{2}\sum_{m,n}a_{m,n}T_mT_n}\sum_{\mathbf{j}\in\mulindZn}h_{\mathbf{j}}\left(\frac{a\tau+b}{c\tau+d}\right)\left(\frac{T_1}{c\tau+d}\right)^{j_1}\dots\left(\frac{T_\dimn}{c\tau+d}\right)^{j_\dimn},
\end{multline}
where $a_{m,n}$ is the $(m,n)$-th entry of $G_{\mathfrak{B}}$. Maybe this form is more natural, since it coincides with the slash operators on Jacobi forms. See the paragraph that precedes Lemma \ref{lemm:slashOperatorIsAction}. In fact, it is from \eqref{eq:deffSlashFormalSeriesAlternate} and the power series expansion of the exponential function that we work out \eqref{eq:slashOperatorFormalSeries}. We can also give an inverse formula expressing $h_{\mathbf{j}}$ by $g_{\mathbf{p}}$.

\begin{prop}
\label{prop:slashOperatorFormalInverse}
Let $\FormalSeries{h}$ and $\FormalSeriesi{g}{p}$ be two formal series in $\FormalSeriesSetLb$. Then \eqref{eq:slashOperatorFormalSeries} holds for any $\mathbf{p}$ if and only if the following relation holds for any $\mathbf{j}$:
\begin{equation}
\label{eq:slashOperatorFormalInverse}
h_{\mathbf{j}}\left(\frac{a\tau+b}{c\tau+d}\right)=\varepsilon^{2k}\sum_{\Lambda \in \mymatset{\numgeq{Z}{0}}{\dimn}{\dimn}}\frac{(\uppi\rmi cG_{\mathfrak{B}})^{\Lambda}}{\Lambda!}(c\tau+d)^{k+s(\mathbf{j})-s(\Lambda)}g_{\mathbf{j}-\pi(\Lambda)}(\tau).
\end{equation}
\end{prop}
\begin{proof}
\eqref{eq:slashOperatorFormalSeries}$\implies$\eqref{eq:slashOperatorFormalInverse}. Note that \eqref{eq:slashOperatorFormalSeries} and \eqref{eq:deffSlashFormalSeriesAlternate} are equivalent. Hence we have
\begin{multline*}
\varepsilon^{-2k}(c\tau+d)^{-(k+s(\mathbf{j}))}\sum_{\mathbf{j}\in\mulindZn}h_{\mathbf{j}}\left(\frac{a\tau+b}{c\tau+d}\right)T^{\mathbf{j}}\\
=\etp{\frac{c}{c\tau+d}\cdot\frac{1}{2}\sum_{m,n}a_{m,n}T_mT_n}\sum_{\mathbf{p}\in\mulindZn}g_{\mathbf{p}}(\tau)T^{\mathbf{p}}.
\end{multline*}
Now expand the first factor in the right-hand side of the above identity as a power series of $T_1,\dots,T_\dimn$ and compare the coefficients of $T^{\mathbf{j}}$ in both sides; the desired relation then follows.

\eqref{eq:slashOperatorFormalInverse}$\implies$\eqref{eq:slashOperatorFormalSeries}. Put $\left(\FormalSeries{h}\right) \vert_{k, B}^{\mathfrak{B}}\eleglptRsimple{a}{b}{c}{d}{\varepsilon}=\FormalSeriesi{g'}{p}$ Then by the first part of this proposition (which we have proved), $h_{\mathbf{j}}$ and $g'_{\mathbf{p}}$ also satisfy relation \eqref{eq:slashOperatorFormalInverse} with $g$ replaced by $g'$. By induction on $s(\mathbf{p})$ one finds that $g_{\mathbf{p}}=g'_{\mathbf{p}}$, for any $\mathbf{p} \in \mulindZn$, from which $\eqref{eq:slashOperatorFormalSeries}$ follows.
\end{proof}

We need a formula dealing with taking successive derivatives of \eqref{eq:slashOperatorFormalInverse}, which is important in the proof of Proposition \ref{prop:isoFormalModForm}, a main result of the next section.
\begin{lemm}
\label{lemm:sucDerModEqu}
Let $\FormalSeries{h}$ be a formal series in $\FormalSeriesSetLb$. Suppose $\sigma \in \GlW$, $\gamma=\left(\tbtmat{a}{b}{c}{d},\varepsilon\right) \in \sltR$, $k \in \halfint$, $\mathbf{p} \in \mulindZn$ and $G_{\mathfrak{B}} \in \mymatset{\numR}{\dimn}{\dimn}$. If the coefficient $h_{\mathbf{p}}$ satisfy
\begin{equation*}
h_{\mathbf{p}}\left(\frac{a\tau+b}{c\tau+d}\right)=\varepsilon^{2k}\sum_{\Lambda \in \mymatset{\numgeq{Z}{0}}{\dimn}{\dimn}}\frac{(\uppi\rmi cG_{\mathfrak{B}})^{\Lambda}}{\Lambda!}(c\tau+d)^{k+s(\mathbf{p})-s(\Lambda)}\sigma\circ h_{\mathbf{p}-\pi(\Lambda)}(\tau),
\end{equation*}
then for any $u \in \numgeq{Z}{0}$, we have
\begin{multline}
\dodtha{u}{h_{\mathbf{p}}}\left(\frac{a\tau+b}{c\tau+d}\right)=\varepsilon^{2k}\sum_{\Lambda \in \mymatset{\numgeq{Z}{0}}{\dimn}{\dimn}}\frac{(\uppi\rmi cG_{\mathfrak{B}})^{\Lambda}}{\Lambda!}(c\tau+d)^{k+s(\mathbf{p})-s(\Lambda)} \\
\times \sum_{0 \leq l \leq u}\binom{u}{l}\frac{\Gamma(k+s(\mathbf{p})-s(\Lambda)+u)}{\Gamma(k+s(\mathbf{p})-s(\Lambda)+l)}c^{u-l}(c\tau+d)^{u+l}\dodth{l}\sigma\circ h_{\mathbf{p}-\pi(\Lambda)}(\tau).
\end{multline}
\end{lemm}
\begin{proof}
Induction on $u$.
\end{proof}

We now describe connections between slash operators on Jacobi forms introduced in Section \ref{sec:Jacobi forms of lattice index} and that on formal series introduced in Definition \ref{deff:slashOperatorFormalSeries}. As above, let $\mathfrak{B}=(e_1,\dots,e_\dimn)$ be an ordered $\numR$-basis of $V$. We define a $\numC$-linear embedding, denoted by $\emb$, of $\holfns$ to $\FormalSeriesSetLbf{\mathbf{0}}$ as follows. For $f \in \holfns$, we expand the function $f(\tau, z_1e_1+\dots z_\dimn e_\dimn)$ as a power series of $(z_1,\dots,z_\dimn) \in \mulindCn$ around zero. Then $\emb(f)$ is defined as the resulting power series with $z_i$'s replaced by $T_i$'s.
\begin{prop}
\label{prop:diagramSlashOperator}
For any $\gamma \in \sltR$, the following diagram (of the category of complex vector spaces) commutes:
\begin{equation*}
\xymatrix{
    \holfns \ar[r]^{\vert_{k,B}\gamma} \ar[d]_{\emb} & \holfns \ar[d]^{\emb} \\
    \FormalSeriesSetLbf{\mathbf{0}} \ar[r]_{\vert_{k,B}^{\mathfrak{B}}\gamma} & \FormalSeriesSetLbf{\mathbf{0}}
}
\end{equation*}
\end{prop}
\begin{proof}
This is a direct consequence of $\eqref{eq:deffSlashFormalSeriesAlternate}$.
\end{proof}

There are another two types of linear operators acting on formal series, besides slash operators, which are important in investigating the Taylor expansion of Jacobi forms. The first is the following:
\begin{deff}
\label{deff:multiplyByTpOperator}
Let $\mathbf{p} \in \mulindZn$. The operator $\mathcal{T}^{\mathbf{p}}$ is defined on $\FormalSeriesSet$ that sends $\FormalSeries{h}$ to $\sum_{\mathbf{j}\in \mulindZn}{h}_{\mathbf{j}} T^{\mathbf{j}+\mathbf{p}}$.
\end{deff}

\begin{prop}
\label{prop:multiplyByTpOperatorAndSlashOperator}
Use notations in Definition \ref{deff:slashOperatorFormalSeries} and put $\gamma=\eleglptRsimple{a}{b}{c}{d}{\varepsilon}$. We have
\begin{equation*}
\left(\mathcal{T}^{\mathbf{p}}\FormalSeries{h}\right)\vert_{k-s(\mathbf{p}), B}^{\mathfrak{B}}\gamma=\mathcal{T}^{\mathbf{p}}\left(\FormalSeries{h}\vert_{k, B}^{\mathfrak{B}}\gamma\right).
\end{equation*}
\end{prop}
\begin{proof}
A straightforward calculation.
\end{proof}

The second one is the following. Note that $E_{i,j}$ denotes the matrix whose $(i,j)$-entry is $1$ but all other entries are $0$.
\begin{deff}
\label{deff:LkBOperator}
Let $\mathfrak{B}$ and $G_\mathfrak{B}$ be as in Definition \ref{deff:slashOperatorFormalSeries} and let $a_{i,j}$ be the $(i,j)$-entry of $G_\mathfrak{B}$. The linear operator $\mathcal{L}_{k,B}^{\mathfrak{B}}$ is defined by $\mathcal{L}_{k,B}^{\mathfrak{B}}\FormalSeries{h}=\FormalSeriesi{g}{p}$, where
\begin{equation*}
g_{\mathbf{p}}=s(\mathbf{p})(s(\mathbf{p})+2k-2)h_{\mathbf{p}}-4\uppi\sqrt{-1}\sum_{1\leq i,j\leq \dimn}a_{i,j}\dodt{h_{\mathbf{p}-\pi(E_{i,j})}}.
\end{equation*}
This is a linear operator on $\FormalSeriesSet$. Sometimes we write $\mathcal{L}_{k}$ instead of $\mathcal{L}_{k,B}^{\mathfrak{B}}$ when $B$ and $\mathfrak{B}$ are implicitly known.
\end{deff}

\begin{prop}
\label{prop:LkBOperator}
Let $\FormalSeries{h} \in \FormalSeriesSetLb$, and $\gamma \in \sltR$. Let $\mathfrak{B}$ be as in Definition \ref{deff:slashOperatorFormalSeries}. Then we have
\begin{equation}
\label{eq:LkBOperator}
\mathcal{L}_{k,B}^{\mathfrak{B}}\left(\FormalSeries{h}\vert_{k,B}^{\mathfrak{B}}\gamma\right)=\left(\mathcal{L}_{k,B}^{\mathfrak{B}}\FormalSeries{h}\right)\vert_{k,B}^{\mathfrak{B}}\gamma.
\end{equation}
\end{prop}
\begin{proof}
Put $\gamma=\eleglptRsimple{a}{b}{c}{d}{\varepsilon}$. Then the coefficient of $T^{\mathbf{p}}$-term in the left-hand side of \eqref{eq:LkBOperator} is
\begin{multline*}
\sum_{\Lambda}\frac{(-\uppi\sqrt{-1}cG_{\mathfrak{B}})^{\Lambda}}{\Lambda!}(c\tau+d)^{-(k+s(\mathbf{p})-s(\Lambda))}\cdot\Big( \\
s(\mathbf{p})(s(\mathbf{p})+2k-2)h_{\mathbf{p}-\pi(\Lambda)}\left(\frac{a\tau+b}{c\tau+d}\right) \\
-4\uppi\sqrt{-1}\sum_{i,j}a_{i,j}\cdot(-c)(k+s(\mathbf{p})-2-s(\Lambda))(c\tau+d)h_{\mathbf{p}-\pi(E_{i,j})-\pi(\Lambda)}\left(\frac{a\tau+b}{c\tau+d}\right) \\
-4\uppi\sqrt{-1}\sum_{i,j}a_{i,j}\cdot\left(\dodt h_{\mathbf{p}-\pi(E_{i,j})-\pi(\Lambda)}\right)\left(\frac{a\tau+b}{c\tau+d}\right)\Big).
\end{multline*}
On the other hand, the coefficient of $T^{\mathbf{p}}$-term in the right-hand side of \eqref{eq:LkBOperator} is
\begin{multline*}
\sum_{\Lambda}\frac{(-\uppi\sqrt{-1}cG_{\mathfrak{B}})^{\Lambda}}{\Lambda!}(c\tau+d)^{-(k+s(\mathbf{p})-s(\Lambda))}\cdot\Big( \\
s(\mathbf{p}-\pi(\Lambda))(s(\mathbf{p}-\pi(\Lambda))+2k-2)h_{\mathbf{p}-\pi(\Lambda)}\left(\frac{a\tau+b}{c\tau+d}\right) \\
-4\uppi\sqrt{-1}\sum_{i,j}a_{i,j}\left(\dodt h_{\mathbf{p}-\pi(\Lambda)-\pi(E_{i,j})}\right)\left(\frac{a\tau+b}{c\tau+d}\right) \Big).
\end{multline*}
Hence \eqref{eq:LkBOperator} is equivalent to
\begin{multline}
\sum_{\Lambda}F_\Lambda\cdot(s(\mathbf{p})(s(\mathbf{p})+2k-2)-s(\mathbf{p}-\pi(\Lambda))(s(\mathbf{p}-\pi(\Lambda))+2k-2))h_{\mathbf{p}-\pi(\Lambda)}\left(\frac{a\tau+b}{c\tau+d}\right) \\
=\sum_{\Lambda}F_\Lambda\cdot 4\uppi\sqrt{-1}\sum_{i,j}a_{i,j}(-c)(k+s(\mathbf{p})-2-s(\Lambda))(c\tau+d)h_{\mathbf{p}-\pi(E_{i,j})-\pi(\Lambda)}\left(\frac{a\tau+b}{c\tau+d}\right),
\end{multline}
where
\begin{equation*}
F_\Lambda=\frac{(-\uppi\sqrt{-1}cG_{\mathfrak{B}})^{\Lambda}}{\Lambda!}(c\tau+d)^{-(k+s(\mathbf{p})-s(\Lambda))}.
\end{equation*}
This formula can be proved, using the fact
\begin{equation*}
s(\mathbf{p})(s(\mathbf{p})+2k-2)-s(\mathbf{p}-\pi(\Lambda))(s(\mathbf{p}-\pi(\Lambda))+2k-2)=4s(\Lambda)(k-1+s(\mathbf{p})-s(\Lambda)),
\end{equation*}
and
\begin{equation*}
4\frac{(-\uppi\sqrt{-1}cG_{\mathfrak{B}})^\Lambda}{\Lambda!}s(\Lambda)=4\sum_{\twoscript{1\leq i,j \leq \dimn}{\lambda_{i,j}>0}}\frac{(-\uppi\sqrt{-1}cG_{\mathfrak{B}})^{\Lambda-E_{i,j}}}{(\Lambda-E_{i,j})!}(-\uppi\sqrt{-1}c\cdot a_{i,j}).
\end{equation*}
Here $\lambda_{i,j}$ means the $(i,j)$-entry of $\Lambda$. This concludes the proof.
\end{proof}

\begin{rema}
\label{remm:compositionTpLkB}
The composition $\mathcal{T}^{-\mathbf{p}}\circ\mathcal{L}_{k,B}^{\mathfrak{B}}$ is of particular interest when $s(\mathbf{p})=2$. For one-dimensional case, that is, the case $\dimn=1$ (in this case $\mathbf{p}$ is the integer $2$), this operator maps $\sum_{n \in \numZ}h_nT_1^n$ to $\sum_{n \in \numZ}((n+2)(n+2k)h_{n+2} - 8\uppi\sqrt{-1}m\dodt h_{n})T_1^n$, where we write $G_{\mathfrak{B}}=(2m)$. This is the modified heat operator introduced in the proof of \cite[Theorem 3.2]{EZ85}, but differs by a factor $-1$.
\end{rema}

After defining necessary operators and exploring their basic properties, we now introduce the main object in this section.
\begin{deff}
\label{deff:modularFormalSeries}
Let $H$ be a subgroup of $\slZ$, and $\rho\colon \widetilde{H}\rightarrow \GlW$ a representation. Fix an ordered $\numR$-basis $\mathfrak{B}$ of $V$. Let $\FormalSeries{h}$ be a formal series in $\FormalSeriesSetLb$. We say that $\FormalSeries{h}$ transforms like a modular form under the group $H$ of weight $k$ with representation $\rho$, if
\begin{equation*}
\FormalSeries{h}\vert_{k,B}^{\mathfrak{B}}\gamma=\FormalSeries{\rho(\gamma)\circ h}
\end{equation*}
for any $\gamma \in \widetilde{H}$. The $\numC$-space of all such series is denoted by $\FormalModularFormSet{k}{B}{H}{\rho}$.
\end{deff}
\begin{rema}
\label{rema:relationshipJacFormalSeries}
The relationship between Jacobi forms of lattice index and formal sereis which transform like modular forms are as follows. Recall notations in Convention \ref{conv1} and \ref{conv2}, and recall $\emb$ used in Proposition \ref{prop:diagramSlashOperator}. Then the image of $\JacFormHol{k}{L}{G}{\rho}$ under $\emb$ is contained in $\FormalModularFormSet{k}{B}{G}{\rho\vert_{\widetilde{G}}}$, which is a direct consequence of Proposition \ref{prop:diagramSlashOperator}.
\end{rema}
Some subspaces of $\FormalModularFormSet{k}{B}{H}{\rho}$ are also necessary for us. We define
\begin{align}
\label{eq:FormalModularFormSubsetVector}\FormalModularFormSubset{k}{B}{H}{\rho}{\mathbf{j_0}}&=\FormalModularFormSet{k}{B}{H}{\rho} \cap \FormalSeriesSetLbf{\mathbf{j_0}} \\
\FormalModularFormSubset{k}{B}{H}{\rho}{\mathbf{j_0},+}&=\{\FormalSeries{h} \in \FormalModularFormSubset{k}{B}{H}{\rho}{\mathbf{j_0}} \colon h_{\mathbf{j}}\neq 0 \implies 2 \mid s(\mathbf{j})\} \label{eq:FormalModularFormSubsetPlus}\\
\FormalModularFormSubset{k}{B}{H}{\rho}{\mathbf{j_0},-}&=\{\FormalSeries{h} \in \FormalModularFormSubset{k}{B}{H}{\rho}{\mathbf{j_0}} \colon h_{\mathbf{j}}\neq 0 \implies 2 \nmid s(\mathbf{j})\}. \label{eq:FormalModularFormSubsetMinus}
\end{align}
The following are another subspaces, in which $s_0$ is any integer:
\begin{align}
\label{eq:FormalModularFormSubsetScalar}\FormalModularFormSubset{k}{B}{H}{\rho}{s_0}&=\{\FormalSeries{h} \in \FormalModularFormSet{k}{B}{H}{\rho} \colon h_{\mathbf{j}}\neq 0 \implies s(\mathbf{j}) \geq s_0\} \\
\FormalModularFormSubset{k}{B}{H}{\rho}{s_0,+}&=\{\FormalSeries{h} \in \FormalModularFormSubset{k}{B}{H}{\rho}{s_0} \colon h_{\mathbf{j}}\neq 0 \implies 2 \mid s(\mathbf{j})\} \label{eq:FormalModularFormSubset2Plus}\\
\FormalModularFormSubset{k}{B}{H}{\rho}{s_0,-}&=\{\FormalSeries{h} \in \FormalModularFormSubset{k}{B}{H}{\rho}{s_0} \colon h_{\mathbf{j}}\neq 0 \implies 2 \nmid s(\mathbf{j})\}. \label{eq:FormalModularFormSubset2Minus}
\end{align}
It is immediate that $\FormalModularFormSubset{k}{B}{H}{\rho}{\mathbf{j_0}}=\FormalModularFormSubset{k}{B}{H}{\rho}{\mathbf{j_0},+}\oplus \FormalModularFormSubset{k}{B}{H}{\rho}{\mathbf{j_0},-}$ and $\FormalModularFormSubset{k}{B}{H}{\rho}{s_0}=\FormalModularFormSubset{k}{B}{H}{\rho}{s_0,+}\oplus \FormalModularFormSubset{k}{B}{H}{\rho}{s_0,-}$. We call \eqref{eq:FormalModularFormSubsetPlus} and \eqref{eq:FormalModularFormSubset2Plus} plus spaces, \eqref{eq:FormalModularFormSubsetMinus} and \eqref{eq:FormalModularFormSubset2Minus} minus spaces.

Now we investigate how the operators $\mathcal{T}^{-\mathbf{p}}$ and $\mathcal{L}_{k,B}^{\mathfrak{B}}$ act on above spaces.
\begin{prop}
\label{prop:TpActOnFormalModularSeries}
Use notations in Definition \ref{deff:modularFormalSeries}. Let $\mathbf{p}$ and $\mathbf{j_0}$ be two vectors in $\mulindZn$, and let $s_0$ be an integer. Then $\mathcal{T}^{-\mathbf{p}}$ maps $\FormalModularFormSet{k}{B}{H}{\rho}$, $\FormalModularFormSubset{k}{B}{H}{\rho}{\mathbf{j_0}}$, and $\FormalModularFormSubset{k}{B}{H}{\rho}{s_0}$, bijectively onto $\FormalModularFormSet{k+s(\mathbf{p})}{B}{H}{\rho}$, $\FormalModularFormSubset{k+s(\mathbf{p})}{B}{H}{\rho}{\mathbf{j_0}-\mathbf{p}}$, and $\FormalModularFormSubset{k+s(\mathbf{p})}{B}{H}{\rho}{s_0-s(\mathbf{p})}$ respectively. Moreover, if $2\mid s(\mathbf{p})$, then $\mathcal{T}^{-\mathbf{p}}$ maps plus (minus resp.) spaces bijectively onto corresponding plus (minus resp.) spaces. If $2\nmid s(\mathbf{p})$, then $\mathcal{T}^{-\mathbf{p}}$ maps plus (minus resp.) spaces bijectively onto corresponding minus (plus resp.) spaces.
\end{prop}
\begin{proof}
It's a consequence of Proposition \ref{prop:multiplyByTpOperatorAndSlashOperator}.
\end{proof}
\begin{prop}
\label{prop:LkBActOnFormalModularSeries}
Use notations in Definition \ref{deff:modularFormalSeries}. Let $\mathbf{j_0}$ be a vector in $\mulindZn$, and let $s_0$ be an integer. Then $\mathcal{L}_{k,B}^{\mathfrak{B}}$ maps $\FormalModularFormSet{k}{B}{H}{\rho}$, $\FormalModularFormSubset{k}{B}{H}{\rho}{\mathbf{j_0}}$, and $\FormalModularFormSubset{k}{B}{H}{\rho}{s_0}$, into themselves respectively.
\end{prop}
\begin{proof}
Using Proposition \ref{prop:LkBOperator} and the fact
\begin{equation*}
\mathcal{L}_{k,B}^{\mathfrak{B}}\left(\FormalSeries{\rho(\gamma)\circ h}\right)=\FormalSeries{\rho(\gamma)\circ g},
\end{equation*}
where
\begin{equation*}
\FormalSeries{g} = \mathcal{L}_{k,B}^{\mathfrak{B}}\left(\FormalSeries{h}\right).
\end{equation*}
\end{proof}
The above two propositions deal with general mapping properties of $\mathcal{T}^{-\mathbf{p}}$ and $\mathcal{L}_{k,B}^{\mathfrak{B}}$ acting on formal series which transform like modular forms, while the following is a special but crucial one for our purpose.
\begin{lemm}
\label{lemm:LkBMappingPropertyCrucial}
The operator $\mathcal{L}_{k,B}^{\mathfrak{B}}$ maps $\FormalModularFormSubset{k}{B}{H}{\rho}{0,+}$ into $\FormalModularFormSubset{k}{B}{H}{\rho}{2,+}$.
\end{lemm}
\begin{proof}
We have known from the last proposition that $\mathcal{L}_{k,B}^{\mathfrak{B}}$ maps $\FormalModularFormSubset{k}{B}{H}{\rho}{0}$ into $\FormalModularFormSubset{k}{B}{H}{\rho}{0}$. Furthermore, since $s(\mathbf{p}-\pi(E_{i,j}))=s(\mathbf{p})-2$ for any $\mathbf{p} \in \mulindZn$ and $1 \leq i,\,j \leq \dimn$, we have $\mathcal{L}_{k,B}^{\mathfrak{B}}$ maps $\FormalModularFormSubset{k}{B}{H}{\rho}{0,+}$ into $\FormalModularFormSubset{k}{B}{H}{\rho}{0,+}$. Finally, put $\mathcal{L}_{k,B}^{\mathfrak{B}}\FormalSeries{h}=\FormalSeriesi{g}{p}$, where $\FormalSeries{h} \in \FormalModularFormSubset{k}{B}{H}{\rho}{0,+}$. Then it follows from the expression of $g_\mathbf{p}$ in Definition \ref{deff:LkBOperator} that, $s(\mathbf{p})=0$ implies $g_\mathbf{p}=0$. Therefore, $\FormalSeriesi{g}{p} \in \FormalModularFormSubset{k}{B}{H}{\rho}{2,+}$, which concludes the proof.
\end{proof}
This is why we are interested in operators of the form $\mathcal{T}^{-\mathbf{p}}\circ\mathcal{L}_{k,B}^{\mathfrak{B}}$, with $s(\mathbf{p})=2$, as announced in Remark \ref{remm:compositionTpLkB}. By composing several such operators, we obtain certain weight-raising operators.
\begin{coro}
\label{coro:DkpMappingProperty}
Let $\lambda \in \numgeq{Z}{1}$. Let $\mathbf{p}_0,\dots,\mathbf{p}_{\lambda-1}$ be $\lambda$ vectors in $\mulindZn$ with $s(\mathbf{p}_i)=2$, and let $\mathbf{p}_{-1}$ be a vector in $\mulindZn$ such that $s(\mathbf{p}_{-1})=1$. Then the operator
\begin{equation}
\label{eq:opDpPlus}
\mathcal{T}^{-\mathbf{p}_{\lambda-1}}\circ\mathcal{L}_{k+2(\lambda-1)}\circ\dots\mathcal{T}^{-\mathbf{p}_1}\circ\mathcal{L}_{k+2}\circ\mathcal{T}^{-\mathbf{p}_0}\circ\mathcal{L}_{k}
\end{equation}
maps $\FormalModularFormSubset{k}{B}{H}{\rho}{0,+}$ into $\FormalModularFormSubset{k+2\lambda}{B}{H}{\rho}{0,+}$, and the operator
\begin{equation}
\label{eq:opDpMinus}
\mathcal{T}^{-\mathbf{p}_{\lambda-1}}\circ\mathcal{L}_{k+1+2(\lambda-1)}\circ\dots\mathcal{T}^{-\mathbf{p}_1}\circ\mathcal{L}_{k+1+2}\circ\mathcal{T}^{-\mathbf{p}_0}\circ\mathcal{L}_{k+1}\circ\mathcal{T}^{-\mathbf{p}_{-1}}
\end{equation}
maps $\FormalModularFormSubset{k}{B}{H}{\rho}{1,-}$ into $\FormalModularFormSubset{k+1+2\lambda}{B}{H}{\rho}{0,+}$.
\end{coro}
\begin{proof}
The first assertion follows from Proposition \ref{prop:TpActOnFormalModularSeries} and Lemma \ref{lemm:LkBMappingPropertyCrucial}. The second assertion follows from the first one and the fact that $\mathcal{T}^{-\mathbf{p}_{-1}}$ maps $\FormalModularFormSubset{k}{B}{H}{\rho}{1,-}$ onto $\FormalModularFormSubset{k+1}{B}{H}{\rho}{0,+}$.
\end{proof}

We want to find explicit descriptions of operators \eqref{eq:opDpPlus} and \eqref{eq:opDpMinus}. Luckily, there exists a simple expression.
%Recall a basic property of Euler $\Gamma$-function. For $a,\,b \in \numZ$ with $a \leq b+1$, and $z,\,w \in \numC$ with $w \neq 0$, we have
%\begin{equation}
%\label{eq:EulerGammaLemma}
%\prod_{\twoscript{a \leq k \leq b}{k \in \numZ}}(z+kw)=w^{b+1-a}\frac{\Gamma(z/w+b+1)}{\Gamma(z/w+a)},
%\end{equation}
%where an empty propduct is understood to be $1$.

\begin{prop}
\label{prop:expressionDkp}
Fix the bilinear form $B$ and the weight $k$ as in Convention \ref{conv1} and \ref{conv2}, and fix an ordered $\numR$-basis $\mathfrak{B}$ of $V$ as before. Suppose $\FormalSeries{h} \in \FormalSeriesSet$. Then the coefficient of $T^{\mathbf{j}}$-term of \eqref{eq:opDpPlus} acting on $\FormalSeries{h}$ is
\begin{multline*}
4^{\lambda}\lambda!\sum_{\Lambda \in \mymatset{\numgeq{Z}{0}}{\dimn}{\dimn}}\frac{(-\uppi\rmi G_{\mathfrak{B}})^{\Lambda}}{\Lambda!}\binom{s(\mathbf{j})/2+\lambda-s(\Lambda)}{\lambda-s(\Lambda)}\\
\cdot\frac{\Gamma(s(\mathbf{j})/2+k+2\lambda-1-s(\Lambda))}{\Gamma(s(\mathbf{j})/2+k+\lambda-1)}\dodth{s(\Lambda)}h_{\mathbf{j}+\mathbf{p}-\pi(\Lambda)},
\end{multline*}
where $\mathbf{p}=\sum_{0 \leq i < \lambda}\mathbf{p_i}$. On the other hand, the coefficient of $T^{\mathbf{j}}$-term of \eqref{eq:opDpMinus} acting on $\FormalSeries{h}$ is the same as the last expression, but with $\mathbf{p}=\sum_{-1 \leq i < \lambda}\mathbf{p_i}$ and the $\Gamma$-factor replaced by
\begin{equation*}
\frac{\Gamma(s(\mathbf{j})/2+k+2\lambda-s(\Lambda))}{\Gamma(s(\mathbf{j})/2+k+\lambda)}.
\end{equation*}
\end{prop}
\begin{proof}
The second assertion follows immediately from the first one, of which we give a sketch of proof. We use induction on $\lambda$. The case $\lambda=1$ is clear using definitions. Assume that the case $\lambda$ has been proved, and proceed to prove the case $\lambda+1$. By induction hypothesis, what we should prove is
\begin{multline*}
(s(\mathbf{j})+2)(s(\mathbf{j})+2k+4\lambda)4^{\lambda}\lambda!\sum_{\Lambda \in \mymatset{\numgeq{Z}{0}}{\dimn}{\dimn}}\frac{(-\uppi\sqrt{-1} G_{\mathfrak{B}})^{\Lambda}}{\Lambda!}\binom{s(\mathbf{j})/2+\lambda+1-s(\Lambda)}{\lambda-s(\Lambda)}\\
\cdot\frac{\Gamma(s(\mathbf{j})/2+k+2\lambda-s(\Lambda))}{\Gamma(s(\mathbf{j})/2+k+\lambda)}\dodth{s(\Lambda)}h_{\mathbf{j}+\mathbf{p}-\pi(\Lambda)} \\
-4\uppi\sqrt{-1}\sum_{1 \leq i,j \leq \dimn}a_{i,j}\cdot 4^{\lambda}\lambda!\sum_{\Lambda \in \mymatset{\numgeq{Z}{0}}{\dimn}{\dimn}}\frac{(-\uppi\sqrt{-1} G_{\mathfrak{B}})^{\Lambda}}{\Lambda!}\binom{s(\mathbf{j})/2+\lambda-s(\Lambda)}{\lambda-s(\Lambda)}\\
\cdot\frac{\Gamma(s(\mathbf{j})/2+k+2\lambda-1-s(\Lambda))}{\Gamma(s(\mathbf{j})/2+k+\lambda-1)}\dodth{s(\Lambda)+1}h_{\mathbf{j}+\mathbf{p}-\pi(\Lambda)-\pi(E_{i,j})} \\
=4^{\lambda+1}(\lambda+1)!\sum_{\Lambda \in \mymatset{\numgeq{Z}{0}}{\dimn}{\dimn}}\frac{(-\uppi\sqrt{-1} G_{\mathfrak{B}})^{\Lambda}}{\Lambda!}\binom{s(\mathbf{j})/2+\lambda+1-s(\Lambda)}{\lambda+1-s(\Lambda)}\\
\cdot\frac{\Gamma(s(\mathbf{j})/2+k+2\lambda+1-s(\Lambda))}{\Gamma(s(\mathbf{j})/2+k+\lambda)}\dodth{s(\Lambda)}h_{\mathbf{j}+\mathbf{p}-\pi(\Lambda)},
\end{multline*}
where $a_{i,j}$ is the $(i,j)$-entry of $G_{\mathfrak{B}}$. This could be proved by using a change of variable $\Lambda'=\Lambda+E_{i,j}$ in the second sum in the left-hand side, and then a straightforward calculation.
\end{proof}

It follows that the operator \eqref{eq:opDpPlus} and \eqref{eq:opDpMinus} depend only on the sum of all $\mathbf{p}_i$'s, not on any individual $\mathbf{p}_i$. This leads to the following definition.
\begin{deff}
\label{deff:Dkp}
Let $B$, $k$, and $\mathfrak{B}$ be as in Proposition \ref{prop:expressionDkp}. Suppose $\mathbf{p}$ is a vector in $\mulindZn$ with $s(\mathbf{p}) \geq 0$. We define an operator $\mathcal{D}_{k,B,\mathbf{p}}^{\mathfrak{B}}$, or simply $\mathcal{D}_{k,\mathbf{p}}$, on $\FormalSeriesSet$ as follows:
\begin{enumerate}
\item If $s(\mathbf{p})=0,\,1$, then  $\mathcal{D}_{k,\mathbf{p}}=\mathcal{T}^{-\mathbf{p}}$.
\item If $s(\mathbf{p}) \geq 2$ and $2 \mid s(\mathbf{p})$, then $\mathcal{D}_{k,\mathbf{p}}$ is defined to be \eqref{eq:opDpPlus} with $\mathbf{p}_0,\dots,\mathbf{p}_{\lambda-1}$ being vectors in $\mulindZn$ such that $s(\mathbf{p}_i)=2$ and $\mathbf{p}=\mathbf{p}_0+\dots+\mathbf{p}_{\lambda-1}$.
\item If $s(\mathbf{p}) \geq 2$ and $2 \nmid s(\mathbf{p})$, then $\mathcal{D}_{k,\mathbf{p}}$ is defined to be \eqref{eq:opDpMinus} with $\mathbf{p}_{-1},\mathbf{p}_{0},\dots,\mathbf{p}_{\lambda-1}$ being vectors in $\mulindZn$ such that $s(\mathbf{p}_{-1})=1$, $s(\mathbf{p}_i)=2$ if $i \geq 0$ and $\mathbf{p}=\mathbf{p}_{-1}+\mathbf{p}_{0}+\dots+\mathbf{p}_{\lambda-1}$.
\end{enumerate}
\end{deff}
\begin{rema}
By Proposition \ref{prop:expressionDkp}, the operator $\mathcal{D}_{k,\mathbf{p}}$ is well-defined. In fact, that proposition also gives explicit expression of each term of $\mathcal{D}_{k,\mathbf{p}}(\FormalSeries{h})$ when $s(\mathbf{p}) \geq 2$. One can see that, this explicit expression for $s(\mathbf{p}) \geq 2$ is also valid for $s(\mathbf{p})=0,\,1$ (in this case, $\lambda=0$). Finally note that, in the explicit expression, $\lambda$ always equals $[s(\mathbf{p})/2]$.
\end{rema}
We restate a useful proposition:
\begin{prop}
\label{prop:DkpMappingPlusMinus}
Let $H$ be a subgroup of $\slZ$, and $\rho\colon \widetilde{H}\rightarrow \GlW$ a representation. Let $k$, $B$ and $\mathfrak{B}$ be as in Proposition \ref{prop:expressionDkp}. Suppose $\mathbf{p} \in \mulindZn$ such that $s(\mathbf{p}) \geq 0$. If $2 \mid s(\mathbf{p})$, then we have a map
\begin{equation*}
\mathcal{D}_{k,\mathbf{p}}\colon \FormalModularFormSubset{k}{B}{H}{\rho}{0,+} \rightarrow \FormalModularFormSubset{k+s(\mathbf{p})}{B}{H}{\rho}{0,+}.
\end{equation*}
On the other hand, if $2 \nmid s(\mathbf{p})$, then we have a map
\begin{equation*}
\mathcal{D}_{k,\mathbf{p}}\colon \FormalModularFormSubset{k}{B}{H}{\rho}{1,-} \rightarrow \FormalModularFormSubset{k+s(\mathbf{p})}{B}{H}{\rho}{0,+}.
\end{equation*}
\end{prop}
\begin{proof}
The case $s(\mathbf{p})=0$ or $1$ is a special case of Proposition \ref{prop:TpActOnFormalModularSeries}, while the other case is just a restatement of Corollary \ref{coro:DkpMappingProperty}.
\end{proof}

We conclude this section by explaining why these operators are relevant to Taylor expansions of Jacobi forms of lattice index. The Taylor coefficients of a Jacobi form $\phi \in \JacFormHol{k}{L}{G}{\rho}$ with respect to some basis $\mathfrak{B}$ are just the coefficients $h_{\mathbf{j}}$'s of $\emb(\phi)$ (See Remark \ref{rema:relationshipJacFormalSeries}). Hence each coefficient of $\mathcal{D}_{k,\mathbf{p}}(\emb(\phi))$ is a finite linear combination of Taylor coefficients of $\phi$. In this way, we can regard each coefficient of $\mathcal{D}_{k,\mathbf{p}}(\emb(\phi))$ as certain kind of modified Taylor coefficient. Some of these modified Taylor coefficients are ordinary modular forms, which will be investigated in the next section.

\section{Connections with modular forms}
\label{sec:Connections with modular forms}
Recall some notations for modular forms. Let $H$ be a subgroup of $\slZ$ and $k$ be a half integer or integer. Let $\rho \colon \widetilde{H}\rightarrow\GlW$ be a group representation. Then the notation $\ModFormHol{k}{H}{\rho}$ denotes all holomorphic functions $h \in \holfnsm$ satisfying modular transformation equations on $\widetilde{H}$ with representation $\rho$ and of weight $k$. If $H$ is of finite index in $\slZ$, then the subspace of those functions meromorphic at all cusps of $H$ is denoted by $\ModFormWHol{k}{H}{\rho}$, and that holomorphic at all cusps is denoted by  $\ModForm{k}{H}{\rho}$. Moreover, $\ModFormCusp{k}{H}{\rho}$ denotes the space of $h \in \ModForm{k}{H}{\rho}$ which vanishes at all cusps of $H$. We omit the information $\myran$ in these notations as in Section \ref{sec:Jacobi forms of lattice index}, since it can be recovered from $\rho$.

\begin{lemm}
\label{lemm:takingOneCoefficientModularForm}
Let $k$, $B$, $\mathfrak{B}$, $H$ and $\rho$ be as in Definition \ref{deff:modularFormalSeries}. Let $\FormalSeries{h} \in \FormalModularFormSet{k}{B}{H}{\rho}$. For any coefficient $h_{\mathbf{j}}$, if there is no $h_{\mathbf{j}'} \neq 0$ with $\mathbf{j}'<\mathbf{j}$, then $h_{\mathbf{j}} \in \ModFormHol{k+s(\mathbf{j})}{H}{\rho}$. In particular, if $\FormalSeries{h} \in \FormalModularFormSubset{k}{B}{H}{\rho}{0}$, then for $\mathbf{j}$ with $s(\mathbf{j})=0$, we have $h_{\mathbf{j}} \in \ModFormHol{k}{H}{\rho}$.
\end{lemm}
\begin{proof}
This follow immediately from Definition \ref{deff:slashOperatorFormalSeries} and Definition \ref{deff:modularFormalSeries}.
\end{proof}

If $\FormalSeries{h} \in \FormalSeriesSet$, we call $h_{\mathbf{0}}$ the constant term. Combining the operation ``taking the constant term'' and $\mathcal{D}_{k,\mathbf{p}}$ introduced in the last section, we can construct higher weight modular forms.

\begin{deff}
\label{deff:DkpFormal}
Let $k$, $B$, $\mathfrak{B}$, $H$ and $\rho$ be as in Definition \ref{deff:modularFormalSeries}. Suppose $\mathbf{p} \in \mulindZn$ with $s(\mathbf{p}) \geq 0$. We define an operator
\begin{equation*}
D_{k,\mathbf{p}} \colon \FormalModularFormSubset{k}{B}{H}{\rho}{0} \rightarrow \ModFormHol{k+s(\mathbf{p})}{H}{\rho}
\end{equation*}
as follows. Let $\FormalSeries{h} \in \FormalModularFormSubset{k}{B}{H}{\rho}{0}$. If $2 \mid s(\mathbf{p})$, then $D_{k, \mathbf{p}}(\FormalSeries{h})$ is the constant term of $\mathcal{D}_{k, \mathbf{p}}(\sum_{2 \mid s(\mathbf{j})}{h}_{\mathbf{j}} T^\mathbf{j})$. On the other hand, if $2 \nmid s(\mathbf{p})$, then $D_{k, \mathbf{p}}(\FormalSeries{h})$ is the constant term of $\mathcal{D}_{k, \mathbf{p}}(\sum_{2 \nmid s(\mathbf{j})}{h}_{\mathbf{j}} T^\mathbf{j})$. If the dependency on $B$ and $\mathfrak{B}$ is important, we write $D_{k,B,\mathbf{p}}^{\mathfrak{B}}$ instead of $D_{k, \mathbf{p}}$.
\end{deff}
\begin{rema}
The assumption $\FormalSeries{h} \in \FormalModularFormSubset{k}{B}{H}{\rho}{0}$ implies that $\sum_{2 \mid s(\mathbf{j})}{h}_{\mathbf{j}} T^\mathbf{j} \in \FormalModularFormSubset{k}{B}{H}{\rho}{0,+}$ and $\sum_{2 \nmid s(\mathbf{j})}{h}_{\mathbf{j}} T^\mathbf{j} \in \FormalModularFormSubset{k}{B}{H}{\rho}{1,-}$. So the definition makes sense by Proposition \ref{prop:DkpMappingPlusMinus} and Lemma \ref{lemm:takingOneCoefficientModularForm}.
\end{rema}

From the explicit expression of $\mathcal{D}_{k, \mathbf{p}}$ (Proposition \ref{prop:expressionDkp}), it is easy to obtain an explicit expression of $D_{k, \mathbf{p}}$.
\begin{prop}
\label{prop:explicitExpressionDkp}
If $2 \mid s(\mathbf{p})$, then $D_{k, \mathbf{p}}(\FormalSeries{h})$ equals
\begin{equation}
\label{eq:DkpSpecialEven}
4^\lambda\lambda!\sum_{\twoscript{\Lambda \in \mymatset{\numgeq{Z}{0}}{\dimn}{\dimn}}{s(\Lambda) \leq \lambda}}\frac{(-\uppi\rmi G_{\mathfrak{B}})^{\Lambda}}{\Lambda!}\frac{\Gamma(k+2\lambda-1-s(\Lambda))}{\Gamma(k+\lambda-1)}\dodth{s(\Lambda)}h_{\mathbf{p}-\pi(\Lambda)}.
\end{equation}
On the other hand, if $2 \nmid s(\mathbf{p})$, then $D_{k, \mathbf{p}}(\FormalSeries{h})$ equals
\begin{equation}
\label{eq:DkpSpecialOdd}
4^\lambda\lambda!\sum_{\twoscript{\Lambda \in \mymatset{\numgeq{Z}{0}}{\dimn}{\dimn}}{s(\Lambda) \leq \lambda}}\frac{(-\uppi\rmi G_{\mathfrak{B}})^{\Lambda}}{\Lambda!}\frac{\Gamma(k+2\lambda-s(\Lambda))}{\Gamma(k+\lambda)}\dodth{s(\Lambda)}h_{\mathbf{p}-\pi(\Lambda)}.
\end{equation}
In each case, $\lambda=[s(\mathbf{p})/2]$.
\end{prop}
\begin{proof}
A special case of Proposition \ref{prop:expressionDkp} with $\mathbf{j}=\mathbf{0}$.
\end{proof}
\begin{rema}
Put $\lambda'=[(s(\mathbf{p})-1)/2]$. Then the Gamma factor in both cases can be written uniformly as
\begin{equation}
\frac{\Gamma(k+s(\mathbf{p})-1-s(\Lambda))}{\Gamma(k+\lambda')}=\prod_{\lambda' \leq s \leq s(\mathbf{p})-2-s(\Lambda)}(k+s).
\end{equation}
\end{rema}

We now apply what have been achieved to Jacobi forms of lattice index, answering the question how to combine Taylolr coefficients of such forms to construct ordinary modular forms. Recall notations in Convention \ref{conv1} and \ref{conv2}.
\begin{thm}
\label{thm:ModularFormFromJacobiForm}
Let $\mathfrak{B}=(e_1,\dots,e_\dimn)$ be a $\numR$-basis of $V$. Let $\phi \in \JacFormWHol{k}{L}{G}{\rho}$ and assume the Taylor expansion of $\phi(\tau,z_1e_1+\dots z_\dimn e_\dimn)$ around $z_1,\dots,z_\dimn=0$ is
\begin{equation*}
\sum_{\mathbf{j}\in \numgeq{Z}{0}^{\dimn}}{h}_{\mathbf{j}}(\tau) z_1^{j_1}\cdot\dots\cdot z_\dimn^{j_\dimn},\qquad \tau \in \uhp,\,z_1,\dots,z_\dimn \in \numC.
\end{equation*}
Suppose $\mathbf{p}\in\numgeq{Z}{0}^{\dimn}$. If $2 \mid s(\mathbf{p})$ (or $2 \nmid s(\mathbf{p})$ resp.), then the expression \eqref{eq:DkpSpecialEven} (the expression \eqref{eq:DkpSpecialOdd} resp.) gives a function in $\ModFormWHol{k+s(\mathbf{p})}{G}{\rho\vert_{\widetilde{G}}}$. If in addition, $\phi \in \JacFormWeak{k}{L}{G}{\rho}$, then the corresponding expressions are in $\ModForm{k+s(\mathbf{p})}{G}{\rho\vert_{\widetilde{G}}}$. Finally, if $\phi \in \JacForm{k}{L}{G}{\rho}$ and $\mathbf{p} \neq \mathbf{0}$, then the corresponding expressions are in $\ModFormCusp{k+s(\mathbf{p})}{G}{\rho\vert_{\widetilde{G}}}$.
\end{thm}
\begin{rema}
\label{rema:vanIttersum}
We must mention that, this theorem belongs to van Ittersum. See \cite[Corollary 2.44]{vI21}. He developped such a theorem for quasi-Jacobi forms in several elliptic variables. Our $D_{k, \mathbf{p}}(\FormalSeries{h})$ is the same as van Ittersum's $\xi_{\mathbf{p}}(\phi)$, up to a factor depending on $k$ and $\mathbf{p}$. The purpose that we state and reprove this theorem here is twofold. One reason is that our method of the proof is different from van Ittersum's, and the other is that this theorem is the beginning of our theory on Taylor expansions of Jacobi forms of lattice index.
\end{rema}
\begin{proof}
By abuse of language, we write $D_{k,\mathbf{p}}(\phi)$, instead of the more heavy but precise notation $D_{k,\mathbf{p}}(\emb(\phi))$. Then we should prove $D_{k,\mathbf{p}}(\phi) \in \ModFormWHol{k+s(\mathbf{p})}{G}{\rho\vert_{\widetilde{G}}}$. We have already known that $D_{k,\mathbf{p}}(\phi) \in \ModFormHol{k+s(\mathbf{p})}{G}{\rho\vert_{\widetilde{G}}}$ by definition. Hence, it remains to show that, for any $\gamma \in \sltZ$, the $q$-expansion of $D_{k,\mathbf{p}}(\phi)\vert_{k+s(\mathbf{p})}\gamma$ has only finitely many terms of negative power of $q$, where $\vert_{k+s(\mathbf{p})}\gamma$ is the usual slash operator on modular forms. Note that $D_{k,\mathbf{p}}(\phi)\vert_{k+s(\mathbf{p})}\gamma=D_{k,\mathbf{p}}(\phi\vert_{k,B}\gamma)$ by Proposition \ref{prop:diagramSlashOperator},  \ref{prop:multiplyByTpOperatorAndSlashOperator} and \ref{prop:LkBOperator}. By the definition of weakly holomorphic Jacobi forms (Definition \ref{deff:JacobiForm}), there exists an $n_0 \in \numZ$ such that \eqref{eq:ForuierExpansionOfJacobiLikeForm} holds with $c(n,t) \neq 0 \implies n \geq n_0$. By absolute convergence, we can rewrite the right-hand side of \eqref{eq:ForuierExpansionOfJacobiLikeForm} as a power series of $z_1,\dots,z_\dimn$ with coefficients $g_\mathbf{j}(\tau)$ whose $q$-expansion has no $q^n$-term with $n<n_0$. Hence, by Proposition \ref{prop:explicitExpressionDkp}, $D_{k,\mathbf{p}}(\phi\vert_{k,B}\gamma)$, which is equal to $D_{k,\mathbf{p}}$ acting on this power series, has no $q^n$-term with $n<n_0$ in its $q$-expansion. This proves the assertion on $\JacFormWHol{k}{L}{G}{\rho}$, while that on $\JacFormWeak{k}{L}{G}{\rho}$ and $\JacForm{k}{L}{G}{\rho}$ can be proved in a similar manner.
\end{proof}

Now we proceed to derive the Fourier development of $D_{k,\mathbf{p}}(\phi) := D_{k,\mathbf{p}}(\emb(\phi))$ for weakly holomorphic Jacobi form $\phi$.
\begin{lemm}
\label{lemm:fourierDevTaylorCoeff}
Use notations in Convention \ref{conv1} and \ref{conv2}. Let $\phi \in \JacFormWHol{k}{L}{G}{\rho}$ and $\gamma \in \Jacgrp{\slZ}{L}$. Suppose the Fourier expansion is
\begin{equation}
\label{eq:fourierDevTaylorCoeffCondition}
\phi \vert_{k,B}\gamma(\tau,z)= \sum_{n,\, t}c^\gamma(n,t) q^n \etp{B(t,z)}, \qquad c^\gamma(n,t) \in \myran,
\end{equation}
and the Taylor expansion with respect to some ordered $\numR$-basis $\mathfrak{B}=(e_1,\dots,e_\dimn)$ of $V$ is
\begin{equation*}
\phi\vert_{k,B}\gamma(\tau,z_1e_1+\dots +z_\dimn e_\dimn)=\sum_{\mathbf{p} \geq \mathbf{0}}h_{\mathbf{p}}^\gamma(\tau)z_1^{p_1}\dots z_\dimn^{p_\dimn}.
\end{equation*}
Then we have
\begin{equation}
\label{eq:fourierDevTaylorCoeff}
h_{\mathbf{p}}^\gamma(\tau)=\sum_{n}\left(\sum_{t}c^\gamma(n,t)\sum_{\twoscript{\Lambda \in \mymatset{\numgeq{Z}{0}}{\dimn}{\dimn}}{s(\Lambda_{\ast,j})=p_j}}\frac{(2\uppi\rmi G_{\mathfrak{B}})^\Lambda}{\Lambda!}t_1^{s(\Lambda_{1,\ast})}\dots t_\dimn^{s(\Lambda_{\dimn,\ast})}\right)q^n,
\end{equation}
where $t=t_1e_1+\dots +t_\dimn e_\dimn$ and $G_{\mathfrak{B}}$ is the Gram matrix of $B$ with respect to $\mathfrak{B}$.
\end{lemm}
Note that $t_1,\dots,t_\dimn$ in the inner sum may not be rationals, unless we choose $\mathfrak{B}$ to be a $\numZ$-basis of $L$.
\begin{proof}
Assume that the $(i,j)$-entry of $G_\mathfrak{B}$ is $a_{i,j}$. By a direct calculation, we obtain
\begin{equation*}
\etp{B(t,z)}=\sum_{\mathbf{p} \geq \mathbf{0}}\left(\sum_{\twoscript{\Lambda \in \mymatset{\numgeq{Z}{0}}{\dimn}{\dimn}}{s(\Lambda_{\ast,j})=p_j}}\frac{(2\uppi\rmi G_{\mathfrak{B}})^\Lambda}{\Lambda!}t_1^{s(\Lambda_{1,\ast})}\dots t_\dimn^{s(\Lambda_{\dimn,\ast})}\right)z_1^{p_1}\dots z_\dimn^{p_\dimn},
\end{equation*}
if we write $z=z_1e_1+\dots +z_\dimn e_\dimn$ and $t=t_1e_1+\dots +t_\dimn e_\dimn$. Inserting this into the Fourier expansion of $\phi \vert_{k,B}\gamma$ gives
\begin{equation*}
\phi \vert_{k,B}\gamma(\tau,z)=\sum_{n,t}\sum_{\Lambda \in \mymatset{\numgeq{Z}{0}}{\dimn}{\dimn}}c^\gamma(n,t)q^n\frac{(2\uppi\rmi G_{\mathfrak{B}})^\Lambda}{\Lambda!}t_1^{s(\Lambda_{1,\ast})}\dots t_\dimn^{s(\Lambda_{\dimn,\ast})}z_1^{s(\lambda_{\ast,1})}\dots z_\dimn^{s(\Lambda_{\ast,\dimn})}.
\end{equation*}
We shall prove the absolute convergence later. Hence by rearranging the summation order and using the uniqueness of coefficients of power series, we obtain the desired \eqref{eq:fourierDevTaylorCoeff}. To prove the absolute convergence, consider the absolute series
\begin{multline*}
\sum_{n,t}\sum_{\Lambda \in \mymatset{\numgeq{Z}{0}}{\dimn}{\dimn}}\abs{c^\gamma(n,t)}\rme^{-2\uppi n\Im\tau}\frac{\abs{(2\uppi G_{\mathfrak{B}})^\Lambda}}{\Lambda!}\abs{t_1}^{s(\Lambda_{1,\ast})}\dots \abs{t_\dimn}^{s(\Lambda_{\dimn,\ast})}\abs{z_1}^{s(\lambda_{\ast,1})}\dots \abs{z_\dimn}^{s(\Lambda_{\ast,\dimn})} \\
=\sum_{n,t}\abs{c^\gamma(n,t)}\rme^{-2\uppi n\Im\tau}\etp{B'(\abs{t_1}e_1+\dots+\abs{t_\dimn}e_\dimn,-\rmi(\abs{z_1}e_1+\dots+\abs{z_\dimn}e_\dimn))},
\end{multline*}
where $B'$ is the bilinear form on $\mathcal{V}$ such that the $(i,j)$-entry of the Gram matrix with respect to $\mathfrak{B}$ is $\abs{a_{i,j}}$. Put $t'=\abs{t_1}e_1+\dots+\abs{t_\dimn}e_\dimn$ and $z'=-\rmi(\abs{z_1}e_1+\dots+\abs{z_\dimn}e_\dimn)$. We shall find some $\tau_1 \in \uhp$ such that
\begin{equation}
\label{eq:toFindTau1}
\abs{c^\gamma(n,t)\rme^{2\uppi\rmi n\tau}\etp{B'(t',z')}} \leq \abs{c^\gamma(n,t)\rme^{2\uppi\rmi n\tau_1}\etp{B(t,z')}}
\end{equation}
for sufficiently large $n$ and arbitrary $t$. From this inequality, and the absolute convergence of \eqref{eq:fourierDevTaylorCoeffCondition} at $(\tau_1,z')$, the desired absolute convergence follows.

Recall a generalized Cauchy-Schwarz inequality. For any real bilinear form $F\colon V\times V \rightarrow\numR$ (whose Gram matrix with respect to the basis $\mathfrak{B}$ is $M$), and $v_1,\,v_2 \in V$, we have
\begin{equation*}
\abs{F(v_1,v_2)} \leq \lVert M\rVert\cdot \abs{v_1}\cdot\abs{v_2},
\end{equation*}
where $\abs{\cdot}$ in the right-hand side is any norm on $V$ and $\lVert M \rVert=\sup_{v \in V}\frac{\abs{Mv}}{\abs{v}}$. Thus
\begin{align*}
\abs{\etp{B'(t',z')-B(t,z')}}&=\rme^{-2\uppi(B'(t'-t,\Im z')+(B'-B)(t,\Im z'))}\\
&\leq \rme^{2\uppi(\lVert M_1\rVert\cdot\abs{t'-t}\cdot\abs{\Im z'}+\lVert M_2\rVert\cdot\abs{t}\cdot\abs{\Im z'})},
\end{align*}
where $M_1$ and $M_2$ are the Gram matrices of $B'$ and $B'-B$ with respect to $\mathfrak{B}$ respectively. The quantity $\abs{t}$, $\abs{t'}$, $\sqrt{Q(t)}$ are of the same order of magnitude when $t$ ranging over $V$. So according to the paragraph following Proposition \ref{prop:FourierExpasionCusps}, there exists some $C > 0$ such that $\abs{\etp{B'(t',z')-B(t,z')}} \leq \rme^{C\sqrt{n}}$ for sufficiently large $n$ and arbitrary $t$ satisfying $c^\gamma(n,t) \neq 0$. Thus we can choose any $\tau_1$ with $\Im \tau_1 < \Im \tau$ in \eqref{eq:toFindTau1}.
\end{proof}
\begin{rema}
\label{rema:hpAnotherExpression}
Put $\partial^\mathbf{p}=\frac{\partial^{s(\mathbf{p})}}{\partial z_1^{p_1}\dots\partial z_\dimn^{p_\dimn}}$. Then we can rewrite \eqref{eq:fourierDevTaylorCoeff} as
\begin{equation*}
h_{\mathbf{p}}^\gamma(\tau)=\frac{1}{\mathbf{p}!}\sum_{n}\left(\sum_{t}c^\gamma(n,t)\left(\partial^\mathbf{p}\etp{B(t,z)}\right)\vert_{z=0}\right)q^n.
\end{equation*}
Maybe a better formula is
\begin{equation*}
h_{\mathbf{p}}^\gamma(\tau)=\frac{(2\uppi\rmi)^{s(\mathbf{p})}}{\mathbf{p}!}\sum_{n}\left(\sum_{t}c^\gamma(n,t)\prod_{l=1}^\dimn B(t, e_l)^{p_l}\right)q^n.
\end{equation*}
\end{rema}

To obtain the Fourier development of $D_{k,\mathbf{p}}(\phi)$, we introduce certain polynomials.
%\begin{deff}
%\label{deff:PkpM}
%Let $\mathbf{p} \in \numgeq{Z}{0}^{\dimn}$, $k \in \halfint$ and $M \in \mymatset{\numR}{\dimn}{\dimn}$. Set $\lambda=[s(\mathbf{p})/2]$. If $2 \mid s(\mathbf{p})$, we define an $(\dimn+1)$-ary polynomial $P_{k,\mathbf{p},M}$ by the following fomula
%\begin{multline}
%P_{k,\mathbf{p},M}(X_0,X_1,\dots,X_\dimn)=\sum_{\twoscript{\mu_0,\dots,\mu_\dimn \in \numgeq{Z}{0}}{2\mu_0+\mu_1+\dots+\mu_\dimn=s(\mathbf{p})}}(-1)^{\mu_0}2^{s(\mathbf{p})-\mu_0}\\
%\times\frac{\Gamma(k+2\lambda-1-\mu_0)}{\Gamma(k+\lambda-1)}\sum_{\threescript{\Lambda,\Omega \in \mymatset{\numgeq{Z}{0}}{\dimn}{\dimn}}{s(\Lambda)=\mu_0,\,s(\Omega_{i,\ast})=\mu_i(1\leq i \leq \dimn)}{s(\Omega_{\ast,j})+\pi_j(\Lambda)=p_j(1\leq j \leq \dimn)}}\frac{M^{\Lambda+\Omega}}{\Lambda!\Omega!}X_0^{\mu_0}\cdots X_\dimn^{\mu_\dimn}.
%\end{multline}
%On the other hand, if $2 \nmid s(\mathbf{p})$, the polynomial $P_{k,\mathbf{p},M}$ is still defined by the above formula, but with the Gamma factor replaced by
%\begin{equation*}
%\frac{\Gamma(k+2\lambda-\mu_0)}{\Gamma(k+\lambda)}.
%\end{equation*}
%\end{deff}
\begin{deff}
\label{deff:PkpM}
Let $\mathbf{p} \in \numgeq{Z}{0}^{\dimn}$, $k \in \halfint$ and $M \in \mymatset{\numR}{\dimn}{\dimn}$. Set $\lambda'=[(s(\mathbf{p})-1)/2]$. We define an $(\dimn+1)$-ary polynomial $P_{k,\mathbf{p},M}$ by the following fomula
\begin{multline}
\label{eq:PkpM}
P_{k,\mathbf{p},M}(X_0,X_1,\dots,X_\dimn)=\sum_{\twoscript{\mu_0,\dots,\mu_\dimn \in \numgeq{Z}{0}}{2\mu_0+\mu_1+\dots+\mu_\dimn=s(\mathbf{p})}}(-1)^{\mu_0}2^{s(\mathbf{p})-\mu_0}\\
\times\frac{\Gamma(k+s(\mathbf{p})-1-\mu_0)}{\Gamma(k+\lambda')}\sum_{\threescript{\Lambda,\Omega \in \mymatset{\numgeq{Z}{0}}{\dimn}{\dimn}}{s(\Lambda)=\mu_0,\,s(\Omega_{i,\ast})=\mu_i(1\leq i \leq \dimn)}{s(\Omega_{\ast,j})+\pi_j(\Lambda)=p_j(1\leq j \leq \dimn)}}\frac{M^{\Lambda+\Omega}}{\Lambda!\Omega!}X_0^{\mu_0}\cdots X_\dimn^{\mu_\dimn}.
\end{multline}
\end{deff}
\begin{prop}
\label{prop:FourierCoeffDkp}
Use notations in Lemma \ref{lemm:fourierDevTaylorCoeff}, with $\gamma$ being an element in $\sltZ$ instead of $\Jacgrp{\slZ}{L}$. If $\mathbf{p} \in \numgeq{Z}{0}^\dimn$, then we have
\begin{equation*}
D_{k,\mathbf{p}}(\phi)\vert_{k+s(\mathbf{p})}\gamma=4^\lambda\cdot\lambda!(\uppi\rmi)^{s(\mathbf{p})}\sum_{n}\left(\sum_{t}c^\gamma(n,t)P_{k,\mathbf{p},G_{\mathfrak{B}}}(n,t_1,\dots,t_\dimn)\right)q^n,
\end{equation*}
where $\lambda=[s(\mathbf{p})/2]$ as usual.
\end{prop}
\begin{proof}
From the proof of Theorem \ref{thm:ModularFormFromJacobiForm}, we know that $D_{k,\mathbf{p}}(\phi)\vert_{k+s(\mathbf{p})}\gamma=D_{k,\mathbf{p}}(\phi\vert_{k,B}\gamma)$. The desired formula then follows from this fact, Proposition \ref{prop:explicitExpressionDkp} and Lemma \ref{lemm:fourierDevTaylorCoeff}.
\end{proof}

\begin{rema}
\label{rema:moreFormulaPkpM}
Although \eqref{eq:PkpM} gives the coefficient of each monomial of $P_{k, \mathbf{p}, M}$, it is computationally inefficient when $s(\mathbf{p})$ is large. In such situations, the following formula is better:
\begin{multline*}
P_{k,\mathbf{p},M}(X_0,X_1,\dots,X_\dimn)=2^{s(\mathbf{p})}\sum_{\mathbf{0} \leq \mathbf{q} \preceq \mathbf{p}}\frac{1}{\mathbf{q}!}\sum_{\twoscript{\Lambda \in \mymatset{\numgeq{Z}{0}}{\dimn}{\dimn}}{\pi(\Lambda)=\mathbf{p}-\mathbf{q}}}\frac{M^\Lambda}{\Lambda!}\\
\times\frac{\Gamma(k+s(\mathbf{p}+\mathbf{q})/2-1)}{\Gamma(k+\lambda')}\left(-\frac{1}{2}X_0\right)^{s(\mathbf{p}-\mathbf{q})/2}\cdot\left(X\cdot M\right)^{\mathbf{q}},
\end{multline*}
where $\mathbf{0} \leq \mathbf{q} \preceq \mathbf{p}$ means that $0 \leq q_j \leq p_j$ for $j=1,\dots,\dimn$ and $2 \mid s(\mathbf{p} - \mathbf{q})$, and $X$ denotes $(X_1,\dots,X_\dimn)$ regarding as a $1 \times \dimn$ matrix. By \eqref{eq:quadraticFormPower}, the sum $\sum_{\Lambda}\frac{M^\Lambda}{\Lambda!}$ in the above formula is equal to the coefficient of $T^{\mathbf{p}-\mathbf{q}}$-term of $\left((s(\mathbf{p}-\mathbf{q})/2)!\right)^{-1}\left(\sum_{i,j}m_{ij}T_iT_j\right)^{s(\mathbf{p}-\mathbf{q})/2}$, where $m_{ij}$ is the $(i,j)$-entry of $M$.
\end{rema}

\begin{examp}
We illustrate Theorem \ref{thm:ModularFormFromJacobiForm} and Proposition \ref{prop:FourierCoeffDkp} by a basic example. Assume that $\dimn=2$, $\mathcal{V}=\numC^2$ and $\myran=\numC$. Put $q=\etp{\tau}$, $\zeta_1=\etp{z_1}$ and $\zeta_2=\etp{z_2}$. Let $\phi$ be the following theta series
\begin{equation*}
\phi(\tau,z_1,z_2)=\sum_{t_1,t_2 \in 1/3+\numZ}q^{t_1^2+t_1t_2+t_2^2}\zeta_1^{2t_1+t_2}\zeta_2^{t_1+2t_2},
\end{equation*}
or equivalently,
\begin{equation*}
\phi(\tau,z_1,z_2)=\sum_{t \in (1/3,1/3)+\numZ^2}q^{Q(t)}\etp{B(t,z)},
\end{equation*}
where $B(t,z)=(t_1,t_2)\tbtmat{2}{1}{1}{2}\left(\begin{smallmatrix}{z_1} \\ {z_2}\end{smallmatrix}\right)$ and $Q(t)=\frac{1}{2}B(t,t)$. Put $\underline{L}=(\numZ^2,B)$, then the above function belongs to the space $\JacForm{1}{L}{\Gamma_1(3)}{\chi}$, where
\begin{equation*}
\Gamma_1(3)=\left\{\tbtmat{a}{b}{c}{d}\colon a \equiv d \equiv 1 \bmod 3,\quad c \equiv 0 \bmod 3\right\},
\end{equation*}
and $\chi$ is the character on $\Jacgrp{\Gamma_1(3)}{L}$ that maps $\eleactgrpnsimple{a}{b}{c}{d}{\varepsilon}{\lambda}{\mu}{\xi}$ to
\begin{equation*}
\xi\cdot(-1)^{(1-\sgn{d})\legendre{c}{-1}/2}(-1)^{B(\lambda,\mu)}\etp{\frac{b}{3}}\cdot\left(\frac{1}{\abs{d}}\sum_{v \in L/\abs{d}L}\etp{\frac{bQ(v)}{d}}\right).
\end{equation*}
Necessary background on theta series of lattice index, which includes the present example, will be given in Section \ref{sec:Application: Linear relations among theta series} (in particular, Proposition \ref{prop:thetaAlphaBetaJacobiForm}). By a direct calculation, we have
\begin{multline*}
P_{k,(2,1),\tbtmat{2}{1}{1}{2}}(X_0,X_1,X_2)\\
=32X_1^3+96X_1^2X_2+72X_1X_2^2+16X_2^3-24X_0X_1-24X_0X_2.
\end{multline*}
Applying Theorem \ref{thm:ModularFormFromJacobiForm} and Proposition \ref{prop:FourierCoeffDkp} to this setting gives the fact that, the function
\begin{equation*}
\sum_{t_1,t_2 \in 1/3+\numZ}(t_1^3+6t_1^2t_2+3t_1t_2^2-t_2^3)q^{t_1^2+t_1t_2+t_2^2}
\end{equation*}
belongs to the space $\ModFormCusp{4}{\Gamma_1(3)}{\chi\vert_{\widetilde{\Gamma_1(3)}}}$. To describe the character $\chi\vert_{\widetilde{\Gamma_1(3)}}$ of this space in another way, note that $\Gamma_1(3)$ is generated by two elements $\tbtmat{1}{1}{0}{1}$ and $\tbtmat{1}{-1}{3}{-2}$. So $\chi\vert_{\widetilde{\Gamma_1(3)}}$ is determined by the following two values:
\begin{equation*}
\chi\widetilde{\tbtmat{1}{1}{0}{1}}=\etp{\frac{1}{3}},\qquad \chi\widetilde{\tbtmat{1}{-1}{3}{-2}}=\etp{-\frac{1}{3}}.
\end{equation*}
\end{examp}

\begin{examp}
We also give a half-integral-weight example. Let $\eta(\tau)$ be the \emph{Dedekind eta function}, that is,
\begin{equation}
\eta(\tau)=q^{1/24}\prod_{n=1}^{\infty}(1-q^n).
\end{equation}
We have $\eta(\tau)^{-1}=q^{-1/24}\sum_{n=0}^{\infty}p(n)q^n$, where $p(n)$ is the \emph{patition function}, which counts the number of patitions of a positive integer $n$. (By convention, $p(0)=1$.) Set
\begin{equation*}
\phi(\tau,z_1,z_2)=\eta(\tau)^{-1}\sum_{t \in \numZ^2}q^{Q(t)}\etp{B(t,z)},
\end{equation*}
where $B$ and $Q$ are same as in the last example. Then $\phi \in \JacFormWHol{1/2}{L}{\Gamma_0(3)}{\chi}$. (See \eqref{eq:Gamma0N} for the definition of $\Gamma_0(3)$.) The character $\chi$, is defined on $\Jacgrp{\Gamma_0(3)}{L}$, and maps $\eleactgrpnsimple{a}{b}{c}{d}{\varepsilon}{\lambda}{\mu}{\xi}$ to
\begin{equation*}
\xi\cdot(-1)^{(1-\sgn{d})\legendre{c}{-1}/2}(-1)^{B(\lambda,\mu)}\cdot\left(\frac{1}{\abs{d}}\sum_{v \in L/\abs{d}L}\etp{\frac{bQ(v)}{d}}\right)\cdot\chi_\eta^{-1}\widetilde{\tbtmat{a}{b}{c}{d}},
\end{equation*}
where $\chi_\eta$ is the character of $\eta$ (see \cite[Section 2]{ZZ21}). We shall apply the operator $D_{1/2, (2,0)}$ to $\phi$. Note that by Definition \ref{deff:PkpM}, we have
\begin{equation*}
P_{1/2,(2,0),\tbtmat{2}{1}{1}{2}}(X_0,X_1,X_2)=-4X_0+4X_1^2+4X_1X_2+X_2^2.
\end{equation*}
Hence, according to Proposition \ref{prop:FourierCoeffDkp}, the function $D_{1/2, (2,0)}(\phi)$ is, up to a constant factor,
\begin{equation}
\label{eq:exampleDkpHalfintegralWeight}
q^{-1/24}\sum_{n\in\numgeq{Z}{0}}\left(\sum_{\twoscript{t_1,t_2\in\numZ}{t_1^2+t_1t_2+t_2^2 \leq n}}p(n-t_1^2-t_1t_2-t_2^2)\cdot(4t_1^2+4t_1t_2+t_2^2-4n+\frac{1}{6})\right)q^n.
\end{equation}
This function belongs to $\ModFormWHol{5/2}{\Gamma_0(3)}{\chi\vert_{\widetilde{\Gamma_0(3)}}}$ by Theorem \ref{thm:ModularFormFromJacobiForm}. The character $\chi\vert_{\widetilde{\Gamma_0(3)}}$ can be described by its action on generators of $\widetilde{\Gamma_0(3)}$ as follows:
\begin{equation*}
\chi\widetilde{\tbtmat{1}{1}{0}{1}}=\etp{-\frac{1}{24}},\quad \chi\widetilde{\tbtmat{-1}{0}{3}{-1}}=\etp{-\frac{3}{8}}, \quad \chi\widetilde{\tbtmat{-1}{0}{0}{-1}}=\etp{-\frac{1}{4}}.
\end{equation*}
\end{examp}

We can collect the operators in Definition \ref{deff:DkpFormal} together to form the following map:
\begin{align}
\label{eq:formalModFormintoModForms}
\prod_{\mathbf{p} \in \mathcal{P}} D_{k,\mathbf{p}}\colon \FormalModularFormSubset{k}{B}{H}{\rho}{0} &\rightarrow \prod_{\mathbf{p} \in \mathcal{P}}\ModFormHol{k+s(\mathbf{p})}{H}{\rho} \\
\notag\FormalSeries{h} &\mapsto \left\langle D_{k,\mathbf{p}}\FormalSeries{h}\middle\vert \mathbf{p} \in \mathcal{P}\right\rangle,
\end{align}
where $\mathcal{P}$ is any nomempty subset of the vectors in $\mulindZn$ with $s(\mathbf{p}) \geq 0$. It is clear a $\numC$-linear map. A natural question is, for which $\mathcal{P}$ this map is an isomorphism? Since we are mainly concerned with weakly holomorphic Jacobi forms of lattice index, we may restrict the domain to $\FormalModularFormSubset{k}{B}{H}{\rho}{\mathbf{0}}$ (a special case of \eqref{eq:FormalModularFormSubsetVector}, not to be confused with \eqref{eq:FormalModularFormSubsetScalar}).
\begin{prop}
\label{prop:isoFormalModForm}
Let $H$ be a subgroup of $\slZ$, $\rho\colon \widetilde{H}\rightarrow\GlW$ be a group representation. Let $B$ be a real symmetric bilinear form on $V$, $\mathfrak{B}$ be a $\numR$-basis of $V$, and $G_{\mathfrak{B}}$ be the Gram matrix of $B$ with respect to $\mathfrak{B}$. Let $k \in \halfint$ with $k \neq 0, -1, -2,\dots$. Then the map \eqref{eq:formalModFormintoModForms} with $\mathcal{P}=\numgeq{Z}{0}^\dimn$ restricting to $\FormalModularFormSubset{k}{B}{H}{\rho}{\mathbf{0}}$ is a $\numC$-linear isomorphism.
\end{prop}
This fact follows from Lemma \ref{lemm:sucDerModEqu} and the next lemma, which actually gives the inverse of the map considered. Put $\lambda_\mathbf{p}=[s(\mathbf{p})/2]$ and $\lambda'_\mathbf{p}=[(s(\mathbf{p})-1)/2]$.
\begin{lemm}
\label{lemm:inverseExpressionDkp}
Use notations and assumptions in Proposition \ref{prop:isoFormalModForm}. Let $\FormalSeries{h}$ and $\FormalSeries{g}$ be two formal series in $\FormalSeriesSetLbf{\mathbf{0}}$. Then the following two statements are equivalent:
\begin{enumerate}
    \item For any $\mathbf{p} \in \numgeq{Z}{0}^\dimn$, we have
    \begin{equation*}
    g_{\mathbf{p}}=\sum_{\twoscript{\Lambda \in \mymatset{\numgeq{Z}{0}}{\dimn}{\dimn}}{s(\Lambda) \leq \lambda_{\mathbf{p}}}}\frac{(-\uppi\rmi G_{\mathfrak{B}})^{\Lambda}}{\Lambda!}\frac{\Gamma(k+s(\mathbf{p})-1-s(\Lambda))}{\Gamma(k+\lambda'_{\mathbf{p}})}\dodth{s(\Lambda)}h_{\mathbf{p}-\pi(\Lambda)}.
    \end{equation*}
    \item For any $\mathbf{p} \in \numgeq{Z}{0}^\dimn$, we have
    \begin{equation*}
    h_\mathbf{p}=\sum_{\twoscript{\Lambda \in \mymatset{\numgeq{Z}{0}}{\dimn}{\dimn}}{s(\Lambda) \leq \lambda_{\mathbf{p}}}}\frac{(\uppi\rmi G_{\mathfrak{B}})^{\Lambda}}{\Lambda!}C_1^{-1}C_2^{-1}\dodth{s(\Lambda)}g_{\mathbf{p}-\pi(\Lambda)},
    \end{equation*}
    where $C_1$ and $C_2$ (depending on $k$, $s(\mathbf{p})$ and $s(\Lambda)$) are given by
    \begin{align*}
    C_1 &= \frac{\Gamma(k+s(\mathbf{p})-1-2s(\Lambda))}{\Gamma(k+\lambda'_{\mathbf{p}}-s(\Lambda))}=\prod_{n=\lambda'_{\mathbf{p}}-s(\Lambda)}^{s(\mathbf{p})-2-2s(\Lambda)}(k+n), \\
    C_2 &= \frac{\Gamma(k+s(\mathbf{p})-s(\Lambda))}{\Gamma(k+s(\mathbf{p})-2s(\Lambda))}=\prod_{n=s(\mathbf{p})-2s(\Lambda)}^{s(\mathbf{p})-1-s(\Lambda)}(k+n).
    \end{align*}
\end{enumerate}
\end{lemm}
The reason why we require $k \neq 0, -1, -2,\dots$ in Proposition \ref{prop:isoFormalModForm} is that we need $C_1$ and $C_2$ to be always nonzero. We emphasize that an empty product is defined to be $1$. When needed, we write $C_i(\mathbf{p}, \Lambda)$ instead of $C_i$ to indicate their dependence on $\mathbf{p}$ and $\Lambda$ for $i=1,2$.
\begin{proof}
A complicated, tedious, but straightforward calcultion.
\end{proof}
We now prove Proposition \ref{prop:isoFormalModForm}.
\begin{proof}
It is obviously the map is $\numC$-linear. To prove the injectiveness, we compute the kernel, which turns out to be zero by the ``(1)$\implies$(2)'' part of the above lemma. It remains to prove the surjectiveness. Let $\left\langle g'_\mathbf{p} \middle\vert \mathbf{p} \in \numgeq{Z}{0}^\dimn\right\rangle \in \prod_{\mathbf{p} \in \numgeq{Z}{0}^\dimn}\ModFormHol{k+s(\mathbf{p})}{H}{\rho}$ be arbitrary, and put $g'_\mathbf{p}=0$ if $\mathbf{p} \notin \numgeq{Z}{0}^\dimn$. Write $g_{\mathbf{p}}=4^{-\lambda_{\mathbf{p}}}(\lambda_{\mathbf{p}}!)^{-1}\cdot g'_{\mathbf{p}}$, and define a formal series $\FormalSeries{h} \in \FormalSeriesSetLbf{\mathbf{0}}$ by the equation in statement (2) of the above lemma. We shall prove two things: one is $\FormalSeries{h} \in \FormalModularFormSubset{k}{B}{H}{\rho}{\mathbf{0}}$, the other is $D_{k,\mathbf{p}}(\FormalSeries{h})=g'_{\mathbf{p}}$, from which the desired surjectiveness follows. The latter assertion is a direct consequence of the above lemma. To prove the former one, let $\gamma=\left(\tbtmat{a}{b}{c}{d},\varepsilon\right) \in \widetilde{H}$ be arbitrary. According to Definition \ref{deff:modularFormalSeries} and Proposition \ref{prop:slashOperatorFormalInverse}, we need to prove
\begin{equation}
\label{eq:toProveIsoFormalModForm}
h_{\mathbf{p}}\left(\frac{a\tau+b}{c\tau+d}\right)=\varepsilon^{2k}\sum_{\Lambda \in \mymatset{\numgeq{Z}{0}}{\dimn}{\dimn}}\frac{(\uppi\rmi cG_{\mathfrak{B}})^{\Lambda}}{\Lambda!}(c\tau+d)^{k+s(\mathbf{p})-s(\Lambda)}\rho(\gamma)\circ h_{\mathbf{p}-\pi(\Lambda)}(\tau)
\end{equation}
for any $\mathbf{p} \in \numgeq{Z}{0}^\dimn$. For this purpose, we shall use the following formula:
\begin{multline}
\label{eq:hModularIdentity}
\sum_{\Lambda \in \mymatset{\numgeq{Z}{0}}{\dimn}{\dimn}}\frac{(-\uppi\rmi G_{\mathfrak{B}})^{\Lambda}}{\Lambda!}C(\mathbf{p},\Lambda)\cdot\dodtha{s(\Lambda)}{h_{\mathbf{p}-\pi(\Lambda)}}\left(\frac{a\tau+b}{c\tau+d}\right)\\
=\varepsilon^{2k}(c\tau+d)^{k+s(\mathbf{p})}\cdot\sum_{\Lambda \in \mymatset{\numgeq{Z}{0}}{\dimn}{\dimn}}\frac{(-\uppi\rmi G_{\mathfrak{B}})^{\Lambda}}{\Lambda!}C(\mathbf{p},\Lambda)\cdot\dodtha{s(\Lambda)}{\rho(\gamma)\circ h_{\mathbf{p}-\pi(\Lambda)}}(\tau),
\end{multline}
where
\begin{equation*}
C(\mathbf{p},\Lambda) = \frac{\Gamma(k+s(\mathbf{p})-1-s(\Lambda))}{\Gamma(k+\lambda'_{\mathbf{p}})}.
\end{equation*}
This is obtained by using the fact $g_{\mathbf{p}} \in \ModFormHol{k+s(\mathbf{p})}{H}{\rho}$ and the statement (1) in Lemma \ref{lemm:inverseExpressionDkp}. Now we use induction on $s(\mathbf{p})$. For the base case $s(\mathbf{p})=0$ or $1$, the function $h_{\mathbf{p}}$ is just $g_{\mathbf{p}}$, so \eqref{eq:toProveIsoFormalModForm} holds. For the induction step, let $\mathbf{p} \in \numgeq{Z}{0}^\dimn$ be arbitrary, and assume that \eqref{eq:toProveIsoFormalModForm}  has been proved for any $\mathbf{p}' \in \numgeq{Z}{0}^\dimn$ with $s(\mathbf{p}') < s(\mathbf{p})$. Inserting the induction hypothesis into \eqref{eq:hModularIdentity}, and using Lemma \ref{lemm:sucDerModEqu}, we obtain, after a tedious simplification, that
\begin{multline*}
h_{\mathbf{p}}\left(\frac{a\tau+b}{c\tau+d}\right)=\varepsilon^{2k}\sum_{\Lambda \in \mymatset{\numgeq{Z}{0}}{\dimn}{\dimn}}\frac{(\uppi\rmi cG_{\mathfrak{B}})^{\Lambda}}{\Lambda!}(c\tau+d)^{k+s(\mathbf{p})-s(\Lambda)}\rho(\gamma)\circ h_{\mathbf{p}-\pi(\Lambda)}(\tau) \\
-\varepsilon^{2k}\sum_{l \in \numgeq{Z}{0}}\sum_{\twoscript{\Lambda \in \mymatset{\numgeq{Z}{0}}{\dimn}{\dimn}}{s(\Lambda)>l}}\left(\sum_{\twoscript{\Lambda_1,\,\Lambda_2 \in \mymatset{\numgeq{Z}{0}}{\dimn}{\dimn}}{\Lambda_1+\Lambda_2=\Lambda}}T_{\Lambda_1,\Lambda_2,l}\right),
\end{multline*}
where\footnote{The binomial coefficient $\binom{a}{b}$ is defined to be $a(a-1)\dots (a-b+1)/b!$ for $a \in \numC$ and $b \in \numgeq{Z}{0}$. Thus, when $a \in \numgeq{Z}{0}$ but $a<b$, we have $\binom{a}{b}=0$.}
\begin{multline*}
T_{\Lambda_1,\Lambda_2,l}=\frac{(\uppi\rmi G_{\mathfrak{B}})^{\Lambda_1+\Lambda_2}}{\Lambda_1!\Lambda_2!}(-1)^{s(\Lambda_1)}c^{s(\Lambda_1+\Lambda_2)-l}(c\tau+d)^{k+s(\mathbf{p})-s(\Lambda_1+\Lambda_2)+l} \\
\times \binom{s(\Lambda_1)}{l}\frac{\Gamma(k+s(\mathbf{p})-1-s(\Lambda_1))}{\Gamma(k+s(\mathbf{p})-1)}\frac{\Gamma(k+s(\mathbf{p})-s(\Lambda_1+\Lambda_2))}{\Gamma(k+s(\mathbf{p})-2s(\Lambda_1)-s(\Lambda_2)+l)} \\
\times \dodth{l}{\rho(\gamma)\circ h_{\mathbf{p}-\pi(\Lambda_1+\Lambda_2)}}(\tau).
\end{multline*}
Thus, it suffices to show that, for any non-negative integer $l$ and $\Lambda \in \mymatset{\numgeq{Z}{0}}{\dimn}{\dimn}$ with $s(\Lambda)>l$, we have
\begin{equation*}
\sum_{\Lambda_1+\Lambda_2=\Lambda}T_{\Lambda_1,\Lambda_2,l}=0,
\end{equation*}
or equivalently,
\begin{multline}
\label{eq:toProveSumLambda}
\sum_{\twoscript{\Lambda_1+\Lambda_2=\Lambda}{s(\Lambda_1) \geq l}}\frac{(-1)^{s(\Lambda_1)}}{\Lambda_1!\Lambda_2!}\binom{s(\Lambda_1)}{l}\frac{\Gamma(k+s(\mathbf{p})-1-s(\Lambda_1))}{\Gamma(k+s(\mathbf{p})-1)} \\
\times \frac{\Gamma(k+s(\mathbf{p})-s(\Lambda))}{\Gamma(k+s(\mathbf{p})-s(\Lambda)-s(\Lambda_1)+l)}=0.
\end{multline}
By \eqref{eq:binomMatrixToprove}, the desired identity \eqref{eq:toProveSumLambda} is a consequence of the following elementary identity
\begin{equation*}
\sum_{u=l}^{s(\Lambda)}\frac{(-1)^u}{(s(\Lambda)-u)!(u-l)!}\frac{\Gamma(k+s(\mathbf{p})-1-u)}{\Gamma(k+s(\mathbf{p})-(s(\Lambda)-l)-u)}=0,
\end{equation*}
which can be proved by induction on $s(\Lambda)-l$. This concludes the induction step, hence the whole proof.
\end{proof}

\section{Jacobi theta series and theta decompositions}
\label{sec:Jacobi theta series and theta decompositions}
The main aim of this paper is to find \emph{finite} subsets $\mathcal{P}_0$ of $\mathcal{P}=\numgeq{Z}{0}^\dimn$, such that the map $\prod_{\mathbf{p} \in \mathcal{P}_0} D_{k,\mathbf{p}}$ restricted to spaces of weakly holomorphic Jacobi forms (see \eqref{eq:formalModFormintoModForms} for definition) is still injective. It turns out that another important decomposition techineque, that of \emph{theta decomposition}, is needed in the proof of this main result (Theorem \ref{thm:mainEmbed}, which will be carried out in the next section). In this section, we review some necessary facts of this techineque.

First recall some facts on \emph{Jacobi theta series} and \emph{Weil representations}. The reader may consult \cite[\S 2.3]{Ajo15}, \cite[\S 2, \S 3.5, \S 4.2]{Boy15}, \cite[\S 5]{Str13} for more details.

In this section, unless declare explicitly, we usually assume that $L$ is an even lattice, which could simplify the discussion. Then $L^\sharp/L$ can be equipped with a finite quadratic module structure. See \cite[\S 2]{Str13}. We can also form the group algebra $\numC[L^\sharp/L]$, which possesses a Hilbert space structure we now explain. For $x \in L^\sharp/L$, let $\delta_x$ denote the element in $\numC[L^\sharp/L]$ such that $\delta_x(y)=1$ if $x=y$ and $\delta_x(y)=0$ otherwise. Then $\delta_x$'s ($x \in L^\sharp/L$) form a $\numC$-basis of $\numC[L^\sharp/L]$. The Hermite form
\begin{equation*}
(\sum a_x\delta_x,\sum b_x\delta_x) \mapsto \sum a_x\overline{b_x}
\end{equation*}
makes $\numC[L^\sharp/L]$ become a Hilbert space. We define a representation
\begin{equation}
\label{eq:WeilRep}
\rho_{\underline{L}} \colon \Jacgrp{\slZ}{L} \rightarrow \mathrm{GL}(\numC[L^\sharp/L])
\end{equation}
by the following three formulae:
\begin{align*}
\rho_{\underline{L}}\widetilde{\tbtmat{1}{1}{0}{1}}\delta_x &= \etp{Q(x)}\delta_x, \\
\rho_{\underline{L}}\widetilde{\tbtmat{0}{-1}{1}{0}}\delta_x &= \frac{(-\rmi)^{\dimn/2}}{\sqrt{\abs{L^\sharp/L}}}\sum_{y \in L^\sharp/L}\etp{-B(x,y)}\delta_y, \\
\rho_{\underline{L}}([v,w],\xi)\delta_x &= \xi\etp{\frac{1}{2}B(v,w)}\delta_x.
\end{align*}
Since $\sltZ$ has a presentation with generators $\widetilde{T}=\widetilde{\tbtmat{1}{1}{0}{1}}$, $\widetilde{S}=\widetilde{\tbtmat{0}{-1}{1}{0}}$, and relations $\widetilde{S}^8=\widetilde{I}$, $\widetilde{S}^2=(\widetilde{S}\widetilde{T})^3$, the above three formulae really extend to a representation of $\Jacgrp{\slZ}{L}$. Using these three formulae, one can verify that $\rho_{\underline{L}}$ is a unitary representation, which means $\rho_{\underline{L}}(\gamma)$ preserves the inner product on $\numC[L^\sharp/L]$ for any $\gamma \in \Jacgrp{\slZ}{L}$. One can also verify that the kernel is of finite index. Note that $\rho_{\underline{L}}$, when restricted to $\sltZ$, is the well-known Weil representation. Str\"omberg \cite{Str13} has worked out an explicit formula for any $\rho_{\underline{L}}(\gamma)$.

Occasionally, we shall use its dual representation (when dealing with theta decompositions)
\begin{align}
\rho_{\underline{L}}^{\ast} \colon \Jacgrp{\slZ}{L} &\rightarrow \mathrm{GL}(\numC[L^\sharp/L]^\ast)\\
\notag \gamma &\mapsto (\rho_{\underline{L}}(\gamma)^\times)^{-1},
\end{align}
where $\numC[L^\sharp/L]^\ast$ is the dual space (all linear functionals), and $\rho_{\underline{L}}(\gamma)^\times$ means the operator adjoint of $\rho_{\underline{L}}(\gamma)$ that maps $f \in \numC[L^\sharp/L]^\ast$ to $f\circ\rho_{\underline{L}}(\gamma)$. We remark that the matrix of $\rho_{\underline{L}}^{\ast}(\gamma)$ with respect to the dual basis is the transposed inverse of the matrix of $\rho_{\underline{L}}(\gamma)$.

The Jacobi theta series of lattice index is defined by
\begin{equation}
\label{eq:deffJacobiThetaLatticeIndex}
\vartheta_{\underline{L}, t}(\tau,z)=\sum_{v \in t+L}\etp{\tau Q(v)+B(v,z)}
\end{equation}
as a scalar-valued function on $\mydom$, where $t \in L^\sharp/L$. Note that $L$ can also be an odd lattice. The reader may verify that the above series converges normally, hence uniformly and absolutely on any compact subsets of $\mydom$. Put
\begin{equation}
\label{eq:deffJacobiThetaLatticeIndex2}
\vartheta_{\underline{L}}=\sum_{t \in L^\sharp/L}\vartheta_{\underline{L}, t}\delta_t,
\end{equation}
which is a map from $\mydom$ to $\numC[L^\sharp/L]$. It is known that
\begin{thm}
\label{thm:thetaSeriesLatticeJacobiForm}
The map $\vartheta_{\underline{L}}$ belongs to $\JacForm{\dimn/2}{L}{\slZ}{\rho_{\underline{L}}}$.
\end{thm}
Equivalently, this theorem says that
\begin{align}
\vartheta_{\underline{L}, t}\vert_{\dimn/2,B}\widetilde{\tbtmat{1}{1}{0}{1}} &= \etp{Q(t)}\vartheta_{\underline{L}, t},\label{eq:thetaTransformationT}\\
\vartheta_{\underline{L}, t}\vert_{\dimn/2,B}\widetilde{\tbtmat{0}{-1}{1}{0}} &= \frac{(-\rmi)^{\dimn/2}}{\sqrt{\abs{L^\sharp/L}}}\sum_{y \in L^\sharp/L}\etp{-B(t,y)}\vartheta_{\underline{L}, y},\label{eq:thetaTransformationS}\\
\vartheta_{\underline{L}, t}\vert_{\dimn/2,B}([v,w],\xi) &= \xi\etp{\frac{1}{2}B(v,w)}\vartheta_{\underline{L}, t}.\label{eq:thetaTransformationL}
\end{align}
For a proof, see \cite[\S 3.5]{Boy15}. We emphasize that this theorem, more precisely, the transformation law \eqref{eq:thetaTransformationT}, requires $\underline{L}$ to be even.

Second, recall the concept of the \emph{tensor product} of two group representations. Let $\rho_1$ and $\rho_2$ be two group representations on the same abstract group $G$ with representation $\numC$-spaces $\mathcal{W}_1$ and $\mathcal{W}_2$ respectively. Then the tensor product $\rho_1 \otimes \rho_2$, is the representation $G \rightarrow \mathrm{GL}(\mathcal{W}_1 \otimes_{\numC} \mathcal{W}_2)$ defined by $\rho_1 \otimes \rho_2(g)(w_1 \otimes w_2)=\rho_1(g)(w_1) \otimes \rho_2(g)(w_2)$.

Now we review the theory of theta decompositions. For the classical theory, see \cite[Theorem 5.1]{EZ85}; for the lattice-index and scalar-valued case, see \cite[Theorem 1]{Kri96}; for the lattice-index and vector-valued case, see \cite[Prososition 7]{Wil19}. The following version is weaker than Williams' in some aspects, but suits us.
\begin{thm}
\label{thm:thetaDecomposition}
Use notations in Convention \ref{conv1} and \ref{conv2}. The lattice $\underline{L}$ may be even or odd. Suppose that $\rho$ satisfies
\begin{equation}
\label{eq:conditionOnRhoThetaDecomp}
\rho([av,bv],1)=\mathrm{id}_\myran
\end{equation} 
for any coprime integers $a,b$ and $v \in L$. Then for any $\gamma \in \sltZ$ and $\phi \in \JacFormWHol{k}{L}{G}{\rho}$, we have
\begin{equation}
\label{eq:thetaDecomposition}
\phi\vert_{k,B}\gamma=\sum_{t \in L^\sharp/L}h_{t}^{\gamma}\vartheta_{\underline{L}, t},
\end{equation}
where $h_{t}^{\gamma}$ is the holomorphic function on $\uhp$ defined by
\begin{equation}
\label{eq:deffhtgamma}
h_{t}^{\gamma}(\tau)=\sum_{n}c^\gamma(n+Q(t), t)q^n,
\end{equation}
and the quantity $c^\gamma(n, t)$ refers to $c(n, t)$ in \eqref{eq:ForuierExpansionOfJacobiLikeForm}. Moreover, if $\underline{L}$ is even, then the map that sends $\phi=\sum_{n,t}c(n,t)q^n\etp{B(t,z)}$ to
\begin{align}
\label{eq:deffh}
h \colon \uhp &\rightarrow \myran \otimes \numC[L^\sharp/L]^\ast \\
\tau &\mapsto \sum_{n}\left(\sum_{t \in L^\sharp/L}c(n+Q(t),t)\otimes \delta_t^\ast\right)q^n \notag
\end{align}
is a $\numC$-linear embedding from $\JacFormWHol{k}{L}{G}{\rho}$ into $\ModFormWHol{k-\frac{1}{2}\dimn}{G}{\rho\otimes\rho_{\underline{L}}^\ast\vert_{\widetilde{G}}}$. In addition, when restricted, this map gives embeddings from $\JacForm{k}{L}{G}{\rho}$ and $\JacFormCusp{k}{L}{G}{\rho}$ into $\ModForm{k-\frac{1}{2}\dimn}{G}{\rho\otimes\rho_{\underline{L}}^\ast\vert_{\widetilde{G}}}$ and $\ModFormCusp{k-\frac{1}{2}\dimn}{G}{\rho\otimes\rho_{\underline{L}}^\ast\vert_{\widetilde{G}}}$ respectively.
\end{thm}
\begin{proof}
Fix $\phi$ and $\gamma$; we proceed to prove \eqref{eq:thetaDecomposition}. A crucial fact is
\begin{equation*}
\rho([v,0]\gamma^{-1},1)\circ c^\gamma(n+Q(v)+B(t,v), t+v)=c^\gamma(n,t)
\end{equation*}
for any $(n,t) \in \numR \times V$, $v \in L$ and $\gamma \in \sltZ$. See the paragraph immediately after Proposition \ref{prop:FourierExpasionCusps}. Now by \eqref{eq:conditionOnRhoThetaDecomp}, we have $c^\gamma(n+Q(v)+B(t,v), t+v)=c^\gamma(n,t)$. Proposition \ref{prop:ForuierExpansionOfJacobiLikeForm} tells us that $c^\gamma(n,t) \neq 0$ implies $t \in \frac{1}{N_1N_2}L^\sharp$. But under the assumption \eqref{eq:conditionOnRhoThetaDecomp} (put $a=0$ and $b=1$), we actually have $c^\gamma(n,t) \neq 0 \implies t \in L^\sharp$. Taking into account these two facts, and using normal convergence we can rewrite \eqref{eq:ForuierExpansionOfJacobiLikeForm} as
\begin{align*}
\phi \vert_{k,B}\gamma(\tau,z) &= \sum_{n \in (N_1N_2)^{-1}\numZ,\, t \in L^\sharp}c^\gamma(n,t) q^n \etp{B(t,z)}\\
 &=\sum_{\twoscript{n_0 \in (N_1N_2)^{-1}\numZ}{t_0+L \in L^\sharp/L}}\sum_{v \in L} c^\gamma(n_0+Q(t_0+v),t_0+v)q^{n_0+Q(t_0+v)}\etp{B(t_0+v,z)}\\
 &=\sum_{n_0,\,t_0}c^\gamma(n_0+Q(t_0),t_0)q^{n_0}\sum_{v \in L}q^{Q(t_0+v)}\etp{B(t_0+v,z)} \\
 &= \sum_{t_0 \in L^\sharp/L}h_{t_0}^\gamma\vartheta_{\underline{L}, t_0},
\end{align*}
from which \eqref{eq:thetaDecomposition} follows. Also, from the above deduction, we see that \eqref{eq:deffhtgamma} and \eqref{eq:deffh} are well-defined, that is, they are independent of the choice of the representative of $t \in L^\sharp/L$, and converge normally, hence are holomorphic on $\uhp$.

Now assume that $L$ is even, and we proceed to prove that the function $h$ given by \eqref{eq:deffh} belongs to $\ModFormWHol{k-\frac{1}{2}\dimn}{G}{\rho\otimes\rho_{\underline{L}}^\ast\vert_{\widetilde{G}}}$. Note that \eqref{eq:deffh} can be rewritten as $h(\tau)=\sum_{t \in L^\sharp/L}h_t(\tau)\otimes \delta_t^\ast$, where $h_t=h_t^{\widetilde{I}}$ with $I=\tbtmat{1}{0}{0}{1}$. So we first investigate transformation equations of $h_t$. The notation $\sigma_{t,t'}$ is defined by
\begin{align*}
\sigma(\delta_{t'})&=\sum_{t \in L^\sharp/L}\sigma_{t,t'}\delta_{t},\quad t' \in L^\sharp/L,\, \sigma \in \mathrm{GL}(\numC[L^\sharp/L]),\\
\sigma(\delta_{t'}^\ast)&=\sum_{t \in L^\sharp/L}\sigma_{t,t'}\delta_{t}^\ast,\quad t' \in L^\sharp/L,\, \sigma \in \mathrm{GL}(\numC[L^\sharp/L]^\ast).
\end{align*}
Use this notation, Theorem \ref{thm:thetaSeriesLatticeJacobiForm} is equivalent to
\begin{equation}
\label{eq:thetaTransformationAnother}
\vartheta_{\underline{L}, t}\vert_{\frac{1}{2}\dimn, B}\gamma=\sum_{t' \in L^\sharp/L}\rho_{\underline{L}}(\gamma)_{t,t'}\vartheta_{\underline{L}, t'},\quad \gamma \in \Jacgrp{\slZ}{L}.
\end{equation}
Put $\gamma=\eleactgrpnsimple{a}{b}{c}{d}{\varepsilon}{v}{w}{\xi}$. Then by \eqref{eq:thetaTransformationAnother} we have
\begin{align}
\phi\vert_{k,B}\gamma &= \sum_{t \in L^\sharp/L}h_t\vert_{k-\frac{1}{2}\dimn}\eleglptRsimple{a}{b}{c}{d}{\varepsilon}\vartheta_{\underline{L}, t}\vert_{\frac{1}{2}\dimn, B}\gamma \notag\\
&=\sum_{t' \in L^\sharp/L}\left(\sum_{t \in L^\sharp/L}\rho_{\underline{L}}(\gamma)_{t,t'}h_t\vert_{k-\frac{1}{2}\dimn}\eleglptRsimple{a}{b}{c}{d}{\varepsilon}\right)\vartheta_{\underline{L}, t'} \label{eq:phigammaExpansion}
\end{align}
If $\gamma \in \Jacgrp{G}{L}$, then by the transofmation laws of $\phi$ (see Definition \ref{deff:JacobiLikeForm}), and the linear independence\footnote{In fact, for any fixed $\tau \in \uhp$, the functions $z \mapsto \vartheta_{\underline{L}, t}(\tau,z),\, t \in L^\sharp/L$ are $\numC$-linear independent. This follows from the uniqueness of coefficients of Fourier series.} of the family $\vartheta_{\underline{L}, t},\, t\in L^\sharp/L$, we have
\begin{equation*}
\rho(\gamma)\circ h_{t'}=\sum_{t \in L^\sharp/L}\rho_{\underline{L}}(\gamma)_{t,t'}h_{t}\vert_{k-\frac{1}{2}\dimn}\eleglptRsimple{a}{b}{c}{d}{\varepsilon},
\end{equation*}
which is equivalent to
\begin{equation*}
h_{t}\vert_{k-\frac{1}{2}\dimn}\eleglptRsimple{a}{b}{c}{d}{\varepsilon}=\sum_{t' \in L^\sharp/L}\rho_{\underline{L}}(\gamma)^{-1}_{t',t}\cdot \rho(\gamma)\circ h_{t'}.
\end{equation*}
We have obtained transformation equations of $h_t$. Next we derive transformation equations of $h$ from those of $h_t$. By the definition of dual representations, we have $\rho_{\underline{L}}(\gamma)^{-1}_{t',t}=\rho_{\underline{L}}^\ast(\gamma)_{t,t'}$. Hence
\begin{align}
\label{eq:hslashgammaTransformation}
h \vert_{k-\frac{1}{2}\dimn}\eleglptRsimple{a}{b}{c}{d}{\varepsilon}(\tau) &= \sum_{t' \in L^\sharp/L}\rho(\gamma)(h_{t'}(\tau))\otimes \sum_{t \in L^\sharp/L}\rho_{\underline{L}}^\ast(\gamma)_{t,t'}\delta_{t}^\ast \\
&=\sum_{t' \in L^\sharp/L}\rho(\gamma)(h_{t'}(\tau))\otimes \rho_{\underline{L}}^\ast(\gamma)\delta_{t'}^\ast \notag\\
&=\left(\rho \otimes \rho_{\underline{L}}^\ast(\gamma)\right)\circ h(\tau),\notag
\end{align}
from which the transformation laws of $h$ follow. To conclude that $h \in \ModFormWHol{k-\frac{1}{2}\dimn}{G}{\rho\otimes\rho_{\underline{L}}^\ast\vert_{\widetilde{G}}}$, it remains to show that, for any $\gamma=\eleglptRsimple{a}{b}{c}{d}{\varepsilon}\in \sltZ$, the $q$-expansion of $h\vert_{k-\frac{1}{2}\dimn}\gamma$ has at most finitely many terms of negative powers of $q$. Comparing \eqref{eq:thetaDecomposition} and \eqref{eq:phigammaExpansion} and using the linear independence of the functions $z \mapsto \vartheta_{\underline{L}, t}(\tau,z),\, t \in L^\sharp/L$, we obtain that
\begin{equation}
\label{eq:htslashgamma}
h_t\vert_{k-\frac{1}{2}\dimn}\gamma=\sum_{t' \in L^\sharp/L}\rho_{\underline{L}}(\gamma)^{-1}_{t',t}h_{t'}^\gamma.
\end{equation}
The desired property follows from this formula, \eqref{eq:deffhtgamma} and the conditions $c^\gamma(n,t)$ satisfy (see Definition \ref{deff:JacobiForm}).

We have shown that the map that sends $\phi$ to $h$ is a well-defined map from $\JacFormWHol{k}{L}{G}{\rho}$ to $\ModFormWHol{k-\frac{1}{2}\dimn}{G}{\rho\otimes\rho_{\underline{L}}^\ast\vert_{\widetilde{G}}}$. It is immediatel{}y that this map is a $\numC$-linear embedding. From the formula \eqref{eq:htslashgamma}, it is not hard to see that this map gives embedding from $\JacForm{k}{L}{G}{\rho}$ to $\ModForm{k-\frac{1}{2}\dimn}{G}{\rho\otimes\rho_{\underline{L}}^\ast\vert_{\widetilde{G}}}$, and from $\JacFormCusp{k}{L}{G}{\rho}$ to $\ModFormCusp{k-\frac{1}{2}\dimn}{G}{\rho\otimes\rho_{\underline{L}}^\ast\vert_{\widetilde{G}}}$. This concludes the proof.
\end{proof}
\begin{rema}
\label{rema:thetaIsomorphism}
Suppose $\underline{L}$ is even. Applying the operator $\vert_{k-\frac{1}{2}\dimn}\widetilde{\tbtmat{1}{0}{0}{1}}$ to \eqref{eq:deffh}, and using \eqref{eq:hslashgammaTransformation}, we obtain that
\begin{equation*}
h=\left(\rho \otimes \rho_{\underline{L}}^\ast([v,w],\xi)\right)\circ h
\end{equation*}
for any $([v,w],\xi) \in \heigrpa{L}{1}$. This is equivalent to
\begin{equation*}
h_t(\tau)=\xi\etp{\frac{1}{2}B(v,w)}\rho([v,w],\xi)(h_t(\tau)),\quad \tau \in \uhp,\, t \in L^\sharp/L.
\end{equation*}
So if the set $\{h_t(\tau)\colon \tau \in \uhp,\, t \in L^\sharp/L\}$ generates $\myran$, then $\rho$ must satisfy
\begin{equation}
\label{eq:rhoCondition}
\rho([v,w],\xi)=\xi\etp{\frac{1}{2}B(v,w)}\mathrm{id}_\myran.
\end{equation}
Without loss of generality, for even lattices, we mostly consider $\rho$ satisfying this condition. In fact, under this condition, by reversing the deduction in the above proof, one can show that the map from $\JacFormWHol{k}{L}{G}{\rho}$ to $\ModFormWHol{k-\frac{1}{2}\dimn}{G}{\rho\otimes\rho_{\underline{L}}^\ast\vert_{\widetilde{G}}}$ in Theorem \ref{thm:thetaDecomposition} is really an isomorphism. So are its restrictions to $\JacForm{k}{L}{G}{\rho}$ and to $\JacFormCusp{k}{L}{G}{\rho}$. Nevertheless, we do not need the surjectiveness below.
\end{rema}

\begin{rema}
\label{rama:conditionOnRhoThetaDecompWeak}
If we replace the condition \eqref{eq:conditionOnRhoThetaDecomp} by weaker ones
\begin{equation}
\label{eq:conditionOnRhoThetaDecompWeak}
\rho([v,0],1)=\mathrm{id}_\myran,\quad \rho([0,v],1)=\mathrm{id}_\myran,\quad v \in L,
\end{equation}
then \eqref{eq:thetaDecomposition} still holds for $\gamma=\widetilde{\tbtmat{1}{0}{0}{1}}$ and any $\phi \in \JacFormWHol{k}{L}{G}{\rho}$. This can be seen immediately from the proof.
\end{rema}

In the remaining of this section, we describe the relationship between Taylor expansions and theta decompositions. For results in the classical setting, see \cite[Proposition 2.40]{BFOR17}. Although we could obtain a formula relating modified Taylor coefficients $D_{k,\mathbf{p}}(\phi)$ to theta coefficients \eqref{eq:deffhtgamma} concerning Rankin-Cohen brackets, we only present a formula involving non-modified Taylor coefficients, for this is precisely what we need to prove the main result (Theorem \ref{thm:mainEmbed}) of this paper.

For a $\numR$-basis $\mathfrak{B}$ of $V$, let $(z_1,z_2,\dots,z_\dimn)$ denote the corresponding coordinate map on $\mathcal{V}$. Let $\mathbf{p} \in \numgeq{Z}{0}^\dimn$. By $\partial^\mathbf{p}_\mathfrak{B}$, or simply $\partial^\mathbf{p}$, we mean the differential operator $\frac{\partial^{s(\mathbf{p})}}{\partial z_1^{p_1}\dots\partial z_\dimn^{p_\dimn}}$ acting on functions defined on $\mydom$, as in Remark \ref{rema:hpAnotherExpression}.
\begin{lemm}
\label{lemm:TaylorCoeffJacobiTheta}
Let $L$ be any integral lattice, and $G_\mathfrak{B}$ be the Gram matrix of the bilinear form $B$ with respect to $\mathfrak{B}$. For $v \in \mathcal{V}$, let $v_\mathfrak{B}$ denote the coordinate vector of $v$ under the basis $\mathfrak{B}$, as a $1 \times \dimn$ matrix. Then
\begin{equation*}
\partial^\mathbf{p}\vartheta_{\underline{L}}\vert_{z=0}(\tau)=(2\uppi \rmi)^{s(\mathbf{p})}\sum_{v \in L^\sharp}(v_\mathfrak{B}\cdot G_\mathfrak{B})^{\mathbf{p}}\delta_{v+L}q^{Q(v)}.
\end{equation*}
Equivalently, for any $t \in L^\sharp/L$, we have
\begin{equation*}
\partial^\mathbf{p}\vartheta_{\underline{L}, t}\vert_{z=0}(\tau)=(2\uppi \rmi)^{s(\mathbf{p})}\sum_{v \in L}((v+t)_\mathfrak{B}\cdot G_\mathfrak{B})^{\mathbf{p}}q^{Q(v+t)}.
\end{equation*}
\end{lemm}
\begin{proof}
Immediately from definitions, that is, \eqref{eq:deffJacobiThetaLatticeIndex} and \eqref{eq:deffJacobiThetaLatticeIndex2}.
\end{proof}
\begin{rema}
Note that $(v_\mathfrak{B}\cdot G_\mathfrak{B})^{\mathbf{p}}$ can be rewritten as $\prod_{l=1}^\dimn B(v,e_l)^{p_l}$ where $p_l$ is the $l$-th component of $\mathbf{p}$ and $\mathfrak{B}=(e_1,\dots,e_\dimn)$.
\end{rema}

\begin{prop}
\label{prop:TaylorCoeffAndThetaDecomp}
Use notations in Convention \ref{conv1} and \ref{conv2}. The lattice $\underline{L}$ may be even or odd. Suppose that $\rho$ satisfies \eqref{eq:conditionOnRhoThetaDecompWeak}. Let $\phi=\sum_{n,t}c(n,t)q^n\etp{B(t,z)} \in \JacFormWHol{k}{L}{G}{\rho}$ and put $h_{t}(\tau)=\sum_{n}c(n+Q(t), t)q^n$ for $t \in L^\sharp/L$. Let $\mathfrak{B}=(e_1,\dots,e_\dimn)$ be a $\numR$-basis of $V$ and $G_\mathfrak{B}$ be the Gram matrix of $B$ with respect to $\mathfrak{B}$. Suppose the Taylor expansion of $\phi$ with respect to $\mathfrak{B}$ is
\begin{equation*}
\phi(\tau,z_1e_1+\dots z_\dimn e_\dimn)=\sum_{\mathbf{j} \in \numgeq{Z}{0}^\dimn}g_{\mathbf{j}}(\tau)z_1^{j_1}\dots z_\dimn^{j_\dimn}.
\end{equation*}
Then we have
\begin{equation*}
g_{\mathbf{j}}(\tau)=\frac{(2\uppi \rmi)^{s(\mathbf{j})}}{\mathbf{j}!}\sum_{t \in L^\sharp/L}h_t(\tau)\sum_{v \in L}((v+t)_\mathfrak{B}\cdot G_\mathfrak{B})^{\mathbf{j}}q^{Q(v+t)}.
\end{equation*}
\end{prop}
\begin{proof}
This follows from \eqref{eq:thetaDecomposition} with $\gamma = \widetilde{\tbtmat{1}{0}{0}{1}}$ and Lemma \ref{lemm:TaylorCoeffJacobiTheta}.
\end{proof}
\begin{coro}
\label{coro:LinearIndependence}
Use notations and assumptions above. Let $\mathcal{P}_0=\{\mathbf{p}_1, \dots, \mathbf{p}_d\}$ be a finite subset of $\numgeq{Z}{0}^\dimn$ with $d=\abs{L^\sharp/L}$ elements. Assume $L^\sharp/L = \{t_1,\dots, t_d\}$. If the function
\begin{equation}
\label{eq:detThetaMatrix}
\det{\left(\frac{1}{(2\uppi \rmi)^{s(\mathbf{p}_i)}}\partial^{\mathbf{p}_i}\vartheta_{\underline{L},t_j}\vert_{z=0}(\tau)\right)_{1 \leq i,j \leq d}}
\end{equation}
is not identically zero for $\tau \in \uhp$, then
\begin{equation*}
g_{\mathbf{p}_i}=0\,(1 \leq i \leq d) \implies \phi = 0.
\end{equation*}
\end{coro}
\begin{proof}
Suppose $g_{\mathbf{p}_i}=0\,(1 \leq i \leq d)$. By Lemma \ref{lemm:TaylorCoeffJacobiTheta} and Proposition \ref{prop:TaylorCoeffAndThetaDecomp} we have
\begin{equation}
\label{eq:hjThetaEqualsZero}
\sum_{j=1}^{d}h_{t_j}(\tau)\partial^{\mathbf{p}_i}\vartheta_{\underline{L},t_j}\vert_{z=0}(\tau)=0
\end{equation}
with $1 \leq i \leq d$. Choose some $\tau_0$ such that the function \eqref{eq:detThetaMatrix} assumes a non-zero value. By continuity \eqref{eq:detThetaMatrix} assumes non-zero values on a neighbourhood $U$ of $\tau_0$. Thus, inverting \eqref{eq:hjThetaEqualsZero} gives that $h_{t_j}(\tau)=0$ with $\tau \in U,\, 1 \leq j \leq d$. By analytic continuation $h_{t_j}$ are all identically zero, and hence by theta expansion (actually by Remark \ref{rama:conditionOnRhoThetaDecompWeak}) $\phi=0$.
\end{proof}
\begin{deff}
\label{deff:FLP0tau}
For any integral lattice $\underline{L}$ of determinant $d$, any $\numR$-basis $\mathfrak{B}$ of $V$, and any finite subset $\mathcal{P}_0 \subset \numgeq{Z}{0}^\dimn$ of cardinality $d$, we denote \eqref{eq:detThetaMatrix} by $\mathcal{F}_{\underline{L},\mathcal{P}_0}(\tau)$, or more precisely $\mathcal{F}_{\underline{L},\mathcal{P}_0}^\mathfrak{B}(\tau)$.
\end{deff}
Note that for fixed $\mathfrak{B}$, $\mathcal{F}_{\underline{L},\mathcal{P}_0}(\tau)$ actually depends on orderings $\mathcal{P}_0$ and $L^\sharp/L$ possess. But different choices of orderings only possibly lead to functions differ by a factor $-1$.

Since the function $\mathcal{F}_{\underline{L},\mathcal{P}_0}(\tau)$ is holomorphic, for any specific $\underline{L}$ and $\mathcal{P}_0$, we can check the desired condition by writing out the $q$-expansion of $\mathcal{F}_{\underline{L},\mathcal{P}_0}(\tau)$. Computer algebra systems are suitable for such task.

\begin{examp}
\label{examp:evenUnimodularLatticeEmbedding}
Suppose $\underline{L}$ is even and $\rho$ satisfies \eqref{eq:rhoCondition}. If it is also unimodular, that is, $\abs{L^\sharp/L}=1$, then the structure of $\JacFormWHol{k}{L}{G}{\rho}$ is in some sense the simplest. For instance, consider the $E_8$ lattice, i.e., the module $\numZ^8$ equipped with the bilinear form whose Gram matrix with respect to the standard basis of $\numZ^8$ is the Cartan matrix of the $E_8$ root system. (See \cite[p. 272]{TY05} for this matrix.) We identify $\numC[E_8^\sharp/E_8]$ with $\numC$, so $\rho\otimes\rho_{\underline{L}}^\ast\vert_{\widetilde{G}}=\rho\vert_{\widetilde{G}}$. By Theorem \ref{thm:thetaDecomposition} and Remark \ref{rema:thetaIsomorphism}, we have an isomorphism
\begin{align*}
\ModForm{k-4}{G}{\rho\vert_{\widetilde{G}}} &\rightarrow \JacForm{k}{E_8}{G}{\rho}\\
h(\tau) &\mapsto h(\tau)\cdot \vartheta_{\underline{E_8}, 0}(\tau,z).
\end{align*}
See also \cite[Corollary 3]{Kri96}. On the other hand, according to Theorem \ref{thm:ModularFormFromJacobiForm} (with $\mathfrak{B}$ the standard basis), we have a $\numC$-linear map
\begin{align*}
D_{k, \mathbf{0}}\colon \JacForm{k}{E_8}{G}{\rho} &\rightarrow \ModForm{k}{G}{\rho\vert_{\widetilde{G}}}\\
\phi(\tau,z) &\mapsto \phi(\tau,0).
\end{align*}
We assert that, the above map $D_{k, \mathbf{0}}$, is an embedding. This follows immediately from Corollary \ref{coro:LinearIndependence}. In fact, this assertion holds for any unimodular lattice, whether it is even or old.
\end{examp}

For more concrete examples, see Section \ref{sec:Application: Linear relations among theta series}.

\section{The main theorem and its corollaries}
\label{sec:The main theorem and its corollaries}
Throughout this section, we adopt Convention \ref{conv1} and \ref{conv2}. Recall the definition of $\emb$ before Proposition \ref{prop:diagramSlashOperator}. For $\phi \in \JacFormWHol{k}{L}{G}{\rho}$, if the basis $\mathfrak{B}$ is implicitly known or fixed, we oftern write $D_{k,\mathbf{p}}(\phi)$, instead of $D_{k,\mathbf{p}}(\emb \phi)$, as in the proof of Theorem \ref{thm:ModularFormFromJacobiForm}.

We begin with the following fact, which is a direct consequence of Proposition \ref{prop:isoFormalModForm}.
\begin{lemm}
\label{lemm:JacobiFormEmbed}
Assume that $k \neq 0,-1,-2,\cdots$. Fix a $\numR$-basis $\mathfrak{B}$ of $V$. Then the map
\begin{align*}
\prod_{\mathbf{p} \in \numgeq{Z}{0}^\dimn} D_{k,\mathbf{p}}\colon \JacFormWHol{k}{L}{G}{\rho} &\rightarrow \prod_{\mathbf{p} \in \numgeq{Z}{0}^\dimn}\ModFormWHol{k+s(\mathbf{p})}{G}{\rho\vert_{\widetilde{G}}} \\
\notag\phi &\mapsto \left\langle D_{k,\mathbf{p}}\phi \middle\vert \mathbf{p} \in \numgeq{Z}{0}^\dimn \right\rangle
\end{align*}
is a $\numC$-linear embedding. Moreover, it maps $\JacFormWeak{k}{L}{G}{\rho}$ into $\prod_{\mathbf{p} \in \numgeq{Z}{0}^\dimn}\ModForm{k+s(\mathbf{p})}{G}{\rho\vert_{\widetilde{G}}}$, and maps $\JacForm{k}{L}{G}{\rho}$ into $\ModForm{k}{G}{\rho\vert_{\widetilde{G}}} \times \prod_{\mathbf{0} \neq \mathbf{p} \in \numgeq{Z}{0}^\dimn}\ModFormCusp{k+s(\mathbf{p})}{G}{\rho\vert_{\widetilde{G}}}$, respectively.
\end{lemm}
\begin{proof}
The space in which the image of $\phi$ lies can be determined by Theorem \ref{thm:ModularFormFromJacobiForm}. The injectiveness follows from Proposition \ref{prop:isoFormalModForm}.
\end{proof}

To state our main theorem, some more notations are needed. We define a partial ordering on $\numgeq{Z}{0}^\dimn$ as follows. Let $\mathbf{p}_1$ and $\mathbf{p}_2$ be vectors in $\numgeq{Z}{0}^\dimn$. By $\mathbf{p}_1 \preceq \mathbf{p}_2$, we mean $\mathbf{p}_2 - \mathbf{p}_1 \in \numgeq{Z}{0}^\dimn$ and $s(\mathbf{p}_1) \equiv s(\mathbf{p}_2) \bmod 2$. Let $\mathcal{P}$ be a subset of $\numgeq{Z}{0}^\dimn$. We put $\widehat{\mathcal{P}}=\{\mathbf{p} \in \numgeq{Z}{0}^\dimn \colon \mathbf{p} \preceq \mathbf{p}' \text{ for some }\mathbf{p}' \in \mathcal{P}\}$. If $\mathcal{P}$ is a singleton, that is $\mathcal{P}=\{\mathbf{p}_0\}$, then we also write $\widehat{\mathbf{p}_0}$ instead of $\widehat{\mathcal{P}}$. The motivation for concerning this partial ordering is that the $\mathbf{p}$-th Taylor coefficient of a Jacobi form $\phi$ is precisely determined by $D_{k,\mathbf{p}'}\phi$ with $\mathbf{p}' \in \widehat{\mathbf{p}}$. See Lemma \ref{lemm:inverseExpressionDkp}.

We can now state and prove our main theorem. We emphasize that $D_{k,\mathbf{p}}$ denotes $D_{k,B, \mathbf{p}}^{\mathfrak{B}}\circ\emb$ (see Definition \ref{deff:DkpFormal} and the paragraph before Proposition \ref{prop:diagramSlashOperator}), $\vartheta_{\underline{L}, t}$ denotes Jacobi theta series \eqref{eq:deffJacobiThetaLatticeIndex}, and $\partial^{\mathbf{p}}$ is defined in the paragraph preceding Lemma \ref{lemm:TaylorCoeffJacobiTheta}.
\begin{thm}
\label{thm:mainEmbed}
Use notations and assumptions in Convention \ref{conv1} and \ref{conv2}. Suppose further $\rho$ satisfies \eqref{eq:conditionOnRhoThetaDecompWeak}, and $k \neq 0,-1,-2,\cdots$. Let $\mathfrak{B}$ be a $\numR$-basis for $V$, $\mathcal{P}_0=\{\mathbf{p}_1, \dots, \mathbf{p}_d\}$ be a finite subset of $\numgeq{Z}{0}^\dimn$ with $d=\abs{L^\sharp/L}$ elements, and $L^\sharp/L = \{t_1+L,\dots, t_d+L\}$. If the function
\begin{equation*}
\mathcal{F}_{\underline{L},\mathcal{P}_0}^\mathfrak{B}(\tau)=\det{\left(\frac{1}{(2\uppi \rmi)^{s(\mathbf{p}_i)}}\partial^{\mathbf{p}_i}\vartheta_{\underline{L},t_j}\vert_{z=0}(\tau)\right)_{1 \leq i,j \leq d}}
\end{equation*}
is not identically zero for $\tau \in \uhp$, then the map
\begin{align*}
\prod_{\mathbf{p} \in \widehat{\mathcal{P}_0}}D_{k,\mathbf{p}} \colon \JacFormWHol{k}{L}{G}{\rho} &\rightarrow \prod_{\mathbf{p} \in \widehat{\mathcal{P}_0}}\ModFormWHol{k+s(\mathbf{p})}{G}{\rho\vert_{\widetilde{G}}} \\
\notag\phi &\mapsto \left\langle D_{k,\mathbf{p}}\phi \middle\vert \mathbf{p} \in \widehat{\mathcal{P}_0} \right\rangle
\end{align*}
is a $\numC$-linear embedding. Moreover, it gives embeddings from $\JacFormWeak{k}{L}{G}{\rho}$ to $\prod_{\mathbf{p} \in \widehat{\mathcal{P}_0}}\ModForm{k+s(\mathbf{p})}{G}{\rho\vert_{\widetilde{G}}}$, and from $\JacForm{k}{L}{G}{\rho}$ to $\ModForm{k}{G}{\rho\vert_{\widetilde{G}}} \times \prod_{\mathbf{0} \neq \mathbf{p} \in \widehat{\mathcal{P}_0}}\ModFormCusp{k+s(\mathbf{p})}{G}{\rho\vert_{\widetilde{G}}}$, respectively. The last assertion on $\JacForm{k}{L}{G}{\rho}$ requires that $\mathbf{0} \in \widehat{\mathcal{P}_0}$.
\end{thm}
\begin{proof}
The fact that the map considered is $\numC$-linear and that it maps $\JacFormWeak{k}{L}{G}{\rho}$ to $\prod_{\mathbf{p} \in \widehat{\mathcal{P}_0}}\ModForm{k+s(\mathbf{p})}{G}{\rho\vert_{\widetilde{G}}}$, and $\JacForm{k}{L}{G}{\rho}$ to $\ModForm{k}{G}{\rho\vert_{\widetilde{G}}} \times \prod_{\mathbf{0} \neq \mathbf{p} \in \widehat{\mathcal{P}_0}}\ModFormCusp{k+s(\mathbf{p})}{G}{\rho\vert_{\widetilde{G}}}$ follow from\footnote{These two facts require neither \eqref{eq:conditionOnRhoThetaDecompWeak} nor \eqref{eq:detThetaMatrix} to be nonzero. Actually, they even do not require that $k \neq 0,-1,-2,\cdots$.} Lemma \ref{lemm:JacobiFormEmbed}. So it remains to show the injectiveness, for which we compute the kernel. Let $\phi \in \JacFormWHol{k}{L}{G}{\rho}$ such that $D_{k,\mathbf{p}'}\phi=0$ for any $\mathbf{p}' \in \widehat{\mathcal{P}_0}$. By Lemma \ref{lemm:inverseExpressionDkp}(2), the $\mathbf{p}$-th Taylor coefficient of $z \mapsto \phi(\tau,z)$ (under the basis $\mathfrak{B}$) is a linear combination of
\begin{equation}
\label{eq:differentialDkp}
\dodth{(s(\mathbf{p})-s(\mathbf{p}'))/2}D_{k, \mathbf{p}'}\phi,\quad \mathbf{p}' \in \widehat{\mathbf{p}}.
\end{equation}
Hence, if $\mathbf{p} \in \mathcal{P}_0$, then the $\mathbf{p}$-th Taylor coefficient is zero since $\widehat{\mathbf{p}} \subseteq \widehat{\mathcal{P}_0}$. It follows immediately from this fact and Corollary \ref{coro:LinearIndependence} that $\phi=0$. This proves the injectiveness, hence concludes the proof.
\end{proof}

\begin{rema}
To apply this theorem, it is essential to obtain a set $\mathcal{P}_0$ first. Note that the choice of $\mathcal{P}_0$ depends only on the lattice $\underline{L}=(L, B)$, not on the weight $k$, the group $G$ or the representation $\rho$. In the remaining, we make no effort to prove the existence of $\mathcal{P}_0$ for all lattices theoretically, but only try to find out $\mathcal{P}_0$ for some specific lattices computationally.
\end{rema}

\begin{rema}
In certain circumstances, we may drop or loosen the restriction $k \neq 0,-1,-2,\cdots$, i.e., $k$ may be some non-positive integer. To see this, note it is in the step showing the $\mathbf{p}$-th Taylor coefficient of $z \mapsto \phi(\tau,z)$ is a linear combination of \eqref{eq:differentialDkp} that we use the fact $k \neq 0,-1,-2,\cdots$, for Lemma \ref{lemm:inverseExpressionDkp} requires this. Hence, for fixed $\mathcal{P}_0$, one may analyse expressions $C_1$ and $C_2$ in Lemma \ref{lemm:inverseExpressionDkp}(2) to seek those $k$'s such that these two expressions are both nonzero for any $\mathbf{p} \in \widehat{\mathcal{P}_0}$. For such $k$, Theorem \ref{thm:mainEmbed} still holds. In the remaining, we always assume that $k \neq 0,-1,-2,\cdots$ for the sake of simplicity.
\end{rema}

Concrete examples will be given in the next section. The rest of this section is devoted to some of immediate corollaries concerning lattices of small determinants. All these corollaries assume the same notations and conditions as Theorem \ref{thm:mainEmbed}.

\begin{coro}
Suppose $\abs{L^\sharp/L}=2$, and put $L^\sharp/L=\{0+L, \alpha+L\}$. Suppose the basis $\mathfrak{B}=(e_1,\dots,e_\dimn)$ satisfies that there is some $w \in L$ such that $Q(\alpha+w)$ is the least in $\{Q(\alpha+v)\colon v \in L\}$ and $B(\alpha+w,e_\dimn) \neq 0$. Put $\mathbf{p}=(0,\dots,0,2)$. Then the following three maps
\begin{align*}
D_{k,\mathbf{0}} \times D_{k,\mathbf{p}} \colon \JacFormWHol{k}{L}{G}{\rho} &\rightarrow \ModFormWHol{k}{G}{\rho\vert_{\widetilde{G}}} \times\ModFormWHol{k+2}{G}{\rho\vert_{\widetilde{G}}} \\
\JacFormWeak{k}{L}{G}{\rho} &\rightarrow \ModForm{k}{G}{\rho\vert_{\widetilde{G}}} \times\ModForm{k+2}{G}{\rho\vert_{\widetilde{G}}} \\
\JacForm{k}{L}{G}{\rho} &\rightarrow \ModForm{k}{G}{\rho\vert_{\widetilde{G}}} \times\ModFormCusp{k+2}{G}{\rho\vert_{\widetilde{G}}}
\end{align*}
are all $\numC$-linear embeddings.
\end{coro}
\begin{proof}
In this case, $\mathbf{P}_0=\{\mathbf{0}, \mathbf{p}\}$. So $\widehat{\mathbf{P}_0}=\mathbf{P}_0$. By Theorem \ref{thm:mainEmbed}, we shall verify that 
\begin{multline*}
\sum_{v \in L}q^{Q(v)}\cdot \sum_{v \in L}((\alpha_1+v_1)g_{1 \dimn}+\dots (\alpha_\dimn+v_\dimn) g_{\dimn \dimn})^2q^{Q(\alpha+v)} \\
-\sum_{v \in L}q^{Q(\alpha+v)}\cdot\sum_{v \in L}(v_1g_{1 \dimn}+\dots v_\dimn g_{\dimn \dimn})^2q^{Q(v)} \neq 0,
\end{multline*}
where $(v_1,\dots,v_\dimn)$ and $(\alpha_1,\dots,\alpha_\dimn)$ are coordinate vectors of $v$ and $\alpha$ with respect to $\mathfrak{B}$ respectively, and $g_{ij}$ is the $(i,j)$-entry of the Gram matrix of $B$ with respect to $\mathfrak{B}$. Set $m=\min\{Q(\alpha+v)\colon v \in L\}$. Then the $q^m$-term of the left-hand side of the desired inequality is non-zero, since $(\alpha_1+w_1)g_{\dimn 1}+\dots (\alpha_\dimn+w_\dimn) g_{\dimn \dimn} = B(\alpha+w, e_\dimn) \neq 0$. This concludes the proof.
\end{proof}

\begin{rema}
\label{rema:explicitFormulaDk2}
We write down the explicit formula of $D_{k,\mathbf{p}}\phi$ with $\mathbf{p}=(0,\dots,0,2)$ for the readers' convenience. By Definition \ref{deff:PkpM}, we have
\begin{equation*}
P_{k,\mathbf{p}, G_\mathfrak{B}}(X_0,X_1,\dots,X_\dimn)=-2g_{\dimn \dimn}X_0+2k(g_{1 \dimn}X_1+\dots+g_{\dimn \dimn}X_\dimn)^2.
\end{equation*}
Hence by Proposition \ref{prop:FourierCoeffDkp}, we have
\begin{equation*}
D_{k,\mathbf{p}}\phi=2\cdot(2\uppi\rmi)^2\sum_n\left(\sum_{t}c(n,t)(-g_{\dimn \dimn}n+k(g_{1 \dimn}t_1+\dots+g_{\dimn \dimn}t_\dimn)^2)\right)q^n,
\end{equation*}
where $\phi=\sum_{n,t}c(n,t)q^n\etp{B(t,z)}$ and $(t_1,\dots,t_\dimn)$ is the coordinate vector of $t$ with respect to $\mathfrak{B}$. This formula is general for all lattices, not only for those discussed in the above corollary. The expression \eqref{eq:exampleDkpHalfintegralWeight} is essentially a special case of this formula.
\end{rema}

\begin{coro}
\label{coro:embeddingDet3}
Suppose $\abs{L^\sharp/L}=3$, and let $\alpha+L$ be a generator of $L^\sharp/L$. Put
\begin{equation*}
S=\{v \in \alpha+L \colon Q(v) \leq Q(v') \text{ for any }v' \in \alpha+L\}.
\end{equation*}
Suppose the basis $\mathfrak{B}=(e_1,\dots,e_\dimn)$ satisfies that there is some $w \in S$ with $B(w, e_\dimn) \neq 0$, and that $\sum_{v \in L} B(\alpha+v,e_\dimn)q^{Q(\alpha+v)}$ is not identically zero. Put $\mathbf{p}_1=(0,\dots,0,1)$ and $\mathbf{p}_2=(0,\dots,0,2)$. Then the following three maps $D_{k,\mathbf{0}} \times D_{k,\mathbf{p}_1} \times D_{k,\mathbf{p}_2}$:
\begin{align*}
\JacFormWHol{k}{L}{G}{\rho} &\rightarrow \ModFormWHol{k}{G}{\rho\vert_{\widetilde{G}}} \times\ModFormWHol{k+1}{G}{\rho\vert_{\widetilde{G}}} \times\ModFormWHol{k+2}{G}{\rho\vert_{\widetilde{G}}} \\
\JacFormWeak{k}{L}{G}{\rho} &\rightarrow \ModForm{k}{G}{\rho\vert_{\widetilde{G}}} \times\ModForm{k+1}{G}{\rho\vert_{\widetilde{G}}} \times\ModForm{k+2}{G}{\rho\vert_{\widetilde{G}}} \\
\JacForm{k}{L}{G}{\rho} &\rightarrow \ModForm{k}{G}{\rho\vert_{\widetilde{G}}} \times\ModFormCusp{k+1}{G}{\rho\vert_{\widetilde{G}}} \times\ModFormCusp{k+2}{G}{\rho\vert_{\widetilde{G}}}
\end{align*}
are all $\numC$-linear embeddings.
\end{coro}
\begin{proof}
In this case, $\mathbf{P}_0=\{\mathbf{0}, \mathbf{p}_1, \mathbf{p}_2\}$. So $\widehat{\mathbf{P}_0}=\mathbf{P}_0$. By Theorem \ref{thm:mainEmbed}, we shall verify that the determinant of the following matrix is not identically zero:
\begin{equation*}
\begin{pmatrix}
\sum q^{Q(v)} & \sum q^{Q(\alpha+v)} & \sum q^{Q(2\alpha+v)} \\
\sum B(v,e_\dimn)q^{Q(v)} & \sum B(\alpha+v,e_\dimn)q^{Q(\alpha+v)} & \sum B(2\alpha+v,e_\dimn)q^{Q(2\alpha+v)} \\
\sum B(v,e_\dimn)^2q^{Q(v)} & \sum B(\alpha+v,e_\dimn)^2q^{Q(\alpha+v)} & \sum B(2\alpha+v,e_\dimn)^2q^{Q(2\alpha+v)}
\end{pmatrix}.
\end{equation*}
All above summations are over $v \in L$. Put $m=\min_{v \in L} Q(\alpha+v)=\min_{v \in L} Q(2\alpha+v)$ and let $a_{ij}$ be the $(i,j)$-entry of this matrix. By positive definiteness we have $m > 0$. By a change of variables we see that $a_{12}=a_{13}$, $a_{22}=-a_{23}$ and $a_{32}=a_{33}$, so the determinant is $2a_{22}(a_{11}a_{32}-a_{12}a_{31})$. By assumption, $a_{22} \neq 0$, and by the existence of $w$, the $q^m$-term of $a_{11}a_{32}-a_{12}a_{31}$ is nonzero. Hence the determinant is nonzero, which concludes the proof.
\end{proof}

\begin{rema}
As in Remark \ref{rema:explicitFormulaDk2}, we write down the explicit formula of $D_{k,\mathbf{p}_1}\phi$ with $\mathbf{p}_1=(0,\dots,0,1)$. By Definition \ref{deff:PkpM}, we have
\begin{equation*}
P_{k,\mathbf{p}_1, G_\mathfrak{B}}(X_0,X_1,\dots,X_\dimn)=2(g_{1 \dimn}X_1+\dots+g_{\dimn \dimn}X_\dimn),
\end{equation*}
where $g_{ij}$ is the $(i,j)$-entry of the Gram matrix $G_\mathfrak{B}$ of $B$ with respect to $\mathfrak{B}$.
Hence by Proposition \ref{prop:FourierCoeffDkp}, we have
\begin{equation*}
D_{k,\mathbf{p}_1}\phi=2\uppi\rmi\sum_n\left(\sum_{t}c(n,t)(g_{1 \dimn}t_1+\dots+g_{\dimn \dimn}t_\dimn)\right)q^n,
\end{equation*}
where $\phi=\sum_{n,t}c(n,t)q^n\etp{B(t,z)}$ and $(t_1,\dots,t_\dimn)$ is the coordinate vector of $t$ with respect to $\mathfrak{B}$.
\end{rema}

For $\underline{L}$ with large determinant, it is computationally inefficient to evaluate $\mathcal{F}_{\underline{L},\mathcal{P}_0}$ directly. We provide a useful criterion for $\mathcal{F}_{\underline{L},\mathcal{P}_0}$ to be nonzero below, which works in some situations.
\begin{prop}
\label{prop:whenFLP0Nonzero}
Use notations in Theorem \ref{thm:mainEmbed}. Suppose $\dimn \geq 2$, $\mathbf{p}_1 = \mathbf{0}$, and $t_1 = 0$. Put
\begin{equation*}
S_i = \{v \in t_i+L \colon Q(v) \leq Q(v') \text{ for any }v' \in t_i+L\}
\end{equation*}
for $i=2,3,\dots,d$. For $\mathbf{p}_i \in \mathcal{P}_0$, the symbol $p_{i,l}$ denotes its $l$-th component. Suppose the ordered basis $\mathfrak{B}$ equals $(e_1,\dots,e_\dimn)$. If the $(d-1)\times(d-1)$ matrix
\begin{equation}
\label{eq:matBp}
\left(\sum_{v \in S_j}\prod_{l=1}^\dimn B(v,e_l)^{p_{i,l}}\right)_{2 \leq i,j \leq d}
\end{equation}
is non-singular, then $\mathcal{F}_{\underline{L},\mathcal{P}_0}$ is not identically zero.
\end{prop}
\begin{proof}
Put $m_j = \min_{v \in t_j+L}Q(v)$ for $j=2,\dots,d$, and
\begin{equation*}
a_{ij}(\tau) = \frac{1}{(2\uppi \rmi)^{s(\mathbf{p}_i)}}\partial^{\mathbf{p}_i}\vartheta_{\underline{L},t_j}\vert_{z=0}(\tau).
\end{equation*}
Treat $a_{ij}$ as its power series expansion with respect to $q=\etp{\tau}$ (with possibly fractional powers). By Lemma \ref{lemm:TaylorCoeffJacobiTheta}, $\sum_{v \in S_j}\prod_{l=1}^\dimn B(v,e_l)^{p_{i,l}}$ is the coefficient of $q^{m_j}$-term of $a_{ij}$ for $i,j \geq 2$. Since the constant term of $a_{11}$ is $1$ and that of $a_{i1}$ is $0$ for $i \geq 2$, the determinant of \eqref{eq:matBp} is equal to the coefficient of the $q^{\sum_{2 \leq j \leq d} m_j}$-term of $\det (a_{ij})_{1 \leq i,j \leq d}$. Hence from the non-singularity of \eqref{eq:matBp} $\mathcal{F}_{\underline{L},\mathcal{P}_0}$ is not identically zero.
\end{proof}

What can we use an embedding to do? One immediate application is to find linear relations among functions in the domain of the embedding. Such a relation holds, if and only if the corresponding relation in the codomain holds. The codomian of the embedding in Theorem \ref{thm:mainEmbed} is roughly a product (direct sum) of spaces of modular forms, in which we have standard algorithm to find linear relations, combining basic linear algebra and some theory of ordinary modular forms, at least in the scalar-valued case. In the next section, we will do this for certain Jacobi theta series of lattice index.

\section{Application: Linear relations among theta series}
\label{sec:Application: Linear relations among theta series}
Functions studied in this section are the followings.
\begin{deff}
\label{deff:thetaFunctionStudied}
Suppose $L$ is an even lattice in $\underline{V}=(V, B)$. Let $N$ be the level of $L$, the least positive integer such that $N\cdot Q(v) \in \numZ$ for all $v \in L^\sharp$. For $\alpha \in L^\sharp/L$ and $\beta \in \frac{1}{N}L/L^\sharp$, we define
\begin{equation*}
\theta_{\alpha,\beta}(\tau,z)=\sum_{v \in L}\etp{B(\beta,v)}\etp{\tau Q(\alpha+v)+B(\alpha+v,z)}, \qquad (\tau,z) \in \mydom.
\end{equation*}
Note that we omit the information $L$ in the notation.
\end{deff}
For fixed $L$, there are totally $N^\dimn$ such theta series. But there is something subtle in this definition: the function $\theta_{\alpha, \beta}$ actually depends on the representatives we choose for $\alpha$ and $\beta$. Different choice of representatives leads to functions differ by a constant factor. So this ambiguity may not cause trouble. One of the aims of this section is to find linear relations among $\{\theta_{\alpha,\beta}^N \colon \alpha \in L^\sharp/L,\,\beta \in \frac{1}{N}L/L^\sharp\}$, for some specific lattices $L$. Note that $\theta_{\alpha, \beta}^N$ is independent of the choice of representatives of $\alpha$ and $\beta$. Furthermore, if $N'$ is a positive divisor of $N$ and $\beta$ satisfies $N'\beta=0+L^\sharp$, then $\theta_{\alpha, \beta}^{N'}$ is also independent of the choice of representatives.

\begin{examp}
\label{examp:oneDimenEvenLattice}
Put $\underline{L}=(\numZ,\, Q(x)=mx^2)$, where $m$ is any positive integer. (So $\mathcal{V}=\numC$.) Then $L^\sharp=\frac{1}{2m}\numZ$ and the level is $N=4m$. We fix a system of representatives of $L^\sharp/L$, namely $\frac{0}{2m},\,\frac{1}{2m},\ldots,\frac{2m-1}{2m}$, and one of $\frac{1}{N}L/L^\sharp$, namely $0, \frac{1}{4m}$. Then
\begin{align}
\theta_{j/2m,0}(4m\tau,z)&=\sum_{v \in \numZ}q^{(j+2mv)^2}\zeta^{j+2mv},\\
\theta_{j/2m,1/4m}(4m\tau,z)&=\sum_{v \in \numZ}(-1)^vq^{(j+2mv)^2}\zeta^{j+2mv},
\end{align}
where $j \in \{0,\,1,\ldots,2m-1\}$, $q=\etp{\tau}$ and $\zeta=\etp{z}$.
\end{examp}

\begin{examp}
Put $\underline{L}=(\numZ^2,\, Q(x,y)=m(x^2+xy+y^2))$, where $m$ is any positive integer. (So $\mathcal{V}=\numC^2$.) The Gram matrix of this lattice with respect to the standard basis of $\numZ^2$ is $\tbtmat{2m}{m}{m}{2m}$. So $L^\sharp=\frac{1}{3m}(\numZ(2,-1)\oplus\numZ(-1,2))$, and the level is $N=3m$. We fix a system of representatives of $L^\sharp/L$ as follows:
\begin{equation*}
\left\{\left(\frac{i}{3m},\frac{j}{3m}\right)\colon 0 \leq i,j \leq 3m-1,\,\text{and } 3 \mid i-j\right\},
\end{equation*}
and one of $\frac{1}{N}L/L^\sharp$ as follows:
\begin{equation*}
\left(0,0\right),\, \left(0,\frac{1}{3m}\right),\, \left(0,\frac{2}{3m}\right).
\end{equation*}
Then for $\alpha=\left(\frac{i}{3m},\frac{j}{3m}\right)$ and $\beta=\left(0,\frac{k}{3m}\right)$, we have
\begin{multline}
\theta_{\alpha,\beta}(9m\tau, 3z)\\
=\sum_{v_1,v_2 \in \numZ}\etp{\frac{k(v_1-v_2)}{3}}q^{(3mv_1+i)^2+(3mv_1+i)(3mv_2+j)+(3mv_2+j)^2}\\
\times\zeta_1^{2(3mv_1+i)+(3mv_2+j)}\zeta_2^{(3mv_1+i)+2(3mv_2+j)},
\end{multline}
where $z=(z_1,z_2) \in \numC^2$, $q=\etp{\tau}$, $\zeta_1=\etp{z_1}$ and $\zeta_2=\etp{z_2}$.
\end{examp}

To apply Theorem \ref{thm:mainEmbed} to functions in Definition \ref{deff:thetaFunctionStudied} (including those in the above two examples), we need to know their transformation laws. Recall that
\begin{align}
\Gamma_0(N)&=\left\{\tbtmat{a}{b}{c}{d} \in \slZ \colon c \equiv 0 \bmod N\right\},\label{eq:Gamma0N}\\
\Gamma_1(N)&=\left\{\tbtmat{a}{b}{c}{d} \in \slZ \colon a \equiv d \equiv 1 \bmod N,\quad c \equiv 0 \bmod N\right\}.\label{eq:Gamma1N}
\end{align}

\begin{lemm}
\label{lemm:thetaLatticeTransformationN}
Suppose $L$ is even, $N$ is the level of $L$, and $t \in L^\sharp/L$. Then for $\tbtmat{a}{b}{c}{d} \in \Gamma_0(N)$ with $d \neq 0$, we have
\begin{equation*}
\vartheta_{\underline{L}, t}\vert_{\dimn/2,B}\widetilde{\tbtmat{a}{b}{c}{d}}=(-\rmi)^{\frac{1}{2}\dimn(1-\sgn{d})\legendre{c}{-1}}\cdot\mathfrak{g}(b,d;t)\vartheta_{\underline{L}, at},
\end{equation*}
where
\begin{equation}
\label{eq:GaussSum}
\mathfrak{g}(b,d;t)=\abs{d}^{-\frac{\dimn}{2}}\sum_{v \in L/dL}\etp{\frac{bQ(t+v)}{d}}.
\end{equation}
\end{lemm}
\begin{rema}
One can work out an explicit expression for the generalized Gauss sum $\mathfrak{g}(b,d;t)$. See \cite[\S 14.3]{CoS17} for instance. But we do not need this. Moreover, the quantity $\mathfrak{g}(b,d;t)$ actually depends on the lattice, so we write $\mathfrak{g}_{\underline{L}}(b,d;t)$ when needed.
\end{rema}
We give a sketch of proof below for the readers' convenience. See also \cite[\S 14.3]{CoS17} for details of this proof.
\begin{proof}
The case $c=0$ can be proved directly. So assume that $c \neq 0$ below. We first prove a more general formula:
\begin{equation}
\label{eq:thetaTransformationMu}
\vartheta_{\underline{L}, t}\vert_{\dimn/2,B}\widetilde{\tbtmat{a}{b}{c}{d}}=(-\rmi)^{\frac{1}{2}\dimn(1-\sgn{d})\legendre{c}{-1}}\cdot\abs{L^\sharp/L}^{-1}\abs{d}^{-\frac{\dimn}{2}}\sum_{x \in L^\sharp/L}\mu_\gamma(t,x)\vartheta_{\underline{L},x},
\end{equation}
where $\gamma=\tbtmat{a}{b}{c}{d} \in \slZ$ (not only $\Gamma_0(N)$), and
\begin{equation}
\label{eq:thetaTransformationMuExpression}
\mu_\gamma(t,x)=\sum_{v \in L/dL}\etp{\frac{bQ(t+v)}{d}}\sum_{y \in L^\sharp/L}\etp{\frac{B(t+v,y)-cQ(y)-dB(y,x)}{d}}.
\end{equation}
Using \eqref{eq:thetaTransformationS} and the identity $\frac{a\tau+b}{c\tau+d}=\frac{a}{c}-\frac{1}{c(c\tau+d)}$, we have
\begin{equation}
\label{eq:thetaTransformationLambda}
\vartheta_{\underline{L}, t}\vert_{\dimn/2,B}\widetilde{\tbtmat{a}{b}{c}{d}}=(-\rmi\cdot\sgn{c})^{\frac{1}{2}\dimn}\cdot\abs{L^\sharp/L}^{-\frac{1}{2}}\abs{c}^{-\frac{\dimn}{2}}\sum_{x \in L^\sharp/L}\lambda_\gamma(t,x)\vartheta_{\underline{L},x},
\end{equation}
where
\begin{equation}
\lambda_\gamma(t,x)=\sum_{y \in L/cL}\etp{\frac{aQ(t+y)-B(t+y,x)+dQ(x)}{c}}.
\end{equation}
Note that $\widetilde{\tbtmat{a}{b}{c}{d}}=\widetilde{\tbtmat{-b}{a}{-d}{c}}\cdot\eleglptRsimple{0}{-1}{1}{0}{\varepsilon(c,d)}$, where $\varepsilon(c,d)=-1$ if $c,d$ are both negative, and $\varepsilon(c,d)=1$ otherwise. Hence
\begin{equation*}
\vartheta_{\underline{L}, t}\vert_{\dimn/2,B}\widetilde{\tbtmat{a}{b}{c}{d}}=\vartheta_{\underline{L}, t}\vert_{\dimn/2,B}\widetilde{\tbtmat{-b}{a}{-d}{c}}\vert_{\dimn/2,B}\eleglptRsimple{0}{-1}{1}{0}{\varepsilon(c,d)}.
\end{equation*}
Now \eqref{eq:thetaTransformationMu} follows from applying \eqref{eq:thetaTransformationLambda} with $\tbtmat{a}{b}{c}{d}$ replaced by $\tbtmat{-b}{a}{-d}{c}$, and then using \eqref{eq:thetaTransformationS} in the above formula.

Then we prove the original formula from \eqref{eq:thetaTransformationMu}. Consider the group homomorphism
\begin{align*}
L^\sharp/L &\rightarrow L^\sharp/L\\
y+L &\mapsto dy+L.
\end{align*}
Since $\tbtmat{a}{b}{c}{d} \in \Gamma_0(N)$ by assumption, we have $cL^\sharp \subseteq L$, so the above homomorphism is an isomorphism. Applying a change of variable corresponding to this isomorphism in the inner sum of the right-hand side of \eqref{eq:thetaTransformationMuExpression}, we obtain that
\begin{equation*}
\mu_\gamma(t,x)=\sum_{v \in L/dL}\etp{\frac{bQ(t+v)}{d}}\sum_{y \in L^\sharp/L}\etp{B(t-dx,y)}.
\end{equation*}
The lemma follows from this and \eqref{eq:thetaTransformationMu}.
\end{proof}

\begin{prop}
\label{prop:thetaAlphaBetaJacobiForm}
Use notations and assumptions in Definition \ref{deff:thetaFunctionStudied}. We have $\theta_{\alpha,\beta} \in \JacForm{\dimn/2}{L}{\Gamma_1(N)}{\chi_{\alpha,\beta}}$, where $\chi_{\alpha,\beta}$ is the linear character on $\Jacgrp{\Gamma_1(N)}{L}$ that maps $\eleactgrpnsimple{a}{b}{c}{d}{\varepsilon}{v}{w}{\xi}$ (with $d \neq 0$) to
%\begin{equation*}
%\varepsilon^{-\dimn}\xi(-\rmi)^{\frac{1}{2}\dimn(1-\sgn{d})\legendre{c}{-1}}\cdot\mathfrak{g}(b,d;0)\etp{abQ(\alpha)-cdQ(\beta)+\frac{1}{2}B(v,w)-B(v,\beta)}.
%\end{equation*}
\begin{equation*}
\varepsilon^{-\dimn}\xi(-\rmi)^{\frac{1}{2}\dimn(1-\sgn{d})\legendre{c}{-1}}\cdot\mathfrak{g}(b,d;0)\etp{bQ(\alpha)-cQ(\beta)+\frac{1}{2}B(v,w)-B(v,\beta)}.
\end{equation*}
\end{prop}
\begin{proof}
Note that for $\alpha \in L^\sharp$ and $\beta \in \frac{1}{N}L$, 
\begin{equation*}
\theta_{\alpha,\beta}=\etp{-B(\alpha,\beta)}\cdot\vartheta_{\underline{L}, \alpha}\vert_{\dimn/2,B}(\widetilde{I}, [0,\beta], 1).
\end{equation*}
Set $C = \etp{-B(\alpha,\beta)}$. Thus, for any $\eleactgrpnsimple{a}{b}{c}{d}{\varepsilon}{v}{w}{\xi} \in \Jacgrp{\slZ}{L}$, we have
\begin{multline}
\label{eq:proofThetaJacobiForm}
\theta_{\alpha,\beta}\vert_{\dimn/2,B}\eleactgrpnsimple{a}{b}{c}{d}{\varepsilon}{v}{w}{\xi}\\
=C\cdot \vartheta_{\underline{L}, \alpha}\vert_{\dimn/2,B}\eleglptRsimple{a}{b}{c}{d}{\varepsilon}\vert_{\dimn/2,B}([v,w],\xi)\vert_{\dimn/2,B}([c\beta,0],1)\\
\vert_{\dimn/2,B}\left([0,d\beta],\etp{-\frac{1}{2}B(c\beta,d\beta)-B(v,d\beta)+B(w,c\beta)}\right).
\end{multline}
The requirement of Fourier expansions of Jacobi forms in Definition \ref{deff:JacobiForm} now follows from this identity and Theorem \ref{thm:thetaSeriesLatticeJacobiForm}. It remains to show the transformation laws of $\theta_{\alpha,\beta}$. To do this, assume that $\tbtmat{a}{b}{c}{d} \in \Gamma_1(N)$. The case $d=0$ is direct. Hence we consider the case $d \neq 0$ below. Using Lemma \ref{lemm:thetaLatticeTransformationN} and \eqref{eq:thetaTransformationL}, the quantity in \eqref{eq:proofThetaJacobiForm} equals
\begin{multline}
\label{eq:proofThetaJacobiForm2}
C\cdot\varepsilon^{-\dimn}\xi(-\rmi)^{\frac{1}{2}\dimn(1-\sgn{d})\legendre{c}{-1}}\cdot\mathfrak{g}(b,d;\alpha)\etp{\frac{1}{2}B(v,w)}\\
\times\vartheta_{\underline{L},a\alpha}\vert_{\dimn,B}([c\beta,0],1)\vert_{\dimn,B}\left([0,d\beta],\etp{-\frac{1}{2}B(c\beta,d\beta)-B(v,d\beta)+B(w,c\beta)}\right).
\end{multline}
Note that $\vartheta_{\underline{L},a\alpha}\vert_{\dimn,B}([c\beta,0],1)=\vartheta_{\underline{L},a\alpha}$ since $c\beta \in L$, and
\begin{equation*}
\vartheta_{\underline{L},a\alpha}\vert_{\dimn,B}([0,d\beta],1)=\etp{B(a\alpha,d\beta)}\theta_{a\alpha,d\beta}.
\end{equation*}
Inserting these into \eqref{eq:proofThetaJacobiForm2} gives that
\begin{multline*}
\theta_{\alpha,\beta}\vert_{\dimn/2,B}\eleactgrpnsimple{a}{b}{c}{d}{\varepsilon}{v}{w}{\xi}\\
=\varepsilon^{-\dimn}\xi(-\rmi)^{\frac{1}{2}\dimn(1-\sgn{d})\legendre{c}{-1}}\cdot\mathfrak{g}(b,d;\alpha)\etp{\frac{1}{2}B(v,w)}\etp{-B(\alpha,\beta)+B(a\alpha,d\beta)}\\
\times\etp{-\frac{1}{2}B(c\beta,d\beta)-B(v,d\beta)+B(w,c\beta)}\theta_{a\alpha,d\beta}.
\end{multline*}
Since $\tbtmat{a}{b}{c}{d} \in \Gamma_1(N)$, we have $\etp{-B(\alpha,\beta)+B(a\alpha,d\beta)}=1$, $\etp{B(w,c\beta)}=1$, $\etp{B(v,d\beta)}=\etp{B(v,\beta)}$ and $\theta_{a\alpha,d\beta}=\theta_{\alpha,\beta}$. Moreover, by using the change of variable
\begin{align*}
L/dL &\rightarrow L/dL\\
v+dL &\mapsto v+bc\alpha+dL,
\end{align*}
we have
\begin{equation*}
\mathfrak{g}(b,d;\alpha)=\etp{abQ(\alpha)}\cdot\mathfrak{g}(b,d;0).
\end{equation*}
Taking all these into account, we obtain the transformation laws of $\theta_{\alpha,\beta}$, in particular the formulae of their characters $\chi_{\alpha,\beta}$.
\end{proof}
\begin{rema}
When $d=0$, we must have $N=1$. So $\underline{L}$ is an even unimodular lattice. The value of $\chi_{\alpha,\beta}$ at such $\eleactgrpnsimple{a}{b}{c}{d}{\varepsilon}{v}{w}{\xi}$ can also be evaluated by \eqref{eq:proofThetaJacobiForm}, while we do not need this.
\end{rema}
\begin{coro}
\label{coro:thetaNtransformation}
The $N^\dimn$ functions $\theta_{\alpha,\beta}^N$ all lie in the same space
\begin{equation*}
J_{\frac{N\dimn}{2}, \underline{L}_N}(\Gamma_1(N),\chi^N),
\end{equation*}
where $\underline{L}_N=(L,\, N\cdot B)$, the $\numZ$-module $L$ equipped with another bilinear form $(v,w) \mapsto N\cdot B(v,w)$, and $\chi^N$ is the linear character on $\widetilde{\Gamma_1(N)} \ltimes H(\underline{L}_N,1)$ that maps $\eleactgrpnsimple{a}{b}{c}{d}{\varepsilon}{v}{w}{\xi}$ (with $d \neq 0$) to
\begin{equation*}
\varepsilon^{-N\dimn}\xi(-\rmi)^{\frac{1}{2}N\dimn(1-\sgn{d})\legendre{c}{-1}}\cdot\mathfrak{g}_{\underline{L}_N}(b,d;0)\etp{\frac{1}{2}N\cdot B(v,w)}.
\end{equation*}
\end{coro}
\begin{proof}
This follows immediately from Proposition \ref{prop:thetaAlphaBetaJacobiForm} and Proposition \ref{prop:multiplicationScalarJacobiForm}.
\end{proof}
Some cautions must be mentioned. Here we have two bilinear forms $B$ and $N\cdot B$, so concepts depending on the bilinear form must be carefully treated. For instance, the slash operators $\vert_{\dimn/2, B}$ and $\vert_{N\dimn/2, N\cdot B}$ are different. For another instance, the Gauss sum occurring in the above corollary is
\begin{equation*}
\mathfrak{g}_{\underline{L}_N}(b,d;0)=\mathfrak{g}_{\underline{L}}(b,d;0)^N=\abs{d}^{-\frac{N\dimn}{2}}\sum_{v \in L/dL}\etp{\frac{bNQ(v)}{d}}.
\end{equation*}

We also state a slight generalization of Corollary \ref{coro:thetaNtransformation} below, whose proof is nearly the same as that of Corollary \ref{coro:thetaNtransformation}.

\begin{prop}
\label{prop:thetaNptransformation}
Let $N'$ be a positive divisor of the level $N$. Then the functions $\theta_{\alpha,\beta}^{N'}$ with $\alpha,\, \beta$ satisfying $Q(\alpha) \in \frac{1}{N'}\numZ$ and $\beta \in (L^\sharp+\frac{1}{N'}L)/L^\sharp$ all lie in the same space
\begin{equation*}
J_{\frac{N'\dimn}{2}, \underline{L}_{N'}}(\Gamma_1(N),\chi^{N'}),
\end{equation*}
where $\underline{L}_{N'}=(L,\, N'\cdot B)$, and $\chi^{N'}$ is the linear character on $\widetilde{\Gamma_1(N)} \ltimes H(\underline{L}_{N'},1)$ that maps $\eleactgrpnsimple{a}{b}{c}{d}{\varepsilon}{v}{w}{\xi}$ (with $d \neq 0$) to
\begin{equation*}
\varepsilon^{-N'\dimn}\xi(-\rmi)^{\frac{1}{2}N'\dimn(1-\sgn{d})\legendre{c}{-1}}\cdot\mathfrak{g}_{\underline{L}_{N'}}(b,d;0)\etp{\frac{1}{2}N'\cdot B(v,w)}.
\end{equation*}
\end{prop}

Now we apply Theorem \ref{thm:mainEmbed} to functions $\theta_{\alpha,\beta}^{N'}$, to explore linear relations among them. We shall first give a general algorithm (Theorem \ref{thm:algorithmFindLinearRelations}), the main result of this section, and then present four examples concerned with low-dimensional root lattices, and one concerned with binary quadratic form.

\begin{thm}
\label{thm:algorithmFindLinearRelations}
Let $B$ be a positive definite symmetric bilinear form on the real space $V$ of dimension $\dimn$, and $L$ be an even integral lattice in $\underline{V}=(V,B)$. Let $\theta_{\alpha,\beta}$ be theta series associated with the lattice $L$, as defined in Definition \ref{deff:thetaFunctionStudied}. Let $N$ be the level of $\underline{L}=(L,B)$, $N'$ be a positive divisor of $N$, and put $\underline{L}_{N'}=(L, N'\cdot B)$. Let $\mathcal{P}_0=\{\mathbf{p}_1, \dots, \mathbf{p}_d\}$ be a finite subset of $\numgeq{Z}{0}^\dimn$ with $d=\abs{\underline{L}_{N'}^\sharp/\underline{L}_{N'}}$. Let the quantity $\delta_N$ equal $2$ if $N=1,2$ and equal $1$ if $N \geq 3$. Let $\mathfrak{I}$ be the set of all pairs $(\mathbf{p},n)$ with $\mathbf{p} \in \widehat{\mathcal{P}_0}$, and
\begin{equation}
\label{eq:thmAlgorithmFindLinearRelationsnRange}
n=0,\,1,\,2,\dots, \left[\frac{\delta_N N^2}{24}\left(\frac{N'\dimn}{2}+s(\mathbf{p})\right)\prod_{p \mid N}\left(1-\frac{1}{p^2}\right)\right].
\end{equation}
For any $(\alpha, \beta) \in L^\sharp/L\times (L^\sharp+\frac{1}{N'}L)/L^\sharp$ with $Q(\alpha) \in \frac{1}{N'}\numZ$, we associate the theta series $\theta_{\alpha,\beta}^{N'}$ with an $\mathfrak{I}$-indexed sequence $\Theta_{\alpha,\beta}\colon \mathfrak{I}\rightarrow \numC$ whose $(\mathbf{p},n)$-term is equal to
\begin{multline}
\label{eq:coeffDkpTheta}
\sum_{\twoscript{v_1,\dots, v_{N'} \in L}{Q(\alpha+v_1)+\dots+Q(\alpha+v_{N'})=n}}\etp{B(\beta,v_1+\dots +v_{N'})}\\
\times P_{N'\dimn/2,\mathbf{p}, N'\cdot G_\mathfrak{B}}\left(n,\left(\frac{(\alpha+v_1)+\dots (\alpha+v_{N'})}{N'}\right)_{\mathfrak{B}}\right),
\end{multline}
where $\mathfrak{B}$ is any $\numZ$-basis of $L$, $G_{\mathfrak{B}}$ is the Gram matrix of $B$ with respect to $\mathfrak{B}$, and $(v)_{\mathfrak{B}}$ means the coordinates $\dimn$-tuple of $v$ under $\mathfrak{B}$. If the function $\mathcal{F}_{\underline{L}_{N'}, \mathcal{P}_0}^\mathfrak{B}$ is not identically zero, then any linear relation among $\theta_{\alpha,\beta}^{N'}$'s holds if and only if the linear relation with the same coefficients among $\Theta_{\alpha,\beta}$'s holds.
\end{thm}
For the notations $V$, $\mathcal{V}$ and $Q$, see the first paragraph of Section \ref{sec:Jacobi forms of lattice index}. For $\mathcal{F}_{\underline{L}_{N'}, \mathcal{P}_0}^\mathfrak{B}$, see Definition \ref{deff:FLP0tau}. For $s(\mathbf{p})$, see the second paragraph of Section \ref{sec:Formal Laurent series which transform like modular forms}. For $\widehat{\mathcal{P}_0}$, see paragraphs before Theorem \ref{thm:mainEmbed}. The polynomial $P_{N'\dimn/2,\mathbf{p}, N'\cdot G_\mathfrak{B}}$ is defined in Definition \ref{deff:PkpM}. For the definition of the dual lattice $L^\sharp$, see the paragraph before Proposition \ref{prop:ComplexFourierSeries}.
\begin{proof}
Suppose $\mathcal{F}_{\underline{L}_{N'}, \mathcal{P}_0}$ is not identically zero. Let $\{a_{\alpha,\beta}\}$ be arbitrary (complex valued) $\{\alpha \in L^\sharp/L \colon Q(\alpha) \in \frac{1}{N'}\numZ\} \times (L^\sharp+\frac{1}{N'}L)/L^\sharp$-indexed sequence. By Theorem \ref{thm:mainEmbed} and Proposition \ref{prop:thetaNptransformation}, the linear relation
\begin{equation}
\label{eq:linearRelationTheta}
\sum_{\alpha,\beta}a_{\alpha,\beta}\theta_{\alpha,\beta}^{N'}=0
\end{equation}
holds, if and only if for any $\mathbf{p} \in \widehat{\mathcal{P}_0}$, the linear relation
\begin{equation*}
\sum_{\alpha,\beta}a_{\alpha,\beta}D_{N'\dimn/2,\mathbf{p}}(\theta_{\alpha,\beta}^{N'})=0
\end{equation*}
holds. Thus, it is necessary to write down the Fourier development of $D_{N'\dimn/2,\mathbf{p}}(\theta_{\alpha,\beta}^{N'})$. By Definition \ref{deff:thetaFunctionStudied}, we have
\begin{equation*}
\theta_{\alpha,\beta}^{N'}(\tau,z)=\sum_{n,t}c_{\alpha,\beta}(n,t)q^n\etp{N'\cdot B(t,z)},
\end{equation*}
where $(\tau,z) \in \mydom$, $q=\etp{\tau}$, and
\begin{equation*}
c_{\alpha,\beta}(n,t)=\sum_{\threescript{v_1,\dots,v_{N'} \in L}{Q(\alpha+v_1)+\dots+Q(\alpha+v_{N'})=n}{(\alpha+v_1)+\dots+(\alpha+v_{N'})=N't}}\etp{B(\beta,v_1+\dots+v_{N'})}.
\end{equation*}
It follows from this and Proposition \ref{prop:FourierCoeffDkp} that, $D_{N'\dimn/2,\mathbf{p}}(\theta_{\alpha,\beta}^{N'})$ is, up to a factor not depending on $\alpha,\beta$,
\begin{equation*}
\sum_{n}\left(\sum_{t}c_{\alpha,\beta}(n,t)P_{N'\dimn/2,\mathbf{p},N'\cdot G_{\mathfrak{B}}}(n,(t)_{\mathfrak{B}})\right)q^n.
\end{equation*}
One can verify immediately that the coefficient of the $q^n$-term in the above sum is equal to \eqref{eq:coeffDkpTheta}, which is denoted by $\Theta_{\alpha,\beta}(\mathbf{p},n)$ from now on. Therefore, the relation \eqref{eq:linearRelationTheta} holds, if and only if
\begin{equation}
\label{eq:linearRelationBigTheta}
\sum_{\alpha,\beta}a_{\alpha,\beta}\Theta_{\alpha,\beta}(\mathbf{p},n)=0
\end{equation}
for any $\mathbf{p} \in \widehat{\mathcal{P}_0}$ and any $n \in \numQ$. It remains to show that, we can replace the condition ``$n \in \numQ$'' with ``$n$ belonging to the numbers given by \eqref{eq:thmAlgorithmFindLinearRelationsnRange}'' in the above statement. On the one hand, since
\begin{equation*}
\theta_{\alpha, \beta}^{N'}\vert_{N'\dimn/2, N'B}\widetilde{\tbtmat{1}{1}{0}{1}}=\etp{N'Q(\alpha)}\theta_{\alpha, \beta}=\theta_{\alpha, \beta},
\end{equation*}
we have $c_{\alpha,\beta}(n,t)=0$ unless $n \in \numgeq{Z}{0}$, and so \eqref{eq:linearRelationBigTheta} holds automatically for $n \notin \numgeq{Z}{0}$. On the other hand, it follows from the fact
\begin{equation*}
[\pslZ \colon \overline{\Gamma_1(N)}]=\frac{\delta_N N^2}{2}\prod_{p \mid N}\left(1-\frac{1}{p^2}\right)
\end{equation*}
(see \cite[Corollary 6.2.13]{CoS17}), and the valence formula (see \cite[Theorem 2.1]{ZZ21}) that, for fixed $\mathbf{p}$, the relation \eqref{eq:linearRelationBigTheta} holds for all $n \in \numgeq{Z}{0}$, if and only if it holds for all $n$ belonging to the numbers given by \eqref{eq:thmAlgorithmFindLinearRelationsnRange}. This concludes the proof.
\end{proof}
In principle, for any specific even integral lattice, one can find out a maximal linearly independent set of functions among $\theta_{\alpha,\beta}^N$'s, and all linear relations among them, by analysing the corresponding vectors $\Theta_{\alpha,\beta}$ using a computer algebra system, provided that we can first find a $\mathcal{P}_0$ such that $\mathcal{F}_{\underline{L}_N, \mathcal{P}_0} \neq 0$. We use SageMath 9.3 \cite{Sage} here.

Now consider specific lattices. First recall the concept of root lattices. The ambient real space is $\underline{V}=(V,B)$, as above. Let $L$ be an integral lattice in $\underline{V}$. We say $L$ is a \emph{root lattice} in $\underline{V}$, if there is a $\numZ$-basis of $L$ containing only vectors $v$ such that $Q(v)=1$. See \cite[\S 1.4]{Ebe13} for more details, or \cite[Appendix of \S 18]{TY05} for the classification and a full description of all irreducible root systems. %\footnote{The basic relationship of root lattices and root systems are: The roots of a root lattice, that is, the vectors with $Q(v)=1$, is a root system. Moreover, two root lattices are isometrically isomorphic, if and only if the corresponding root systems are isomorphic, in the sense that there is a bijection between these two root systems that preserves the bilinear form. So, a classification of root systems leads to a classification of root lattices.}.

In the remaining of this section, we always put $V=\numR^\dimn$, hence $\mathcal{V}=\numC^\dimn$. 

\subsection{The root lattice $D_4$}
\label{subsec:The root lattice D4}
One model of $D_4$ is the lattice $\numZ^4$ in $(\numR^4, B)$, where $B$ is the bilinear form given by the Gram matrix
\begin{equation*}
\begin{pmatrix}
2 & -1 & 0 & 0 \\
-1 & 2 & -1 & -1 \\
0 & -1 & 2 & 0 \\
0 & -1 & 0 & 2
\end{pmatrix}
\end{equation*}
under the standard basis of $\numR^4$. The determinant is $4$, and the level is $2$. We fix a set of representatives of $D_4^\sharp/D_4$ as follows:
\begin{equation}
\label{eq:repD4alpha}
(0,0,0,0),\quad (0,0,1/2,1/2),\quad (1/2,0,0,1/2)\quad (1/2,0,1/2,0),
\end{equation}
and that of $\frac{1}{2}D_4/D_4^\sharp$ as follows:
\begin{equation}
\label{eq:repD4beta}
(0,0,0,0),\quad (1/2,1/2,1/2,1/2),\quad (1/2,0,0,0)\quad (0,1/2,0,0).
\end{equation}
There are $16$ theta series $\theta_{\alpha,\beta}$ associated with $D_4$, where $\alpha$ is chosen from $\eqref{eq:repD4alpha}$, and $\beta$ is chosen from \eqref{eq:repD4beta}. They are all defined on $\uhp\times\numC^4$, and the second power of them all lie in the space $J_{4, \underline{D_4}_2}(\Gamma_1(2),\chi^2)$, where the lattice $\underline{D_4}_2$ is $(\numZ^4, 2\cdot B)$. See Definition \ref{deff:thetaFunctionStudied} and Corollary \ref{coro:thetaNtransformation} for details. To give insights into these functions, we write out one of them explicitly, as a Laurent series of $q=\etp{\tau}$ and
\begin{align*}
\zeta_1&=\etp{2z_1-z_2}, & \zeta_2&=\etp{-z_1+2z_2-z_3-z_4},\\
\zeta_3&=\etp{-z_2+2z_3}, & \zeta_4&=\etp{-z_2+2z_4}.
\end{align*}
Namely, let $\alpha=(1/2,0,1/2,0)$ and $\beta=(0,1/2,0,0)$; then
\begin{multline*}
\theta_{\alpha,\beta}(\tau,(z_1,z_2,z_3,z_4))=\sum_{v_1,v_2,v_3,v_4 \in \numZ}(-1)^{v_1+v_3+v_4}\\
\times q^{(v_1+1/2)^2+v_2^2+(v_3+1/2)^2+v_4^2-(v_1+1/2)v_2-v_2(v_3+1/2)-v_2v_4}\zeta_1^{v_1+1/2}\zeta_2^{v_2}\zeta_3^{v_3+1/2}\zeta_4^{v_4}.
\end{multline*}
One can plug in $4\tau$ for $\tau$, and $(2z_1,2z_2,2z_3,2z_4)$ for $(z_1,z_2,z_3,z_4)$ to avoid fractional powers.

To apply Theorem \ref{thm:algorithmFindLinearRelations} to functions $\theta_{\alpha,\beta}^2$, we must first find a set $\mathcal{P}_0 \subset \numgeq{Z}{0}^4$ of $64$ elements such that $\mathcal{F}_{\underline{D_4}_2,\mathcal{P}_0}$ is not identically zero. (The basis $\mathfrak{B}$ is the standard basis of $\numR^4$.) To do this, note that a set of representatives of $\underline{D_4}_2^\sharp/\underline{D_4}_2$ can be obtained from one of $\underline{D_4}^\sharp/\underline{D_4}$ as follows
\begin{equation*}
\frac{1}{2}t+\frac{1}{2}(v_1, v_2, v_3, v_4),
\end{equation*}
where $t$ ranges over vectors in \eqref{eq:repD4alpha}, and $v_1$, $v_2$, $v_3$, $v_4$ range over $\{0,1\}$. By a SageMath program, we find that the set $\mathcal{P}_0$ consisting of following vectors makes the matrix \eqref{eq:matBp} non-singular:
\begin{equation*}
\begin{matrix}
(0, 0, 0, 0) & (0, 0, 0, 1) & (0, 0, 1, 0) & (0, 1, 0, 0) & (1, 0, 0, 0) & (0, 0, 0, 2)\\
(0, 0, 1, 1) & (0, 0, 2, 0) & (0, 1, 0, 1) & (0, 1, 1, 0) & (0, 2, 0, 0) & (1, 0, 0, 1)\\
(1, 0, 1, 0) & (1, 1, 0, 0) & (2, 0, 0, 0) & (0, 0, 1, 2) & (0, 0, 2, 1) & (0, 1, 0, 2)\\
(0, 1, 2, 0) & (1, 0, 0, 2) & (1, 0, 2, 0) & (1, 1, 0, 1) & (1, 1, 1, 0) & (0, 0, 0, 4)\\
(0, 0, 2, 2) & (0, 0, 4, 0) & (0, 1, 1, 2) & (0, 1, 2, 1) & (0, 2, 0, 2) & (0, 2, 1, 1)\\
(1, 0, 1, 2) & (1, 0, 2, 1) & (1, 1, 0, 2) & (1, 1, 2, 0) & (2, 0, 0, 2) & (2, 0, 1, 1)\\
(2, 2, 0, 0) & (0, 0, 1, 4) & (0, 0, 4, 1) & (0, 1, 0, 4) & (0, 1, 2, 2) & (1, 0, 0, 4)\\
(1, 0, 2, 2) & (1, 1, 1, 2) & (1, 1, 2, 1) & (0, 0, 2, 4) & (0, 0, 4, 2) & (0, 1, 1, 4)\\
(0, 1, 4, 1) & (0, 2, 0, 4) & (1, 0, 1, 4) & (1, 0, 4, 1) & (1, 1, 0, 4) & (1, 1, 2, 2)\\
(2, 0, 0, 4) & (2, 2, 0, 2) & (2, 2, 1, 1) & (0, 1, 2, 4) & (1, 0, 2, 4) & (1, 1, 1, 4)\\
(1, 1, 4, 1) & (0, 0, 4, 4) & (1, 1, 2, 4) & (2, 2, 0, 4).
\end{matrix}
\end{equation*}
Hence $\mathcal{F}_{\underline{D_4}_2,\mathcal{P}_0} \neq 0$, and Theorem \ref{thm:algorithmFindLinearRelations} is applicable\footnote{This $\mathcal{P}_0$ may not be the simplest, that is, $\abs{\widehat{\mathcal{P}_0}}$ may not be the least, since there may be some $\mathcal{P}_0'$ such that the matrix \eqref{eq:matBp} is singular, but the function $\mathcal{F}_{\underline{D_4}_2,\mathcal{P}_0'}$ is not identically zero.} for this $\mathcal{P}_0$.

The set $\widehat{\mathcal{P}_0}$ now has cardinality $94$, and the set $\mathfrak{I}$ in Theorem \ref{thm:algorithmFindLinearRelations} has cardinality $258$. There are totally $94$ polynomials $P_{4,\mathbf{p},2\cdot G_{\mathfrak{B}}}$'s to calculate first. By the SageMath program, we work out the vectors $\Theta_{\alpha,\beta}$ (whose components in this case, are all integers) and find out a maximal linearly independent set and all linear relations, as follows.
\begin{thm}
\label{thm:D4linearRelation}
For the $D_4$ lattice, the seven functions $\theta_{\alpha,0}^2$, $\theta_{0,\beta}^2$ with $\alpha$ in \eqref{eq:repD4alpha} and $\beta$ in \eqref{eq:repD4beta} are linearly independent. Moreover, for any $\alpha \in D_4^\sharp$ and $\beta \in \frac{1}{2}\numZ^4$, we have
\begin{equation*}
\theta_{\alpha,\beta}^2+\theta_{0,0}^2=\theta_{\alpha,0}^2+\theta_{0,\beta}^2.
\end{equation*}
\end{thm}
The program is completely presented in Appendix \ref{apx:SageMath programs}.

\subsection{The root lattice $A_2$}
\label{subsec:The root lattice A2}
One model of $A_2$ is the lattice $\numZ^2$ in $(\numR^2, B)$, where $B$ is the bilinear form given by the Gram matrix
\begin{equation*}
\begin{pmatrix}
2 & -1\\
-1 & 2
\end{pmatrix}
\end{equation*}
under the standard basis of $\numR^2$. The determinant is $3$, and the level is $3$. We fix a set of representatives of $A_2^\sharp/A_2$ as follows:
\begin{equation}
\label{eq:repA2alpha}
(0,0),\quad (2/3,1/3),\quad (1/3,2/3),
\end{equation}
and that of $\frac{1}{3}A_2/A_2^\sharp$ as follows:
\begin{equation}
\label{eq:repA2beta}
(0,0),\quad (1/3,1/3),\quad (2/3,2/3).
\end{equation}
There are nine theta series $\theta_{\alpha,\beta}$ associated with $A_2$, where $\alpha$ is chosen from $\eqref{eq:repA2alpha}$, and $\beta$ is chosen from \eqref{eq:repA2beta}. They are all defined on $\uhp\times\numC^2$, and the third power of them all lie in the space $J_{3, \underline{A_2}_3}(\Gamma_1(3),\chi^3)$, where the lattice $\underline{A_2}_3$ is $(\numZ^2, 3\cdot B)$. See Definition \ref{deff:thetaFunctionStudied} and Corollary \ref{coro:thetaNtransformation} for details.

The linear relations among the third power of these nine functions are known by D. Schultz. In a systematic treatment of cubic theta functions \cite{Sch13}, he defined nine theta functions, $\Theta_{ij}(\mathbf{u} \mid \tau)$ in his notation, in Section 2.3. One may check that, Schultz's functions are essentially functions $\theta_{\alpha,\beta}$ associated with $A_2$ defined here, up to some trivial changes of variables. So the following theorem should be attributed to Schultz. We only provide an alternate proof here, based on Theorem \ref{thm:algorithmFindLinearRelations} and the aid of any computer algebra system.
\begin{thm}
\label{thm:A2linearRelation}
For the $A_2$ lattice, the five functions $\theta_{\alpha,0}^3$, $\theta_{0,\beta}^3$ with $\alpha$ in \eqref{eq:repA2alpha} and $\beta$ in \eqref{eq:repA2beta} are linearly independent. Moreover, for any $\alpha \in A_2^\sharp$ and $\beta \in \frac{1}{3}\numZ^2$, we have
\begin{equation*}
\theta_{\alpha,\beta}^3+\theta_{0,0}^3=\theta_{\alpha,0}^3+\theta_{0,\beta}^3.
\end{equation*}
\end{thm}
\begin{proof}
By Theorem \ref{thm:algorithmFindLinearRelations}, we shall calculate nine vectors $\Theta_{\alpha,\beta}$ with components \eqref{eq:coeffDkpTheta}. This is done by the SageMath program as in the last subsection. For the $A_2$ case, we may choose the set $\mathcal{P}_0$ consisting of following vectors:
\begin{equation*}
\begin{matrix}
(0, 0) & (0, 1) & (1, 0) & (0, 2) & (1, 1) & (2, 0) & (0, 3) & (1, 2) & (2, 1)\\
(3, 0) & (0, 4) & (1, 3) & (2, 2) & (3, 1) & (4, 0) & (1, 4) & (2, 3) & (3, 2)\\
(0, 6) & (2, 4) & (3, 3) & (4, 2) & (1, 6) & (3, 4) & (4, 3) & (2, 6) & (3, 6).\\
\end{matrix}
\end{equation*}
This $\mathcal{P}_0$ makes the matrix \eqref{eq:matBp} non-singular, hence the function $\mathcal{F}_{\underline{A_2}_3,\mathcal{P}_0}$ not identically zero.
The set $\widehat{\mathcal{P}_0}$ has cardinality $31$, and the set $\mathfrak{I}$ in Theorem \ref{thm:algorithmFindLinearRelations} has cardinality $99$. There are totally $31$ polynomials $P_{3,\mathbf{p},3\cdot G_{\mathfrak{B}}}$'s to calculate first. Note that these vectors $\Theta_{\alpha,\beta}$ lie in a vector space over the quadratic field $\numQ[\sqrt{-3}]$ (of dimension $99$), so we can find out a maximal linearly independent set, and all linear relations among them, by standard linear algebra algorithms.
\end{proof}

\subsection{The root lattice $A_3$}
\label{subsec:The root lattice A3}
One model of $A_3$ is the lattice $\numZ^3$ in $(\numR^3, B)$, where $B$ is the bilinear form given by the Gram matrix
\begin{equation*}
\begin{pmatrix}
2 & -1 & 0 \\
-1 & 2 & -1 \\
0 & -1 & 2
\end{pmatrix}
\end{equation*}
under the standard basis of $\numR^3$. The determinant is $4$, and the level is $8$. We fix a set of representatives of $A_3^\sharp/A_3$ as follows:
\begin{equation}
\label{eq:repA3alpha}
(0,0,0),\quad (1/2,0,1/2),\quad (3/4,1/2,1/4),\quad (1/4,1/2,3/4)
\end{equation}
and that of $\frac{1}{8}A_3/A_3^\sharp$ as follows:
\begin{equation}
\label{eq:repA3beta}
(v_1/8, v_2/8, v_3/8),\,v_1,v_2=0,1,2,\dots,7,\,v_3=0,1.
\end{equation}
There are 512 theta series $\theta_{\alpha,\beta}$ associated with $A_3$, where $\alpha$ is chosen from $\eqref{eq:repA3alpha}$, and $\beta$ is chosen from \eqref{eq:repA3beta}. They are all defined on $\uhp\times\numC^3$, and the eighth power of them all lie in the space $J_{12, \underline{A_3}_8}(\Gamma_1(8),\chi^8)$, where the lattice $\underline{A_3}_8$ is $(\numZ^3, 8\cdot B)$. See Definition \ref{deff:thetaFunctionStudied} and Corollary \ref{coro:thetaNtransformation} for details.

Our program based on Theorem \ref{thm:algorithmFindLinearRelations} is incapable of finding out and proving linear relations among $\theta_{\alpha,\beta}^8$'s (even among $\theta_{\alpha,\beta}^4$'s) within a reasonable period of time. Hence we present linear relations among $\theta_{\alpha,\beta}^2$'s below. By Theorem \ref{thm:algorithmFindLinearRelations}, $\alpha$ and $\beta$ should satisfy $Q(\alpha) \in \frac{1}{2}\numZ$ and $\beta \in (A_3^\sharp +\frac{1}{2}A_3)/A_3^\sharp$. We choose such $\alpha$ and $\beta$ from \eqref{eq:repA3alpha} and \eqref{eq:repA3beta}, obtaining following representatives respectively:
\begin{equation}
\label{eq:repA3alpha2}
(0,0,0),\quad (1/2,0,1/2)
\end{equation}
 for $\alpha$ and
\begin{equation}
\label{eq:repA3beta2}
(0,0,0),\,(0,1/2,0),\,(1/2,0,0),\,(1/2,1/2,0)
\end{equation}
for $\beta$. All eight functions $\theta_{\alpha, \beta}^2$ with $\alpha$ in \eqref{eq:repA3alpha2} and $\beta$ in \eqref{eq:repA3beta2} lie in the space $J_{3, \underline{A_3}_2}(\Gamma_1(8),\chi^3)$ (Proposition \ref{prop:thetaNptransformation}).

To apply Theorem \ref{thm:algorithmFindLinearRelations} to functions $\theta_{\alpha,\beta}^2$, we first find a set $\mathcal{P}_0 \subset \numgeq{Z}{0}^3$ of $32$ elements such that $\mathcal{F}_{\underline{A_3}_2,\mathcal{P}_0}$ is not identically zero. (The basis $\mathfrak{B}$ is the standard basis of $\numR^3$.) To do this, note that a set of representatives of $\underline{A_3}_2^\sharp/\underline{A_3}_2$ can be obtained from one of $\underline{A_3}^\sharp/\underline{A_3}$ as follows
\begin{equation*}
\frac{1}{2}t+\frac{1}{2}(v_1, v_2, v_3),
\end{equation*}
where $t$ ranges over vectors in \eqref{eq:repA3alpha}, and $v_1$, $v_2$, $v_3$ range over $\{0,1\}$. By a SageMath program, we find that the set $\mathcal{P}_0$ consisting of following vectors makes the matrix \eqref{eq:matBp} non-singular:
\begin{equation*}
\begin{matrix}
(0, 0, 0) & (0, 0, 1) & (0, 1, 0) & (1, 0, 0) & (0, 0, 2) & (0, 1, 1) & (0, 2, 0) & (1, 0, 1)\\
(1, 1, 0) & (2, 0, 0) & (0, 1, 2) & (1, 0, 2) & (1, 1, 1) & (1, 2, 0) & (0, 0, 4) & (0, 2, 2)\\
(0, 3, 1) & (1, 1, 2) & (2, 0, 2) & (2, 1, 1) & (2, 2, 0) & (0, 1, 4) & (1, 0, 4) & (1, 2, 2)\\
(1, 3, 1) & (0, 2, 4) & (1, 1, 4) & (2, 0, 4) & (2, 2, 2) & (2, 3, 1) & (1, 2, 4) & (2, 2, 4).
\end{matrix}
\end{equation*}
Hence $\mathcal{F}_{\underline{A_3}_2,\mathcal{P}_0} \neq 0$, and Theorem \ref{thm:algorithmFindLinearRelations} is applicable for this $P_0$.

The set $\widehat{\mathcal{P}_0}$ now has cardinality $43$, and the set $\mathfrak{I}$ in Theorem \ref{thm:algorithmFindLinearRelations} has cardinality $635$. There are totally $43$ polynomials $P_{3,\mathbf{p},2\cdot G_{\mathfrak{B}}}$'s to calculate first. By the SageMath program, we work out the vectors $\Theta_{\alpha,\beta}$ (whose components in this case, are all integers) and find out a maximal linearly independent set and all linear relations, as follows.
\begin{thm}
\label{thm:A3linearRelation2}
For the $A_3$ lattice, the five functions $\theta_{\alpha,0}^2$, $\theta_{0,\beta}^2$ with $\alpha$ in \eqref{eq:repA3alpha2} and $\beta$ in \eqref{eq:repA3beta2} are linearly independent. Moreover, for any $\alpha$ in \eqref{eq:repA3alpha2} and $\beta$ in \eqref{eq:repA3beta2}, we have
\begin{equation*}
\theta_{\alpha,\beta}^2+\theta_{0,0}^2=\theta_{\alpha,0}^2+\theta_{0,\beta}^2.
\end{equation*}
\end{thm}

\subsection{The root lattice $A_2\oplus A_2$}
\label{subsec:The root lattice A2A2}
The aim of this subsection is to present a concrete and non-trivial example, showing that when $\mathcal{P}_0$ is not sufficiently large, then the map $\prod_{\mathbf{p} \in \widehat{\mathcal{P}_0}}D_{k,\mathbf{p}}$ in Theorem \ref{thm:mainEmbed} may not be injective.

One model of $A_2\oplus A_2$ is the lattice $\numZ^4$ in $(\numR^4, B)$, where $B$ is the bilinear form given by the Gram matrix
\begin{equation*}
\begin{pmatrix}
2 & -1 & 0 & 0\\
-1 & 2 & 0 & 0\\
0 & 0 & 2 & -1\\
0 & 0 & -1 & 2
\end{pmatrix}
\end{equation*}
under the standard basis of $\numR^4$. The determinant is $9$, and the level is $3$. For simplicity, let $2A_2$ denote this root lattice. We fix a set of representatives of $(2A_2)^\sharp/2A_2$ as follows:
\begin{align}
\label{eq:rep2A2alpha}
(0,0,0,0)&, &(0,0,2/3,1/3)&, &(0,0,1/3,2/3),\\
(2/3,1/3,0,0)&, &(2/3,1/3,2/3,1/3)&, &(2/3,1/3,1/3,2/3),\notag\\
(1/3,2/3,0,0)&, &(1/3,2/3,2/3,1/3)&, &(1/3,2/3,1/3,2/3)\notag
\end{align}
and that of $\frac{1}{3}(2A_2)/(2A_2)^\sharp$ as follows:
\begin{align}
\label{eq:rep2A2beta}
(0,0,0,0)&, &(0,0,1/3,1/3)&, &(0,0,2/3,2/3),\\
(1/3,1/3,0,0)&, &(1/3,1/3,1/3,1/3)&, &(1/3,1/3,2/3,2/3),\notag\\
(2/3,2/3,0,0)&, &(2/3,2/3,1/3,1/3)&, &(2/3,2/3,2/3,2/3).\notag
\end{align}
There are 81 theta series $\theta_{\alpha,\beta}$ associated with $2A_2$, where $\alpha$ is chosen from $\eqref{eq:rep2A2alpha}$, and $\beta$ is chosen from \eqref{eq:rep2A2beta}. They are all defined on $\uhp\times\numC^4$, and the third power of them all lie in the space $J_{6, \underline{2A_2}_3}(\Gamma_1(3),\chi^3)$, where the lattice $\underline{2A_2}_3$ is $(\numZ^4, 3\cdot B)$. See Definition \ref{deff:thetaFunctionStudied} and Corollary \ref{coro:thetaNtransformation} for details. To make formulae concise, put $\theta_{ij}=\theta_{\alpha_i,\beta_j}$ where $\alpha_i$ is the $(i+1)$-th item in \eqref{eq:rep2A2alpha} and $\beta_j$ is the $(j+1)$-th item in \eqref{eq:rep2A2beta} ($i,j=0,1,\dots,8$).
\begin{examp}
Put
\begin{equation*}
f=\theta_{00}^3 - \theta_{01}^3 - \theta_{03}^3 + \theta_{04}^3 - \theta_{10}^3 - \theta_{30}^3 + \theta_{40}^3 + \theta_{13}^3 + \theta_{31}^3 - \theta_{88}^3.
\end{equation*}
Then $D_{6, \mathbf{p}}f=0$ for any $\mathbf{p} \in \numgeq{Z}{0}^4$ with $s(\mathbf{p}) \leq 6$, but $D_{6, (0, 3, 0, 4)}f \neq 0$. This implies that the map $\prod_{\mathbf{p} \in \widehat{\mathcal{P}_0}}D_{k,\mathbf{p}}$ with $\mathcal{P}_0=\{\mathbf{p} \in \numgeq{Z}{0}^4 \colon s(\mathbf{p}) \leq 6)\}$ is not injective, although $\abs{\mathcal{P}_0}=210$.
\end{examp}

\subsection{The principal binary form of discriminant $-15$}
\label{subsec:The principal binary form of discriminant -15}
The purpose of this subsection is to show with an example that, it seems that there is no linear relation among the powers of certain functions $\theta_{\alpha,\beta}$ associated with a lattice that is not root lattice.

There are two classes of binary quadratic forms of discriminant $-15$, one is $x^2+xy+4y^2$ and the other is $2x^2+xy+2y^2$. This subsection is devoted to the former form. In language of lattices, we deal with $L=\numZ^2$ in $(\numR^2, B)$, where $B$ is the bilinear form given by the Gram matrix
\begin{equation*}
\begin{pmatrix}
2 & 1  \\
1 & 8
\end{pmatrix}
\end{equation*}
under the standard basis of $\numR^2$. The determinant and level are both $15$. We fix a set of representatives of $L^\sharp/L$ as follows:
\begin{equation}
\label{eq:repLat2118alpha}
\begin{matrix}
(0,0) & (1/15,13/15) & (2/15,11/15) & (3/15, 9/15) & (4/15, 7/15) \\
(5/15, 5/15) & (6/15, 3/15) & (7/15, 1/15) & (8/15, 14/15) & (9/15, 12/15) \\
(10/15, 10/15) & (11/15, 8/15) & (12/15, 6/15) & (13/15, 4/15) & (14/15, 2/15),
\end{matrix}
\end{equation}
and that of $\frac{1}{15}L/L^\sharp$ as follows:
\begin{equation}
\label{eq:repLat2118beta}
(v/15, 0),\quad v=0,1,2,\dots,14.
\end{equation}
We shall show the third powers $\theta_{\alpha, \beta}^3$, where $\alpha$ and $\beta$ satisfy $Q(\alpha) \in \frac{1}{3}\numZ$ and $\beta \in (L^\sharp+\frac{1}{3}L)/L^\sharp$, are $\numC$-linearly independent. (For definitions, see Definition \ref{deff:thetaFunctionStudied}.) We choose such $\alpha$ and $\beta$ from \eqref{eq:repLat2118alpha} and \eqref{eq:repLat2118beta} respectively, obtaining following sets of representatives:
\begin{align}
\alpha &\colon \quad (0,0) \quad (1/3,1/3) \quad (2/3,2/3), \label{eq:repLat2118alpha3}\\
\beta  &\colon \quad (0,0) \quad (1/3,0) \quad (2/3,0). \label{eq:repLat2118beta3}
\end{align}

To apply Theorem \ref{thm:algorithmFindLinearRelations} to functions $\theta_{\alpha,\beta}^3$, we first find a set $\mathcal{P}_0 \subset \numgeq{Z}{0}^2$ of $135$ elements such that $\mathcal{F}_{\underline{L}_3,\mathcal{P}_0}$ is not identically zero. (The basis $\mathfrak{B}$ is the standard basis of $\numR^2$.) To do this, note that a set of representatives of $\underline{L}_3^\sharp/\underline{L}_3$ can be obtained from one of $\underline{L}^\sharp/\underline{L}$ as follows
\begin{equation*}
\frac{1}{3}t+\frac{1}{3}(v_1, v_2),
\end{equation*}
where $t$ ranges over vectors in \eqref{eq:repLat2118alpha}, and $v_1$, $v_2$ range over $\{0,1,2\}$. By a SageMath program, we find out a set $\mathcal{P}_0$ that makes the matrix \eqref{eq:matBp} non-singular, but we omit it for the sake of place. Hence $\mathcal{F}_{\underline{L}_3,\mathcal{P}_0} \neq 0$, and Theorem \ref{thm:algorithmFindLinearRelations} is applicable for this $P_0$.

The set $\widehat{\mathcal{P}_0}$ now has cardinality $139$, and the set $\mathfrak{I}$ in Theorem \ref{thm:algorithmFindLinearRelations} has cardinality $18563$. It is unnecessary to work out the vectors $\Theta_{\alpha,\beta}$ completely, since we are proving the linear independence, not finding linear relations. It turns out that, the nine subsequences $\Theta_{\alpha,\beta}\vert_{\mathfrak{I}'}$, where
\begin{equation*}
\mathfrak{I}'=\{(\mathbf{p},n) \in \mathfrak{I} \colon n=0,\,1,\,2\},
\end{equation*}
are $\numC$-linearly independent. (These subsequences can be calculated by our SageMath program very fast, while those original sequences can not.) Hence, $\Theta_{\alpha,\beta}$'s themselves are linearly independent. So according to Theorem \ref{thm:algorithmFindLinearRelations}, we have proved
\begin{prop}
The nine functions $\theta_{\alpha,\beta}^3$ associated with the quadratic form $x^2+xy+4y^2$ are $\numC$-linearly independent, where $\alpha$ and $\beta$ are chosen from \eqref{eq:repLat2118alpha3} and \eqref{eq:repLat2118beta3} respectively.
\end{prop}

Now consider fifth powers $\theta_{\alpha,\beta}^5$ with $Q(\alpha) \in \frac{1}{5}\numZ$ and $\beta \in (L^\sharp+\frac{1}{5}L)/L^\sharp$. We choose $\alpha$ and $\beta$ from \eqref{eq:repLat2118alpha} and \eqref{eq:repLat2118beta} respectively, obtaining following sets of representatives:
\begin{align}
\alpha &\colon \quad (0,0) \quad (1/5,3/5) \quad (2/5,1/5) \quad (3/5,4/5) \quad (4/5,2/5), \label{eq:repLat2118alpha5}\\
\beta  &\colon \quad (0,0) \quad (1/5,0) \quad (2/5,0) \quad (3/5,0) \quad (4/5,0). \label{eq:repLat2118beta5}
\end{align}
Similar to the case of third powers, we can prove the following:
\begin{prop}
The 25 functions $\theta_{\alpha,\beta}^5$ associated with the quadratic form $x^2+xy+4y^2$ are $\numC$-linearly independent, where $\alpha$ and $\beta$ are chosen from \eqref{eq:repLat2118alpha5} and \eqref{eq:repLat2118beta5} respectively.
\end{prop}

\section{Miscellaneous observations and open questions}
\label{sec:Miscellaneous observations and open questions}

\subsection{The case of odd lattices}
\label{subsec:The case of odd lattices}
Our main theorem, Theorem \ref{thm:mainEmbed}, is valid for both even lattices and odd lattices. In fact, we have made much effort to state many propositions general for both kinds of lattices. The first place where even and odd lattices behave differently is in Theorem \ref{thm:thetaSeriesLatticeJacobiForm}. And then in Theorem \ref{thm:thetaDecomposition}, the assertion for even lattices is stronger than that for odd lattices, while the weak version (that is, Remark \ref{rama:conditionOnRhoThetaDecompWeak}) for both kinds of lattices suffices in the proof of the main theorem.

However, the whole Section \ref{sec:Application: Linear relations among theta series}, in particular Theorem \ref{thm:algorithmFindLinearRelations}, concerns only even lattices. It is for the reason of simplicity. Actually, there is a version of Theorem \ref{thm:algorithmFindLinearRelations} that works for both cases. Here we briefly describe what this general version looks like.

First, we generalize Theorem \ref{thm:thetaSeriesLatticeJacobiForm} to allow odd lattices. Let $\Gamma_\theta$ be the subgroup\footnote{Equivalently, $\Gamma_\theta=\{\tbtmat{a}{b}{c}{d} \in \slZ \colon a \equiv d \pmod 2,\,b \equiv c \pmod 2\}$. It has index $3$ in $\slZ$. See \cite[Proposition 2.3.3(a)]{CoS17}.} of $\slZ$ generated by $\tbtmat{1}{2}{0}{1}$ and $\tbtmat{0}{-1}{1}{0}$. Then one can prove that
\begin{equation*}
\vartheta_{\underline{L}} \in \JacForm{\dimn/2}{L}{\Gamma_\theta}{\rho_{\underline{L}}},
\end{equation*}
where $L$ is any integral lattice in $\underline{V}=(V, B)$ and $\rho_{\underline{L}}$ is the group representation of $\Jacgrp{\Gamma_\theta}{L}$ defined by the same formulae used to define \eqref{eq:WeilRep}, except for the first formula, which should be replaced by
\begin{equation*}
\rho_{\underline{L}}\widetilde{\tbtmat{1}{2}{0}{1}}\delta_x = \etp{2Q(x)}\delta_x.
\end{equation*}
One may prove that these three formulae indeed can be extended to a representation of $\Jacgrp{\Gamma_\theta}{L}$ (by transformation laws for Jacobi theta series or by firstly finding a presentation of $\Gamma_\theta$). More general transformation formulae, \eqref{eq:thetaTransformationLambda}, still hold for odd lattices, provided that $\tbtmat{a}{b}{c}{d} \in \Gamma_{\theta}$ and $c \neq 0$. The proof is almost the same as that for even lattices. In this proof, factors like $\etp{acQ(v)}$ and $\etp{bdQ(v)}$ occur ($v \in L$). It is because $\tbtmat{a}{b}{c}{d} \in \Gamma_{\theta}$ that such factors all equal $1$. Another version of general transformation formulae, \eqref{eq:thetaTransformationMu}, again holds for odd lattices, provided that $\tbtmat{a}{b}{c}{d} \in \Gamma_{\theta}$ and $d \neq 0$. ($c=0$ is permitted.)

Next, consider Definition \ref{deff:thetaFunctionStudied}. It also works for odd lattices, but perhaps the least positive integer $N$ such that $N\cdot Q(v) \in \numZ$ for all $v \in L^\sharp$ is not commonly called ``level'' in the context of odd lattices, as we shall do. Lemma \ref{lemm:thetaLatticeTransformationN} can be immediately generalized to cover the case of odd lattices --- just replace the condition $\tbtmat{a}{b}{c}{d} \in \Gamma_0(N)$ by $\tbtmat{a}{b}{c}{d} \in \Gamma_0(N)\cap \Gamma_\theta$ and leave the remaining part unchanged. Proposition \ref{prop:thetaAlphaBetaJacobiForm}, Corollary \ref{coro:thetaNtransformation} and Proposition \ref{prop:thetaNptransformation} are based on Lemma \ref{lemm:thetaLatticeTransformationN}, so they can also be generalized to include the odd-lattice case by replacing the group, say $\Gamma_1(N)$, with $\Gamma_1(N) \cap \Gamma_\theta$. But there is some subtle difference in values of $\chi_{\alpha,\beta}$ in Proposition \ref{prop:thetaAlphaBetaJacobiForm}. For odd lattice $L$, the value of $\chi_{\alpha,\beta}$ at $\eleactgrpnsimple{a}{b}{c}{d}{\varepsilon}{v}{w}{\xi} \in \Jacgrp{\Gamma_1(N)\cap \Gamma_\theta}{L}$ (with $d \neq 0$) is
\begin{equation*}
\varepsilon^{-\dimn}\xi(-\rmi)^{\frac{1}{2}\dimn(1-\sgn{d})\legendre{c}{-1}}\cdot\mathfrak{g}(b,d;0)\etp{bQ(\alpha)-cdQ(\beta)+\frac{1}{2}B(v,w)-B(v,\beta)},
\end{equation*}
which is not totally same as that in Proposition \ref{prop:thetaAlphaBetaJacobiForm}. % We may simply consider only those $\beta$ satisfying $Q(N\beta) \in \numZ$, and drop other choices of $\beta$, to retain the formula in Proposition \ref{prop:thetaAlphaBetaJacobiForm}. Another way to retain this formula is to replace the group $\Gamma_1(N)$ by $\Gamma_1(2N)$, which has the advantage that we need not ignore some choices of $\beta$, but has the disadvantage that the program would take more time. In the following, we choose the former solution.
Moreover, we state the odd-lattice version of Proposition \ref{prop:thetaNptransformation} explicitly.
\begin{prop}
\label{prop:thetaNptransformationOdd}
Let $L$ be odd, $N'$ be a positive divisor of the ``level'' $N$ such that $N/N'$ is even. Then the functions $\theta_{\alpha,\beta}^{N'}$ with $\alpha,\, \beta$ satisfying $\alpha \in L^\sharp/L$, $Q(\alpha) \in \frac{1}{2N'}\numZ$ and $\beta \in (L^\sharp+\frac{1}{N'}L)/L^\sharp$ all lie in the same space
\begin{equation*}
J_{\frac{N'\dimn}{2}, \underline{L}_{N'}}(\Gamma_1(N)\cap \Gamma_\theta,\chi^{N'}),
\end{equation*}
where $\chi^{N'}$ is the linear character on $\widetilde{\Gamma_1(N)\cap \Gamma_\theta} \ltimes H(\underline{L}_{N'},1)$ that maps $\eleactgrpnsimple{a}{b}{c}{d}{\varepsilon}{v}{w}{\xi}$ (in this case $d$ must be nonzero) to
\begin{equation*}
\varepsilon^{-N'\dimn}\xi(-\rmi)^{\frac{1}{2}N'\dimn(1-\sgn{d})\legendre{c}{-1}}\cdot\mathfrak{g}_{\underline{L}_{N'}}(b,d;0)\etp{\frac{1}{2}N'\cdot B(v,w)}.
\end{equation*}
\end{prop}

Finally, we are now able to generalize Theorem \ref{thm:algorithmFindLinearRelations} to odd lattices. We shall make following changes:
\begin{enumerate}
\item The lattice $L$ is now assumed to be an odd lattice in $\underline{V}=(V,B)$.
\item Besides the condition $(\alpha, \beta) \in L^\sharp/L\times (L^\sharp+\frac{1}{N'}L)/L^\sharp$, we need another assumption $2 \mid N/N'$. We shall also replace the condition $Q(\alpha) \in \frac{1}{N'}\numZ$ by $Q(\alpha) \in \frac{1}{2N'}\numZ$.
\item The set $\mathfrak{I}$ should be replaced by the set of all pairs $(\mathbf{p},n)$ with $\mathbf{p} \in \widehat{\mathcal{P}_0}$, and
\begin{equation*}
%\label{eq:thmAlgorithmFindLinearRelationsnRangeOdd}
2n=0,\,1,\,2,\dots, \left[\frac{1}{6}m\left(\frac{N'\dimn}{2}+s(\mathbf{p})\right)\right],
\end{equation*}
where $m=[\pslZ \colon \overline{\Gamma_1(N) \cap \Gamma_\theta}]$.
\end{enumerate}
After such changes, Theorem \ref{thm:algorithmFindLinearRelations} still holds! This concludes the generalization we promise.

\subsection{The image of the embedding in the main theorem}
\label{subsec:The image of the embedding in the main theorem}
In Theorem \ref{thm:mainEmbed}, we establish embeddings from certain space of Jacobi forms to some direct product of finitely many spaces of modular forms. It is natural and interesting to describe images of such embeddings. In the context of classical Jacobi forms (integral-index, scalar-valued, integral-weight, and trivial character), Eichler and Zagier have found that the space of weak Jacobi forms is isomorphic to a direct product of finitely many spaces of modular forms (\cite[Theorem 9.2]{EZ85}). We wish to investigate the general case:
\begin{oq}
Use notations and assumptions in Theorem \ref{thm:mainEmbed}. We invite interested readers to describe images of spaces $\JacFormWHol{k}{L}{G}{\rho}$, $\JacFormWeak{k}{L}{G}{\rho}$, $\JacForm{k}{L}{G}{\rho}$ and $\JacFormCusp{k}{L}{G}{\rho}$ under the map in that theorem. In particular, we ask that, is the space $\JacFormWeak{k}{L}{G}{\rho}$ isomorphic to some product of finitely many spaces of modular forms?
\end{oq}

There is another interesting question concerned with these embeddings:
\begin{oq}
Use notations and assumptions in Theorem \ref{thm:mainEmbed}. What is the set $\widehat{\mathcal{P}_0}$ with the least number of elements, such that the map is injective? How can we describe such set theoretically, not computationally?
\end{oq}
In this direction, for the integral-index setting, there are some good works sharpening the original reslut of Eichler and Zagier. For instance, consult
\cite{Kra86}, \cite{AB99}, \cite{RS13}, \cite{DR15} and \cite{DP17}. We wish that, their ideas could also be used to deal with lattice-index case.

\begin{appendix}
\section{SageMath code}
\label{apx:SageMath code}
We present here the SageMath source code used to obtain linear relations among Jacobi theta series of lattice index associated with the $D_4$ lattice (Theorem \ref{thm:D4linearRelation}). The reader may modify this program such that it can be used for other lattices. We have run this program successfully in the SageMath version 9.3.

First, we present some common Sage functions that are independent of specific lattices.
\begin{lstlisting}[caption = Common Sage functions \label{listing:commonSageFunctions}]
from itertools import product as cp

def poly_subs(P, values):
    '''
    Return the value of the multivariate polynomial P at values
    Input:
    * "P" -- a polynomial
    * "values" -- a sequence of objects plugged in for indeterminants of P
    The singular implementation of the subs method of the class PolynomialRing has some problem of memory leak,
    so we use this function.
    '''
    R = P.parent()
    X = R.gens()
    result = 0
    for c, M in P:
        es_M = M.exponents()[0]
        result += c * prod(v**e for v, e in zip(values, es_M))
    return result

def fac(obj):
    '''
    Clever factorial. If obj is a Sage vector (matrix), then return the product of factorials of each components (entries) of this vector (matrix).
    Input:
    * "obj" -- a vector or a matrix
    '''
    re = 1
    l = list(obj)
    if l[0] not in ZZ:
        for row in l:
            re *= prod(map(factorial, row))
    else:
        re = prod(map(factorial, l))
    return re

def mat_power(M, E):
    '''
    Return the product of powers m^e where m and e range over corresponding entries of M and E
    Input:
    * "M" and "E" -- matrices of the same dimensions
    '''
    rM = M.nrows()
    cM = M.ncols()
    return prod(M[i, j]^E[i, j] for i in srange(rM) for j in srange(cM))
\end{lstlisting}
Meanings of functions in Listing \ref{listing:commonSageFunctions} are self-evident (see their docstrings).

\begin{lstlisting}[caption = Functions dealing with $\mathcal{P}_0$ \label{listing:functionsP0}]
def closure_preceq(p):
    s = sum(p)
    it = cp(*[srange(p[i] + 1) for i in srange(len(p))])
    return [p0 for p0 in it if (sum(p0) - s) % 2 == 0]

def closure_preceq_set(P0):
    re = set()
    for p in P0:
        re |= set(closure_preceq(p))
    return re
\end{lstlisting}
The parameter \lstinline{p} of the function \lstinline{closure_preceq} represents a vector $\mathbf{p} \in \numgeq{Z}{0}^\dimn$, and a call of \lstinline{closure_preceq(p)} returns a Python list representing $\widehat{\mathbf{p}}$. The parameter \lstinline{P0} of the function \lstinline{closure_preceq_set} represents a finite subset $\mathcal{P}_0$ of $\numgeq{Z}{0}^\dimn$, and a call of \lstinline{closure_preceq_set(P0)} returns a Python set representing $\widehat{\mathcal{P}_0}$.

\begin{lstlisting}[caption = Some Python generator functions \label{listing:PyGen}]
def vectors_leq_given_sum(d, s, init_list=[]):
    pre_sum = sum(init_list)
    
    if len(init_list) == d:
        yield init_list
        return
    
    for i in srange(0, s - pre_sum + 1):
        new_list = init_list + [i]
        yield from vectors_leq_given_sum(d, s, new_list)


def integer_gen():
    yield 0
    j = 1
    while True:
        yield j
        yield -j
        j += 1


#To make this function correct, components of alpha must >=0 and <1
def vectors_square_leq_given_sum(d, s, alpha=[], init_list=[]):
    pre_sum = sum((e+j)**2 for e, j in zip(init_list, alpha))
    l = len(init_list)
        
    if pre_sum > s:
        return
    
    if l == d:
        if pre_sum <= s:
            yield init_list
        return
    
    for i in integer_gen():
        if pre_sum + (i+alpha[l])^2 > s and i < 0:
            return
        new_list = init_list + [i]
        yield from vectors_square_leq_given_sum(d, s, alpha, new_list)


#Important: To make this valid, vs must be sorted by an increasing order w.r.t. Q(alpha+v) for v in vs
def list_vectors_quad_leq_given_sum(N, vs, alpha, gram_m, s, init_list=[]):
    quad_value_fun = lambda v: (vector(QQ, v) + vector(QQ, alpha)) * gram_m * (vector(QQ, v) + vector(QQ, alpha)) / 2

    pre_sum = sum(quad_value_fun(e) for e in init_list)
    
    if len(init_list) == N:
        if pre_sum <= s:
            yield init_list
        return
    
    for v in vs:
        if pre_sum + quad_value_fun(v) > s:
            return
        new_list = init_list + [v]
        yield from list_vectors_quad_leq_given_sum(N, vs, alpha, gram_m, s, new_list)
\end{lstlisting}
Listing \ref{listing:PyGen} contains four Python generator functions. We illustrate them by examples. A call of \lstinline{vectors_leq_given_sum(4, 10)} returns a Python generator, which can be used in a \lstinline{for} loop to produce all vectors (represented by Python lists of Integer objects) $\mathbf{p}$ in $\numgeq{Z}{0}^4$ with $s(\mathbf{p}) \leq 10$. A call of \lstinline{integer_gen()} returns a Python generator, which can be used in a \lstinline{for} loop to produce all integers, in the order $0,\,1,\,-1,\,2,\,-2,\dots$. A call of \lstinline{vectors_square_leq_given_sum(4, 3, alpha=[0,0,1/2,1/2])} returns a generator, which in a loop generates all vectors $v \in \numZ^4$ with $v_1^2+v_2^2+(v_3+1/2)^2+(v_4+1/2)^2 \leq 3$. A call of \lstinline{list_vectors_quad_leq_given_sum(2, vs, [0,0,1/2,1/2], IntegralLattice('D4').gram_matrix(), 3)} returns all pairs of vectors $(v_1, v_2) \in \numZ^4\times\numZ^4$ such that $Q(v_1+\alpha)+Q(v_2+\alpha) \leq 3$ with $\alpha=(0,0,1/2,1/2)$ and $Q$ being the quadratic form associated with the $D_4$ lattice.

\begin{lstlisting}[caption = More common functions \label{listing:moreCommonFun}]
def find_least_vectors(gram, offset, c):
    d = gram.dimensions()[0]
    Q = lambda x: vector(x) * gram * vector(x) / 2
    
    #to use vectors_square_leq_given_sum, we must first ensure components of alpha are >= 0 and < 1
    offset = [(n.numerator() % n.denominator())/ n.denominator() for n in offset]
    
    result = {}
    cQ = c * Q(offset)
    for v in vectors_square_leq_given_sum(d, cQ, offset):
        v_p_offset = vector(v) + vector(offset)
        Qvalue = Q(v_p_offset)
        result.setdefault(Qvalue, []).append(v_p_offset)
        
    m = min(result.keys())
    return result[m]


def cal_P(k, vec_p, mat_M, n, t):
    sp = sum(vec_p)
    dimn = len(vec_p)
    co_sp = 1 - sp % 2
    lamb = floor(sp/2)
    lambp = lamb - co_sp
    
    #Formal variables
    Ts = ['T' + str(n) for n in srange(1, dimn+1)]
    FPR = PolynomialRing(QQ, Ts)
    Ts_var = FPR.gens()
    quad_Ts = sum(mat_M[i, j] * Ts_var[i] * Ts_var[j] for i in srange(dimn) for j in srange(dimn))
    
    re = 0
    
    for mu in closure_preceq(vec_p):
        smu = sum(mu)
        p_minus_mu = vector(vec_p) - vector(mu)
        c1 = gamma(k+(sp+smu)/2-1) / gamma(k+lambp)
        
        c2 = (quad_Ts^((sp-smu)/2))[p_minus_mu] / factorial((sp-smu)/2)
        
        re += c1 * c2 * (1/fac(mu)) * (-1/2 * n)^((sp-smu)/2) *\
                mat_power(matrix(vector(t) * mat_M), matrix(mu))
    
    re *= 2^sp
    return re
\end{lstlisting}
Listing \ref{listing:moreCommonFun} contains two Sage functions. The former is \lstinline{find_least_vectors}, of which we explain its parameters. The parameter \lstinline{gram} represents a Gram matrix $G$. The dimension $\dimn$ of the lattice may be inferred from this matrix. The second parameter \lstinline{offset} represents a vector $\alpha$ in $\numQ^\dimn$. A call of \lstinline{find_least_vectors(gram, offset, c)} returns the set
\begin{equation*}
\{v \in \alpha + \numZ^\dimn \colon Q(v) \leq Q(w) \text{ for any } w \in \alpha + \numZ^\dimn\},
\end{equation*}
where $Q(v)=\frac{1}{2}v\cdot G \cdot v^T$. This returned value is represented by a Python list. One should pass a positive real number to the parameter \lstinline{c}, and ensure the following inequality:
\begin{equation*}
v_1^2+v_2^2+\dots v_\dimn^2 \leq c \cdot Q(v),\quad v \in \numR^\dimn.
\end{equation*}
The best choice is $c=2/\uplambda_\dimn$, where $\uplambda_\dimn$ is the least eigenvalue of the Gram matrix $G$. We give such values of $c$ for some lattices in Table \ref{table:cValues}.
\begin{table}[ht]
\centering
\caption{For each lattice $\underline{L}=(\numZ^\dimn, B)$, $c$ is a number satisfying $\sum_{i=1}^\dimn v_i^2 \leq c\cdot Q(v)$ for any $v \in \numR^\dimn$, where $Q(v)=\frac{1}{2}B(v,v)$ \label{table:cValues}}
\begin{tabular}{llllll}
\toprule
Lattice & $c$ & Ref. & Lattice & $c$ & Ref. \\
\midrule
$D_4$ & $7.464102$ & \S \ref{subsec:The root lattice D4} & $A_2$ & $2$ & \S \ref{subsec:The root lattice A2} \\
$A_3$ & $3.414214$ & \S \ref{subsec:The root lattice A3} & $\tbtmat{2}{1}{1}{8}$ & $1.088304$ & \S \ref{subsec:The principal binary form of discriminant -15} \\
\bottomrule
\end{tabular}
\end{table}
In our program, the value $c$ is used in such a way that when we need the set $\{v \in \alpha+\numZ^\dimn \colon Q(v) \leq n\}$, we first generate the set $\{v \in \alpha+\numZ^\dimn \colon \sum_{i=1}^\dimn v_i^2 \leq c\cdot n\}$, which can be done by the function \lstinline{vectors_square_leq_given_sum}, and then obtain the former set as a subset of the latter one. This concludes the discussion on the former function in Listing \ref{listing:moreCommonFun}. The latter function is \lstinline{cal_P}, which calculates the polynomial $P_{k,\mathbf{p},M}$ in Definition \ref{deff:PkpM}, using the more efficient formula discussed in Remark \ref{rema:moreFormulaPkpM}. The meaning of each parameter of this function is self-evident by its name.

The following listing contains some attributes (stored in Python global variables) specific to the lattice (here $D_4$). The reader may modify these attributes to deal with other lattices.
\begin{lstlisting}[caption = Attributes of the $D_4$ lattice \label{listing:attrD4}]
#Information on the root lattice D4
dimL = 4  #dimension
ZM = FreeModule(ZZ, dimL)
GrL = IntegralLattice('D4').gram_matrix()
GrL_2 = 2 * GrL
dL = det(GrL)  #determinant
leL = 2  #level
Q = lambda x: vector(x) * GrL * vector(x) / 2
B = lambda x, y: vector(x) * GrL * vector(y)
alpha_reps = [(0,0,0,0),
              (0,0,1/2,1/2),
              (1/2,0,0,1/2),
              (1/2,0,1/2,0)]
beta_reps =  [(0,0,0,0),
              (1/2,1/2,1/2,1/2),
              (1/2,0,0,0),
              (0,1/2,0,0)]
reps_lattice_2 = [vector(t)/2 + vector([v1, v2, v3, v4])/2 for t in alpha_reps for v1, v2, v3, v4 in cp(srange(2), repeat=4)]
S_lattice_2 = None  #should be obtained by calling cal_S_lattice_2()
P0 = None  #should be obtained by calling find_P0()
n_max = None  #should be calculated by calling find_P0()
n_max_p = lambda vec_p: floor(1 + sum(vec_p) / 4)
c_L = 7.464102
\end{lstlisting}
The global variables \lstinline{dimL}, \lstinline{ZM}, \lstinline{GrL}, \lstinline{dL}, \lstinline{leL}, \lstinline{Q}, \lstinline{B} are self-evident. The variable \lstinline{GrL_2} stores the Gram matrix of $\underline{D_4}_2$. The variables \lstinline{alpha_reps} and \lstinline{beta_reps} store vectors in \eqref{eq:repD4alpha} and \eqref{eq:repD4beta} respectively, while \lstinline{reps_lattice_2} stores a list of representatives of $\underline{D_4}_2^\sharp/\underline{D_4}_2$. The variable \lstinline{S_lattice_2} is designed to store the list of $S_i,\,i=2,3,\dots,d$ in Proposition \ref{prop:whenFLP0Nonzero} with respect to the lattice $\underline{D_4}_2$. One shall assign the correct value to it by calling \lstinline{cal_S_lattice_2()}. The variable \lstinline{P0} is designed to store the set $\mathcal{P}_0$ in Theorem \ref{thm:algorithmFindLinearRelations}, and should be assigned a correct value to by calling \lstinline{find_P0()}. A call \lstinline{n_max_p(vec_p)} returns the last number in \eqref{eq:thmAlgorithmFindLinearRelationsnRange}, where the parameter \lstinline{vec_p} represents $\mathbf{p}$, and the variable \lstinline{n_max} should be the maximal value of \lstinline{n_max_p(vec_p)} when $\mathbf{p}$ ranges over the set $\mathcal{P}_0$. The value stored in \lstinline{c_L} has been explained in Table \ref{table:cValues}.

Now we present the main code realizing the algorithm described in Theorem \ref{thm:algorithmFindLinearRelations} regarding the $D_4$ lattice.
\begin{lstlisting}[caption = Main code \label{listing:mainCode}]
coR.<ze> = CyclotomicField(leL)  #Coefficient field of resulting vectors
def coeff_beta_one(beta, v):
    val = B(beta, v)
    remainder = (leL*val) % leL
    return ze^remainder


def coeff_beta(beta, vs):
    return coeff_beta_one(beta, sum([vector(v) for v in vs]))


def cal_S_lattice_2():
    global S_lattice_2
    S_lattice_2 = []
    print('Now calculating S_lattice_2...')
    for i, offset in enumerate(reps_lattice_2):
        print(f'    Now proceed {i, offset}')
        if vector(offset) == ZM(0):
            continue
        rs = find_least_vectors(GrL_2, offset, c_L / 2)
        S_lattice_2.append(rs)
        
def find_P0(max_sump=10):
    global reps_lattice_2, GrL_2, S_lattice_2, P0, n_max
    e = list(matrix.identity(dimL))
    
    P0_result = []
    mat = []
    
    print('\nNow begin to find P0...')
    l = sorted(list(vectors_leq_given_sum(dimL, max_sump)), key=lambda x: (sum(list(x)), x))
    l.pop(0)
    for vec_p in l:
        print(f'    Now proceed {vec_p}', end=' ')
        row = []
        for rs in S_lattice_2:
            entry = 0
            for v in rs:
                entry += prod(B(v, e[i])^vec_p[i] for i in srange(dimL))
            row.append(entry)
        mat_temp = mat + [row]
        ra = matrix(mat_temp).rank()
        if len(mat) == 0 or matrix(mat).rank() < ra:
            P0_result.append(vec_p)
            mat = mat_temp
            print('in P0')
            print('    Present P0 dimensions: ', len(P0_result), matrix(mat).dimensions())
            if ra == len(reps_lattice_2) - 1:
                P0 = P0_result
                n_max = n_max_p(vec_p)
                print(f'Done! The max value of n is {n_max}')
                return P0_result
    return P0_result


def proceed_lattice_D4(P0): 
    #Indetermints used to calculate polynomial P
    R.<X0, X1, X2, X3, X4> = PolynomialRing(QQ)
    
    f=lambda x: (sum(list(x)), x)
    p_list = sorted(list(closure_preceq_set(P0)), key=f)
    
    print('\nNow calculating polynomials P_{k,p,M}...')
    P_list = []
    for vec_p in p_list:
        P = cal_P(dimL * leL / 2, vec_p, leL * GrL, X0, [X1, X2, X3, X4])
        P_list.append((vec_p, P))
        print('    ', vec_p, ': ', P)
    
    print('\nNow calculating vectors...')
    Salpha_dict = {}
    for alpha in alpha_reps:
        re_alpha = []
        for v in vectors_square_leq_given_sum(dimL, c_L * n_max, alpha):
            v1 = vector(QQ, v) + vector(QQ, alpha)
            if Q(v1) <= n_max:
                re_alpha.append(v)
        key_s = lambda v: (vector(QQ, v) + vector(QQ, alpha)) *\
            GrL * (vector(QQ, v) + vector(QQ, alpha))
        Salpha_dict[alpha] = sorted(re_alpha, key=key_s)
    for key, item in Salpha_dict.items():
        print(f'    The offset {key} corresponds {len(item)} vectors')

    print('\nNow calculating mathfrakI...')
    mathfrakI = {}
    mathfrakI_list = []
    for vec_p, P in P_list:
        for n in srange(n_max_p(vec_p) + 1):
            mathfrakI_list.append((vec_p, n))
            mathfrakI.setdefault(n, []).append((vec_p, P))
    print(f'    The set mathfrakI has cardinality {len(mathfrakI_list)}')

    print('\nNow calculating vectors big_theta_alpha_beta...')
    big_theta_alpha_beta = []
    big_theta_alpha_beta_dict = {(alpha, beta): 0 for alpha in alpha_reps
                                for beta in beta_reps}
    for alpha, beta in cp(alpha_reps, beta_reps):
        big_theta_alpha_beta_dict[(alpha, beta)] = {(tuple(vec_p), n): 0
                               for n in mathfrakI.keys() for vec_p, _ in mathfrakI[n]}
    for alpha in alpha_reps:
        print(f'    Now calculating alpha={alpha}')
        count_vs = 0
        for vs in list_vectors_quad_leq_given_sum(leL, Salpha_dict[alpha], alpha, GrL, n_max):
            count_vs += 1
            if count_vs % 100 == 0:
                print('        ', alpha, 'progress count: ', count_vs)
            vec_vs=[vector(QQ, v) + vector(QQ, alpha) for v in vs]
            sum_vec_vs_by_2 = sum(vec_vs) / 2
            n = sum(Q(vec_v) for vec_v in vec_vs)
            for beta in beta_reps:
                for vec_p, P in mathfrakI[n]:
                    big_theta_alpha_beta_dict[(alpha, beta)][(tuple(vec_p), n)] += coeff_beta(beta, vs) *\
                        poly_subs(P, [n] + list(sum_vec_vs_by_2))
    
    for alpha, beta in cp(alpha_reps, beta_reps):
        seq_alpha_beta = []
        for vec_p, n in mathfrakI_list:
            seq_alpha_beta.append(big_theta_alpha_beta_dict[(alpha, beta)][(tuple(vec_p), n)])
        big_theta_alpha_beta.append(seq_alpha_beta)
    
    return big_theta_alpha_beta


def show_result(l):
    mat_D4 = matrix(l)
    print('\nMatrix dimensions: ', mat_D4.dimensions())
    print('Matrix rank: ', mat_D4.rank())
    ranks = [0]
    for j in srange(16):
        M_jrows = mat_D4[0:j+1]
        ranks.append(M_jrows.rank())
        if ranks[j+1] == ranks[j]:
            print(f'The {divmod(j, 4)}-th function:')
            sol = mat_D4[0:j].solve_left(mat_D4[j])
            print(matrix(4, 4, list(sol) + [-1] + [0] * (15 - len(sol))))
\end{lstlisting}

To call the main code, use the following instructions.
\begin{lstlisting}[caption = Do the computation \label{listing:doTheComp}]
cal_S_lattice_2()
find_P0()
l = proceed_lattice_D4(P0)
show_result(l)
\end{lstlisting}
When the code in Listing \ref{listing:doTheComp} is executing, the process of the computation and some intermediate results will be printed. It may take about several minutes to produce the final result, and the result is printed by \lstinline{show_result(l)} in an apparent way. To check each linear relation, one can define these nine theta series in his or her SageMath session by constructing objects in the parent \lstinline{PowerSeriesRing(LaurentPolynomialRing(QQ, 4), 'q')} and work with these nine Sage objects. We omit the code for this for the sake of space.
\end{appendix}

\bibliographystyle{amsalpha}
\bibliography{main}

\providecommand{\bysame}{\leavevmode\hbox to3em{\hrulefill}\thinspace}
\providecommand{\MR}{\relax\ifhmode\unskip\space\fi MR }
% \MRhref is called by the amsart/book/proc definition of \MR.
\providecommand{\MRhref}[2]{%
  \href{http://www.ams.org/mathscinet-getitem?mr=#1}{#2}
}
\providecommand{\href}[2]{#2}
\begin{thebibliography}{BFOR17}

\bibitem[AB99]{AB99}
T.~Arakawa and S.~B\"{o}cherer, \emph{A note on the restriction map for
  {J}acobi forms}, Abh. Math. Sem. Univ. Hamburg \textbf{69} (1999), 309--317.
  \MR{1722941}

\bibitem[Ajo15]{Ajo15}
Ali Ajouz, \emph{Hecke operators on jacobi forms of lattice index and the
  relation to elliptic modular forms}, Ph.D. thesis, 01 2015.

\bibitem[BB91]{BB91}
J.~M. Borwein and P.~B. Borwein, \emph{A cubic counterpart of {J}acobi's
  identity and the {AGM}}, Trans. Amer. Math. Soc. \textbf{323} (1991), no.~2,
  691--701. \MR{1010408}

\bibitem[BFOR17]{BFOR17}
Kathrin Bringmann, Amanda Folsom, Ken Ono, and Larry Rolen, \emph{Harmonic
  {M}aass forms and mock modular forms: theory and applications}, American
  Mathematical Society Colloquium Publications, vol.~64, American Mathematical
  Society, Providence, RI, 2017. \MR{3729259}

\bibitem[BMR14]{BMR14}
Kathrin Bringmann, Karl Mahlburg, and Robert~C. Rhoades, \emph{Taylor
  coefficients of mock-{J}acobi forms and moments of partition statistics},
  Math. Proc. Cambridge Philos. Soc. \textbf{157} (2014), no.~2, 231--251.
  \MR{3254591}

\bibitem[Boy15]{Boy15}
Hatice Boylan, \emph{Jacobi forms, finite quadratic modules and {W}eil
  representations over number fields}, Lecture Notes in Mathematics, vol. 2130,
  Springer, Cham, 2015, With a foreword by Nils-Peter Skoruppa. \MR{3309829}

\bibitem[Bri18]{Bri18}
Kathrin Bringmann, \emph{Taylor coefficients of non-holomorphic {J}acobi forms
  and applications}, Res. Math. Sci. \textbf{5} (2018), no.~1, Paper No. 15,
  16. \MR{3768006}

\bibitem[Bru02]{Bru02}
Jan~H. Bruinier, \emph{Borcherds products on {O}(2, {$l$}) and {C}hern classes
  of {H}eegner divisors}, Lecture Notes in Mathematics, vol. 1780,
  Springer-Verlag, Berlin, 2002. \MR{1903920}

\bibitem[Cip83]{Cip83}
Barry~A. Cipra, \emph{On the {N}iwa-{S}hintani theta-kernel lifting of modular
  forms}, Nagoya Math. J. \textbf{91} (1983), 49--117. \MR{716787}

\bibitem[CS17]{CoS17}
Henri Cohen and Fredrik Str\"omberg, \emph{Modular forms}, Graduate Studies in
  Mathematics, vol. 179, American Mathematical Society, Providence, RI, 2017, A
  classical approach. \MR{3675870}

\bibitem[DP17]{DP17}
Soumya Das and Ritwik Pal, \emph{Jacobi forms and differential operators: odd
  weights}, J. Number Theory \textbf{179} (2017), 113--125. \MR{3657159}

\bibitem[DR15]{DR15}
Soumya Das and B.~Ramakrishnan, \emph{Jacobi forms and differential operators},
  J. Number Theory \textbf{149} (2015), 351--367. \MR{3296015}

\bibitem[Ebe13]{Ebe13}
Wolfgang Ebeling, \emph{Lattices and codes}, third ed., Advanced Lectures in
  Mathematics, Springer Spektrum, Wiesbaden, 2013, A course partially based on
  lectures by Friedrich Hirzebruch. \MR{2977354}

\bibitem[EZ85]{EZ85}
Martin Eichler and Don Zagier, \emph{The theory of {J}acobi forms}, Progress in
  Mathematics, vol.~55, Birkh\"auser Boston, Inc., Boston, MA, 1985.
  \MR{781735}

\bibitem[Fre11]{Fre11}
Eberhard Freitag, \emph{Complex analysis. 2}, Universitext, Springer,
  Heidelberg, 2011, Riemann surfaces, several complex variables, abelian
  functions, higher modular functions. \MR{2810329}

\bibitem[Gri88]{Gri88}
V.~A. Gritsenko, \emph{Fourier-{J}acobi functions in {$n$} variables}, Zap.
  Nauchn. Sem. Leningrad. Otdel. Mat. Inst. Steklov. (LOMI) \textbf{168}
  (1988), no.~Anal. Teor. Chisel i Teor. Funktsi\u{\i}. 9, 32--44, 187--188.
  \MR{982481}

\bibitem[GSZ19]{GSZ19}
Valery Gritsenko, Nils-Peter Skoruppa, and Don Zagier, \emph{Theta blocks},
  2019, preprint, arXiv:1907.00188.

\bibitem[Ibu12]{Ibu12}
Tomoyoshi Ibukiyama, \emph{Taylor expansions of {J}acobi forms and applications
  to explicit structures of degree two}, Publ. Res. Inst. Math. Sci.
  \textbf{48} (2012), no.~3, 579--613. \MR{2973394}

\bibitem[Kra86]{Kra86}
J\"{u}rg Kramer, \emph{Jacobiformen und {T}hetareihen}, Manuscripta Math.
  \textbf{54} (1986), no.~3, 279--322. \MR{819403}

\bibitem[Kri96]{Kri96}
Aloys Krieg, \emph{Jacobi forms of several variables and the {M}aa\ss space},
  J. Number Theory \textbf{56} (1996), no.~2, 242--255. \MR{1373550}

\bibitem[Moc19]{Moc19}
Andreea Mocanu, \emph{Poincar\'{e} and {E}isenstein series for {J}acobi forms
  of lattice index}, J. Number Theory \textbf{204} (2019), 296--333.
  \MR{3991423}

\bibitem[Nik79]{Nik79}
V.~V. Nikulin, \emph{Integer symmetric bilinear forms and some of their
  geometric applications}, Izv. Akad. Nauk SSSR Ser. Mat. \textbf{43} (1979),
  no.~1, 111--177, 238. \MR{525944}

\bibitem[Niw75]{Niw75}
Shinji Niwa, \emph{Modular forms of half integral weight and the integral of
  certain theta-functions}, Nagoya Math. J. \textbf{56} (1975), 147--161.
  \MR{364106}

\bibitem[RS13]{RS13}
B.~Ramakrishnan and K.~D. Shankhadhar, \emph{On the restriction map for
  {J}acobi forms}, Abh. Math. Semin. Univ. Hambg. \textbf{83} (2013), no.~2,
  163--174. \MR{3123590}

\bibitem[Sag21]{Sage}
SageMath, \emph{The sage mathematics software system (version 9.3)}, The Sage
  Developers (2021).

\bibitem[Sch13]{Sch13}
Daniel Schultz, \emph{Cubic theta functions}, Adv. Math. \textbf{248} (2013),
  618--697. \MR{3107523}

\bibitem[Shi73]{Shi73}
Goro Shimura, \emph{On modular forms of half integral weight}, Ann. of Math.
  (2) \textbf{97} (1973), 440--481. \MR{332663}

\bibitem[Str13]{Str13}
Fredrik Str\"{o}mberg, \emph{Weil representations associated with finite
  quadratic modules}, Math. Z. \textbf{275} (2013), no.~1-2, 509--527.
  \MR{3101818}

\bibitem[SZ88]{SZ88}
Nils-Peter Skoruppa and Don Zagier, \emph{Jacobi forms and a certain space of
  modular forms}, Invent. Math. \textbf{94} (1988), no.~1, 113--146.
  \MR{958592}

\bibitem[TY05]{TY05}
Patrice Tauvel and Rupert W.~T. Yu, \emph{Lie algebras and algebraic groups},
  Springer Monographs in Mathematics, Springer-Verlag, Berlin, 2005.
  \MR{2146652}

\bibitem[vI21]{vI21}
Jan-Willem~M. van Ittersum, \emph{The bloch-okounkov theorem for congruence
  subgroups and taylor coefficients of quasi-jacobi forms}, 2021, preprint,
  arXiv:2102.12964.

\bibitem[Wil19]{Wil19}
Brandon Williams, \emph{Remarks on the theta decomposition of vector-valued
  {J}acobi forms}, J. Number Theory \textbf{197} (2019), 250--267. \MR{3906500}

\bibitem[Zie89]{Zie89}
C.~Ziegler, \emph{Jacobi forms of higher degree}, Abh. Math. Sem. Univ. Hamburg
  \textbf{59} (1989), 191--224. \MR{1049896}

\bibitem[ZZ21]{ZZ21}
Hai-Gang Zhou and Xiao-Jie Zhu, \emph{Double coset operators and eta
  quotients}, 2021, preprint, arXiv:2110.06768.

\end{thebibliography}

\end{document}